\documentclass[11pt]{amsart}
\usepackage[dvips]{color}
\usepackage{amsmath}
\usepackage{amsxtra}
\usepackage{amscd}
\usepackage{amsthm}
\usepackage{amsfonts}
\usepackage{amssymb}
\usepackage{eucal}
\usepackage{epsfig}
\usepackage{graphics}
\usepackage[matrix,arrow]{xy}

\theoremstyle{plain}
\newtheorem{theorem}{Theorem}[section]

\newtheorem{thm}[theorem]{Theorem}

\newtheorem{prop}[theorem]{Proposition}
\newtheorem{lem}[theorem]{Lemma}
\newtheorem{conj}[theorem]{Conjecture}
\newtheorem{defi}[theorem]{Definition}
\newtheorem{example}[theorem]{Example}

\newcommand{\Glie}{\mathfrak{g}}
\newcommand{\Yim}{\mathcal{Y}}

\newcommand{\ZZ}{\mathbb{Z}}
\newcommand{\CC}{\mathbb{C}}

\newcommand{\QQ}{\mathbb{Q}}

\newcommand{\C}{\mathbb{C}}
\newcommand{\Q}{\mathbb{Q}}
\newcommand{\Z}{\mathbb{Z}}

\newcommand{\g}{\mathfrak{g}}
\newcommand{\bo}{\mathfrak{b}}
\newcommand{\tb}{\mathbf{\mathfrak{t}}}
\newcommand{\ga}{\overline{\alpha}}

\newcommand{\Psib}{\mbox{\boldmath$\Psi$}}
\newcommand{\Psibs}{\scalebox{.7}{\boldmath$\Psi$}}
\newcommand{\qbin}[2]{{\left[
\begin{matrix}{\,\displaystyle #1\,}\\
{\,\displaystyle #2\,}\end{matrix}
\right]
}}

\newcommand{\nc}{\newcommand}
\nc{\on}{\operatorname}
\nc{\la}{\lambda}
\nc{\wh}{\widehat}
\nc{\wt}{\widetilde}
\nc{\sw}{{\mathfrak s}{\mathfrak l}}
\nc{\ghat}{\wh{\g}}
\nc{\hhat}{\wh{\h}}
\nc{\mc}{\mathcal}
\nc{\bi}{\bibitem}
\nc{\pa}{\partial}
\nc{\ppart}{(\!(t)\!)}
\nc{\pparl}{(\!(\la)\!)}
\nc{\zpart}{(\!(z^{-1})\!)}
\nc{\n}{{\mathfrak n}}
\nc{\ol}{\overline}
\nc{\mb}{\mathbf}
\nc{\bb}{{\mathfrak b}}
\nc{\su}{\wh\sw_2}
\nc{\h}{{\mathfrak h}}
\nc{\can}{\on{can}}
\nc{\ntil}{\wt{\n}}
\nc{\pone}{{\mathbb P}^1}
\nc{\bs}{\backslash}
\nc{\al}{\alpha}
\nc{\gt}{{\mathfrak g}'}
\nc{\ds}{\displaystyle}

\theoremstyle{definition}
\newtheorem{rem}{Remark}[section]

\begin{document}

\begin{title}
{Spectra of quantum KdV Hamiltonians, Langlands
  duality, and affine opers}
\end{title}

\author{Edward Frenkel}

\address{Department of Mathematics, University of California,
  Berkeley, CA 94720, USA}

\author[David Hernandez]{David Hernandez}

\address{Sorbonne Paris Cit\'e, Univ Paris Diderot, CNRS Institut de
  Math\'ematiques de Jussieu-Paris Rive Gauche UMR 7586, B\^atiment
  Sophie Germain, Case 7012, 75205 Paris Cedex 13, France.}

\begin{abstract}
  We prove a system of relations in the Grothendieck ring of the
  category ${\mc O}$ of representations of the Borel subalgebra of an
  untwisted quantum affine algebra $U_q(\ghat)$ introduced in
  \cite{HJ}. This system was discovered, under the name
  $Q\wt{Q}$-system, in \cite{MRV1,MRV2}, where it was shown that
  solutions of this system can be attached to certain $^L\ghat$-affine
  opers, introduced in \cite{FF:sol}, where $^L\ghat$ is the Langlands
  dual affine Kac--Moody algebra of $\ghat$. Together with the results
  of \cite{BLZ3,BHK} which enable one to associate quantum $\ghat$-KdV
  Hamiltonians to representations from the category ${\mc O}$, this
  provides strong evidence for the conjecture of \cite{FF:sol} linking
  the spectra of quantum $\ghat$-KdV Hamiltonians and $^L\ghat$-affine
  opers. As a bonus, we obtain a direct and uniform proof of the Bethe
  Ansatz equations for a large class of quantum integrable models
  associated to arbitrary untwisted quantum affine algebras, under a
  mild genericity condition. We also conjecture analogues of these
  results for the twisted quantum affine algebras and elucidate the
  notion of opers for twisted affine algebras, making a connection to
  twisted opers introduced in \cite{FG}.
\end{abstract}

\maketitle

\setcounter{tocdepth}{1}
\tableofcontents

\setcounter{tocdepth}{1}

\section{Introduction}

The purpose of this paper is two-fold. First, it is to elucidate the
link, proposed in \cite{FF:sol}, between (i) the spectra of the {\em
  quantum $\ghat$-KdV Hamiltonians} acting on highest weight
representations of ${\mc W}$-algebras, and (ii) {\em affine opers} --
differential operators in one variable associated to $^L\ghat$, the
affine Kac--Moody algebra that is {\em Langlands dual} to $\ghat$. And
second, it is to prove a system of relations in the Grothendieck ring
$K_0({\mc O})$ of the category ${\mc O}$ of representations of the
Borel subalgebra of the quantum affine algebra $U_q(\ghat)$,
introduced in \cite{HJ}. These relations generalize the ``quantum
Wronskian relation'' in the case of $\ghat=\wh{\sw}_2$ described in
\cite{BLZ3}. As far as we know, for general affine algebras $\ghat$
these relations, viewed as relations in $K_0({\mc O})$, are new. Among
other things, these relations enable one to quickly derive the Bethe
Ansatz equations in a uniform fashion and under minimal assumptions.

Remarkably, these relations were discovered on the other side of the
above KdV-oper correspondence -- namely, in the domain of affine opers
-- in the striking recent papers \cite{MRV1, MRV2} by Masoero,
Raimondo, and Valeri, which provided a catalyst for the present
work. Following \cite{MRV1, MRV2}, we call these relations the
$Q\wt{Q}$-{\em system}.\footnote{This $Q\wt{Q}$-system should not be
  confused with the $Q$-system satisfied by (ordinary) characters of
  Kirillov--Reshetikhin modules, see \cite{HCrelle} and references
  therein.} According to \cite{MRV1,MRV2}, solutions of the
$Q\wt{Q}$-system can be attached to affine $^L\ghat$-opers of special
kind -- precisely the simplest $^L\ghat$-affine opers proposed in
\cite{FF:sol} to describe the spectra of the quantum $\ghat$-KdV
Hamiltonians. Now let's look on the KdV side: following the
construction of Bazhanov, e.a. \cite{BLZ2,BLZ3,BHK}, one can attach
non-local quantum $\ghat$-KdV Hamiltonians to elements of $K_0({\mc
  O})$. Therefore, our results (in the present paper) imply that the
spectra of these Hamiltonians should also yield solutions of this
$Q\wt{Q}$-system. Thus, we obtain strong evidence for the conjecture
of \cite{FF:sol} linking the spectra of these Hamiltonians to
$^L\ghat$-opers:

\medskip

$$
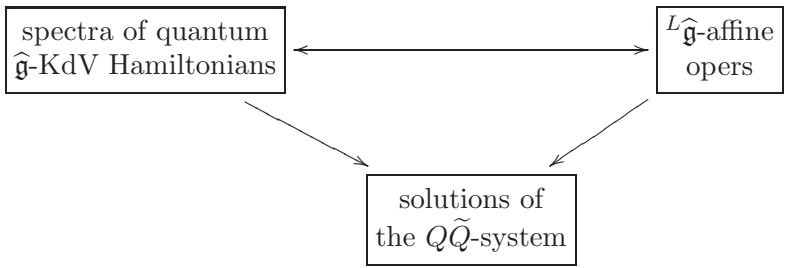
\xymatrix{\boxed{\begin{matrix} \text{spectra of quantum} \\
      \text{$\ghat$-KdV
      Hamiltonians} \end{matrix}} \ar[rd] \ar[rr] & &
\boxed{\begin{matrix} \text{$^L\ghat$-affine} \\
    \text{opers} \end{matrix}} \ar[ll]
 \ar[ld] \\
& \boxed{\begin{matrix}\text{solutions of} \\ \text{the
      $Q\wt{Q}$-system} \end{matrix}} &
}
$$

\bigskip

We now explain the interrelations between these topics in more detail.

In the case of $\ghat = \wh{\sw}_2$ the link between the spectra of
the quantum $\ghat$-KdV Hamiltonians and ordinary differential
operators (the precursors of affine opers of \cite{FF:sol}) was
discovered and investigated in \cite{DT,BLZ4,BLZ}. It was further
generalized and studied in the case of $\ghat = \wh{\sw}_3$ in
\cite{BHK}, and $\ghat = \wh{\sw}_r$ in \cite{DDT,Dorey1}.

Motivated by those works, Feigin and one of the authors of the present
paper interpreted in \cite{FF:sol} the quantum $\ghat$-KdV integrable
system as a generalization of the Gaudin model, in which a simple Lie
algebra $\g$ is replaced by the affine Kac--Moody algebra $\ghat$. It
is known \cite{FFR,F:icmp,FFT} that the spectra of the Hamiltonians of
the Gaudin model associated to $\g$ can be encoded by differential
operators known as $^L\g$-opers, where $^L\g$ is the Langlands dual
Lie algebra of $\g$. (This follows from an isomorphism between the
center of the completed enveloping algebra of $\ghat$ at the critical
level and the algebra of functions on $^L\g$-opers on the formal disc
\cite{FF:gd,F:book}.)  Therefore, by analogy with the simple Lie
algebra case, it was conjectured in \cite{FF:sol} that the spectra of
the quantum $\ghat$-KdV Hamiltonians should be encoded by what was
called in \cite{FF:sol} the affine opers associated to the Langlands
dual {\em affine} algebra $^L\g$, or $^L\ghat$-affine opers
($^L\ghat$-opers for short). Moreover, it was shown in
\cite{FF:sol} that this proposal is consistent with the results
obtained in the papers cited above for $\ghat = \wh{\sw}_r =
{}^L\ghat$. In the present paper we give more details on the structure
of the relevant $^L\ghat$-opers in the case that $^L\ghat$ is a
twisted affine algebra, making a connection to the twisted opers
introduced in \cite{FG} (Section \ref{nslcase}).

Already in the pioneering works \cite{DT,BLZ4,DT1,DDT,BHK} various
systems of functional equations were constructed for $\ghat =
\wh{\sw}_r$ with the property that its solution could be attached to
an $r$th order differential operator of a special kind. From the
perspective of \cite{FF:sol}, those differential operators are the
same as $\wh{\sw}_r$-affine opers corresponding to the highest weight
vectors for $r>2$, and to all eigenvectors for $r=2$. On the other
hand, the eigenvalues of certain non-local quantum $\wh{\sw}_r$-KdV
Hamiltonians were shown to satisfy the same relations; namely, for
$\wh{\sw}_2$, in \cite{BLZ4,BLZ}; for $\wh{\sw}_3$, in \cite{BHK}; and
for $\wh{\sw}_r$, in \cite{Ko}. Therefore, in the case of $\ghat =
\wh{\sw}_r$ these relations provided a link between the spectra of
quantum $\ghat$-KdV Hamiltonians and $^L\ghat$-affine opers (because
$^L\wh{\sw}_r = \wh{\sw}_r$).

Unfortunately, these systems were either analogues of the Baxter's
$TQ$-relation -- and so they involved the classes of finite-dimensional
representations of $U_q(\ghat)$, with the number of terms in the
relation growing as the dimensions of those representations -- or
Wronskian-type relations with the number of terms growing as
$r!$. Analogues of such systems were unknown for a general affine
algebra $\ghat$, and this impeded further progress in understanding
the link between the spectra of quantum $\ghat$-KdV Hamiltonians and
$^L\ghat$-affine opers, beyond the case of $\widehat{\sw}_r$.

That's why an elegant and uniform $Q\wt{Q}$-system proposed for an
arbitrary untwisted affine Kac--Moody algebra $\ghat$ in the papers
\cite{MRV1,MRV2} is such an important development. However, the
$Q\wt{Q}$-system was constructed in \cite{MRV1,MRV2} only on the
affine oper side of the KdV-oper correspondence. In fact, it was shown
in \cite{MRV1,MRV2} that solutions of the $Q\wt{Q}$-system can be
attached to the simplest $^L\ghat$-affine opers of the kind proposed
in \cite{FF:sol} (those are in fact the $^L\ghat$-affine opers that
are supposed to encode the eigenvalues of the quantum $\ghat$-KdV
Hamiltonians on a highest weight vector of a representation of the
${\mc W}$-algebra). We note that some partial results in this
direction were obtained earlier in \cite{S}.

If the solutions of this $Q\wt{Q}$-system could also be constructed
using the methods of \cite{MRV1,MRV2} for more general
$^L\ghat$-affine opers from \cite{FF:sol} and the present paper (which
are supposed to encode other eigenvalues of the $\ghat$-KdV
Hamiltonians on representations of ${\mc W}$-algebras), this would
open the possibility of using this $Q\wt{Q}$-system to establish the
link between the spectra of the quantum $\ghat$-KdV Hamiltonians and
$^L\ghat$-affine opers proposed in \cite{FF:sol} (analogously to the
link established in \cite{BLZ} in the case of $\wh\sw_2$).

\vspace*{4mm}

However, in order to do that, we first need to understand the meaning
of the $Q\wt{Q}$-system on the side of the quantum KdV
Hamiltonians. In other words, we need to answer the following
question: Can a solution of the $Q\wt{Q}$-system from \cite{MRV1,MRV2}
be attached to each joint eigenvector of the quantum $\ghat$-KdV
Hamiltonians?

In this paper we show that in fact the $Q\wt{Q}$-system of
\cite{MRV1,MRV2} is a {\em universal system} of relations in the
(commutative) Grothendieck ring $K_0({\mc O})$ of the category ${\mc
  O}$ of representations of the Borel subalgebra of the quantum affine
algebra $U_q(\ghat)$ introduced by Jimbo and one of the authors in
\cite{HJ}. Since one can attach non-local quantum $\ghat$-KdV
Hamiltonians to elements of $K_0({\mc O})$ using the construction of
\cite{BLZ2,BLZ3,BHK} (modulo some convergence issues discussed in
Section \ref{nlham} below), this gives us a way to attach a solution
of the $Q\wt{Q}$-system to each joint eigenvector of the quantum
$\ghat$-KdV Hamiltonians in a representation of the corresponding
${\mc W}$-algebra.

An interesting aspect of the $Q\wt{Q}$-system is that it involves two
sets of variables, denoted in \cite{MRV1,MRV2} by $Q_i$ and
$\wt{Q}_i$, where $i$ runs over the set of simple roots of $\g$. The
challenge is then to find the corresponding two sets of
representations of $U_q(\ghat)$ from the category ${\mc O}$ whose
classes satisfy the $Q\wt{Q}$-system.

We find these representations in the present paper. Namely, the first
set of representations, corresponding to the $Q_i$, are the
representations denoted by $L^+_{i,a}$ and called prefundamental in
\cite{FH}. They were first constructed for $\ghat = \wh{sl}_2$ in
\cite{BLZ2,BLZ3}, for $\ghat = \wh{sl}_3$ in \cite{BHK}, and for
$\ghat = \wh{sl}_{n+1}$ with $i = 1$ in \cite{Ko}. For general
$\ghat$, these representations were introduced and studied in
\cite{HJ}, and they were further investigated in \cite{FH}.

The representations of the second set, corresponding to the
$\wt{Q}_i$, have not been previously studied for general affine
algebras, as far as we know. We denote these representations by
$X_{i,a}$ in Section \ref{QwtQ}. If $L^+_{i,a}$ in some sense
corresponds to the $i$th fundamental weight $\omega_i$ of $\g$, then
$X_{i,a}$ corresponds to the weight $\omega_i-\al_i$. We prove that
together, $Q_i=[L^+_{i,a}]$ and $\wt{Q}_i=[X_{i,a}]$ satisfy the
$Q\wt{Q}$-system of \cite{MRV1,MRV2}, up to some scalar multiples
(which are inessential normalization constants from the point of view
of the eigenvalues of the quantum KdV Hamiltonians).

Thus, we prove that the $Q\wt{Q}$-system of \cite{MRV1,MRV2} appears
naturally in representation theory of the Borel subalgebra of the
quantum affine algebra $U_q(\ghat)$ for an arbitrary untwisted affine
Kac--Moody algebra $\ghat$. Furthermore, we conjecture an analogous
$Q\wt{Q}$-system in the Grothendieck ring $K_0({\mc O})$ for an
arbitrary twisted affine Kac--Moody algebra $\ghat$ (Conjecture
\ref{QQ t}), as well as an analogue of the result of \cite{MRV1,MRV2}
that solutions of this system can be attached to $^L\ghat$-opers
(Conjecture \ref{tw aff op}).

It follows from out results that the $Q\wt{Q}$-system arises whenever
there is an action of the Grothendieck ring $K_0({\mc O})$ on a vector
space, as the relation between the joint eigenvalues of the commuting
operators corresponding to the classes of the representations
$L^+_{i,a}$ and $X_{i,a}$ in $K_0({\mc O})$. For instance, let $V$ be
the tensor product a finite number of irreducible finite-dimensional
representation of $U_q(\ghat)$. The well-known transfer-matrix
construction yields an action of $K_0({\mc O})$ on $V$ (see
\cite{FH}). Therefore we obtain that the joint eigenvalues of the
transfer-matrices corresponding to $L^+_{i,a}$ and $X_{i,a}$ satisfy
the $Q\wt{Q}$-system. There is also a similar relation in the dual
category ${\mc O}^*$, in which the role of $L^+_{i,a}$ is played by
its dual representation denoted by $R^+_{i,a}$ in
\cite{FH}. Furthermore, in \cite{FH} it was proved that every
eigenvalue of the transfer-matrix of $R^+_{i,a}$ on such a $V$ is
equal, up to a common scalar factor that depends only on $V$ and $i$,
to a polynomial in the spectral parameter (this was originally
conjectured in \cite{Fre}). These polynomials are the generalizations
of the celebrated Baxter's polynomials (see \cite{FH} for more details
and references).

\vspace*{4mm}

As an immediate consequence of this $Q\wt{Q}$-system, one can derive a
system of equations on the roots of these polynomials, which are
nothing but the {\em Bethe Ansatz} equations for the quantum
integrable systems associated to $U_q(\ghat)$. Previously, for a
general $\ghat$ these equations were essentially guessed from the
relations between these eigenvalues and the eigenvalues of the
transfer-matrices corresponding to finite-dimensional representations
of $U_q(\ghat)$, as explained in \cite{Fre,FH}; these relations are
generalizations of Baxter's $TQ$-relation (see, e.g., \cite{FH}). This
argument, which goes back to Reshetikhin's analytic Bethe Ansatz
method \cite{R1,R2,R3} (see also \cite{BR,KS}), gives strong evidence
for the Bethe Ansatz equations, but short of a proof. On the other
hand, the $Q\wt{Q}$-system gives us a direct proof (albeit under a
genericity assumption) of the Bethe Ansatz equations for the roots of
the generalized Baxter polynomials arising from every joint
eigenvector of the transfer-matrices on $V$. This is explained in
Section \ref{bethe} below, following \cite{MRV1,MRV2}.

The non-local quantum $\ghat$-KdV Hamiltonians give us another example
of an action of the Grothendieck ring $K_0({\mc O})$ via the
construction of \cite{BLZ2,BLZ3,BHK} -- in this case, on
representations of the corresponding ${\mc W}$-algebra (modulo the
convergence issues discussed in Section \ref{nlham}). Thus, for each
common eigenvector of the quantum $\ghat$-KdV Hamiltonians we also
obtain a solution of the above $Q\wt{Q}$-system. In this case, the
$Q_i$ and $\wt{Q}_i$, viewed as functions of the spectral parameter,
are expected to be entire functions on the complex plane with a
particular asymptotic behavior at infinity
\cite{BLZ2,BLZ3,BLZ4,BLZ,BHK}. The corresponding Bethe Ansatz
equations are then the equations on the positions of the zeros of the
functions $Q_i$. In the case of $\ghat=\wh{\sw}_r$ these equations are
equivalent to the ones that have been previously considered in the
literature.

In a similar way, solutions of the $Q\wt{Q}$-system can be attached to
joint eigenvalues of the non-local quantum Hamiltonians of the ``shift
of argument'' affine Gaudin model introduced in \cite[Section
3]{FF:laws}. The zeros of the corresponding functions $Q_i$ satisfy
the same Bethe Ansatz equations, but these functions have analytic
properties different from the analytic properties of the functions
$Q_i$ corresponding to the joint eigenvalues of the non-local quantum
$\ghat$-KdV Hamiltonians.

We also want to note that Baxter's $Q$-operators $Q_i$ arise in other
important quantum integrable models, such as the ones studied in
\cite{NPS} that appear in the $\Omega$-deformations of the
five-dimensional supersymmetric quiver gauge theories. It is natural
to expect that the construction of these $Q$-operators can be extended
to the entire Grothendieck ring $K_0({\mc O})$. Then the operators
$\wt{Q}_i$ can be constructed in those models as well, so that
together with the Baxter's operators $Q_i$ they satisfy the
$Q\wt{Q}$-system. The corresponding Bethe Ansatz equations may then be
derived from the $Q\wt{Q}$-system.

\vspace*{5mm}

As we mentioned above, according to \cite{MRV1,MRV2}, a solution of
the $Q\wt{Q}$-system can be obtained from special $^L\ghat$-affine
opers; namely, the ones that correspond to highest weight vectors of
the representations of the ${\mc W}$-algebra (they are automatically
eigenvectors of the quantum $\ghat$-KdV Hamiltonians because the
Virasoro operator $L_0$ is one of these Hamiltonians). Thus, the
$Q\wt{Q}$-system links a $^L\ghat$-affine oper of this kind and the
joint eigenvalues of the $\ghat$-KdV Hamiltonians on the highest
weight vector. (This generalizes the earlier results for $\wh\sw_r$
\cite{DT,BLZ4,DT1,DDT,BHK,BLZ}).

Davide Masoero and Andrea Raimondo informed us that they expect that
their construction can be generalized to the more general
$^L\ghat$-affine opers from \cite{FF:sol} (see also Section \ref{kdv
  quantum} below) that were conjectured to correspond to other
eigenvectors of the quantum $\ghat$-KdV Hamiltonians. If this is
indeed the case, then the $Q\wt{Q}$-system will provide a link between
the $^L\ghat$-affine opers and eigenvalues of the quantum $\ghat$-KdV
Hamiltonians (similarly to the case of $\wh\sw_2$ in \cite{BLZ}).

Finally, we note that an analogue of the $Q\wt{Q}$-system also exists
for the quantum ${\mathfrak g}{\mathfrak l}_1$ toroidal algebra. We
present it in Section \ref{toroidal}. We also note that for quantum
affine algebras another system of relations in $K_0({\mc O})$ was
established in \cite{HL}. It arises naturally from a cluster algebra
structure introduced in \cite{HL}, as the first step of the
Fomin--Zelevinsky mutation relations. We call it the $QQ^*$-system and
discuss it in Section \ref{QQstar}. This system gives rise to the same
Bethe Ansatz equations as the $Q\wt{Q}$-system (see Section
\ref{bethe}). Furthermore, a system analogous to the $QQ^*$-system has
been recently defined in \cite{FJMM} for the quantum ${\mathfrak
  g}{\mathfrak l}_1$ toroidal algebra. The corresponding Bethe Ansatz
equations were also proved in \cite{FJMM} (unconditionally). We show
in Section \ref{bethe} that the same Bethe Ansatz equations also
follow from the $Q\wt{Q}$-system of Section \ref{toroidal} (under a
genericity assumption).

\vspace*{5mm}

The paper is organized as follows. In Section \ref{back}, we present
the necessary definitions and results concerning quantum affine
algebras and the category ${\mc O}$. In Section \ref{QwtQ} we
introduce the irreducible representations $X_{i,a}$ and state the
$Q\wt{Q}$-system (Theorem \ref{relation}). We describe
\footnote{Vladimir Bazhanov drew out attention to a system of
  relations in the case of $\sw_{n+1}$ which were introduced in
  \cite[Equation (1.3)]{BFLMS}. Moreover, a referee gave us an
  explicit comparison between this system and the $Q\tilde{Q}$-system,
  see Remark \ref{exbf} below for more details.} the examples of the
$Q\wt{Q}$-system for $\g=\sw_2$ and $\sw_3$, connecting the
$Q\wt{Q}$-system in these cases to relations found in earlier works
\cite{BLZ3,BHK}. We also conjecture an analogue of the
$Q\wt{Q}$-system for twisted affine algebras (Section \ref{QQ
  twisted}), state the $QQ^*$-system for quantum affine algebras
(Section \ref{QQstar}) and an analogue of the $Q\wt{Q}$-system for the
quantum ${\mathfrak g}{\mathfrak l}_1$ toroidal algebra (Section
\ref{toroidal}). In Section \ref{proof} we prove the $Q\wt{Q}$-system
for untwisted affine algebras using the theory of $q$-characters.  We
then derive the Bethe Ansatz equations from the $Q\wt{Q}$-system in
Section \ref{bethe} and discuss applications of the Bethe Ansatz
equations in various situations.

After that, we shift our focus to the KdV system. In Section
\ref{ckdv} we recall the definition of the classical KdV system and
the corresponding spaces of opers for both twisted and untwisted
affine algebras. Then we discuss the quantization of the KdV
Hamiltonians in Section \ref{qkdv}. We explain the construction of
\cite{BLZ1,BLZ2,BLZ3,BHK} assigning non-local quantum KdV Hamiltonians
to elements of $K_0({\mc O})$. In Section \ref{duality} we state
Conjecture \ref{commute} that the quantum $\ghat$- and $^L\ghat$-KdV
Hamiltonians commute with each other. If true, this should yield a
somewhat surprising correspondence between solutions of the
$Q\wt{Q}$-systems (as well as other equations stemming from $K_0({\mc
  O})$ such as the $QQ^*$-system) for $U_q(\ghat)$ and
$U_{\check{q}}({}^L\ghat)$, where $q=e^{\pi i \beta^2}$ and $\check{q}
= e^{\pi i \check{r}/\beta^2}$.

In Section \ref{kdv quantum} we discuss and give more details on the
conjecture of \cite{FF:sol} linking the spectra of quantum $\ghat$-KdV
Hamiltonians and $^L\ghat$-opers of a certain kind, elucidating a
number of points, including a more precise interpretation of the
$^L\ghat$-opers in the non-simply laced case. We also discuss the
results of \cite{MRV1,MRV2} associating a solution of the
$Q\wt{Q}$-system (for $U_q(\ghat)$) to the simplest $^L\ghat$-opers of
this kind and conjecture a generalization of these results to the case
of twisted affine algebras $\ghat$ (Section \ref{QQ tw aff}). Finally,
we discuss a conjectural duality between $\ghat$- and $^L\ghat$-affine
opers (Section \ref{duality opers}) and comment on the appearance of
opers associated to two Langlands dual Lie algebras in the classical
and the quantum pictures, which still largely remains a mystery
(Section \ref{last}).

\vspace*{4mm}

\noindent{\bf Acknowledgments.} We are grateful to Davide Masoero and
Andrea Raimondo for fruitful discussions and explanations about their
work \cite{MRV1,MRV2}, which was the main motivation for this
paper. We also thank them for raising a question about the meaning of
our formula for $^L\ghat$-opers in the non-simply laced case, which
helped us to formulate it more precisely (Section \ref{nslcase}).

We thank Bernard Leclerc for his comments on the first version of this
paper, and the referees for their thorough reading of the paper and
useful comments.

E. Frenkel was supported by the NSF grant DMS-1201335. D. Hernandez
was supported in part by the European Research Council under the
European Union's Framework Programme H2020 with ERC Grant Agreement
number 647353 QAffine.

\section{Background on quantum affine algebras}    \label{back}

In this section we collect some definitions and results on quantum
affine algebras and their representations. We refer the reader to
\cite{CP} for a canonical introduction, and to \cite{CH,L} for more
recent surveys on this topic. We also discuss representations of the
Borel subalgebra of a quantum affine algebra, see \cite{HJ, FH} for
more details.

\subsection{Quantum affine algebras and Borel algebras}

Let $C=(C_{i,j})_{0\leq i,j\leq n}$ be an indecomposable Cartan matrix
of untwisted affine type.  We denote by $\wh{\Glie}$ the Kac--Moody
Lie algebra associated with $C$.  Set $I=\{1,\ldots, n\}$, and denote
by $\Glie$ the finite-dimensional simple Lie algebra associated with
the Cartan matrix $(C_{i,j})_{i,j\in I}$.  Let $\{\alpha_i\}_{i\in
  I}$, $\{\alpha_i^\vee\}_{i\in I}$, $\{\omega_i\}_{i\in I}$,
$\{\omega_i^\vee\}_{i\in I}$, and ${\mathfrak{h}}$ be the simple
roots, the simple coroots, the fundamental weights, the fundamental
coweights, and the Cartan subalgebra of $\Glie$, respectively.  We set
$Q=\oplus_{i\in I}\Z\alpha_i$, $Q^+=\oplus_{i\in I}\Z_{\ge0}\alpha_i$,
$P=\oplus_{i\in I}\Z\omega_i$.  Let $D=\mathrm{diag}(d_0\ldots, d_n)$
be the unique diagonal matrix such that $B=DC$ is symmetric and
$d_i$'s are relatively prime positive integers.  We will also use
$P_\Q = P\otimes \Q$ with its partial ordering defined by $\omega\leq
\omega'$ if and only if $\omega'-\omega\in Q^+$.  We denote by
$(~,~):Q\times Q\to\Z$ the invariant symmetric bilinear form such that
$(\alpha_i,\alpha_i)=2d_i$.  We use the numbering of the Dynkin
diagram as in \cite{ka}.  Let $a_0,\cdots,a_n$ stand for the Kac
labels (\cite{ka}, pp.55-56). We have $a_0 = 1$ and we set $\alpha_0 =
-(a_1\alpha_1 + a_2\alpha_2 + \cdots + a_n\alpha_n)$.  We set
$$i\sim j\text{ if }C_{i,j} < 0.$$

Throughout this paper, we fix a non-zero complex number $q$ which is
not a root of unity.  We set $q_i=q^{d_i}$. We also set $h\in\CC$ such
that $q = e^h$, so that $q^r$ is well-defined for any $r\in\QQ$. We
will use the standard symbols for $q$-integers
\begin{align*}
[m]_z=\frac{z^m-z^{-m}}{z-z^{-1}}, \quad
[m]_z!=\prod_{j=1}^m[j]_z,
 \quad 
\qbin{s}{r}_z
=\frac{[s]_z!}{[r]_z![s-r]_z!}. 
\end{align*}

The quantum loop algebra $U_q(\wh{\Glie})$ is the $\C$-algebra defined
by generators
$e_i,\ f_i,\ k_i^{\pm1}$ ($0\le i\le n$) 
and the following relations for $0\le i,j\le n$.
\begin{align*}
&k_ik_j=k_jk_i,\quad k_0^{a_0}k_1^{a_1}\cdots k_n^{a_n}=1,\quad
&k_ie_jk_i^{-1}=q_i^{C_{i,j}}e_j,\quad k_if_jk_i^{-1}=q_i^{-C_{i,j}}f_j,
\\
&[e_i,f_j]
=\delta_{i,j}\frac{k_i-k_i^{-1}}{q_i-q_i^{-1}},
\\
&\sum_{r=0}^{1-C_{i.j}}(-1)^re_i^{(1-C_{i,j}-r)}e_j e_i^{(r)}=0\quad (i\neq j),
&\sum_{r=0}^{1-C_{i.j}}(-1)^rf_i^{(1-C_{i,j}-r)}f_j f_i^{(r)}=0\quad (i\neq j)\,.
\end{align*}
Here we have set $x_i^{(r)}=x_i^r/[r]_{q_i}!$ ($x_i=e_i,f_i$). 
The algebra $U_q(\wh{\Glie})$ has a Hopf algebra structure such that
\begin{align*}
&\Delta(e_i)=e_i\otimes 1+k_i\otimes e_i,\quad
\Delta(f_i)=f_i\otimes k_i^{-1}+1\otimes f_i,
\quad
\Delta(k_i)=k_i\otimes k_i\,,
\end{align*}
where $i=0,\cdots,n$. 

The algebra $U_q(\wh{\Glie})$ can also be presented in terms of the Drinfeld
generators \cite{Dri2,bec}
\begin{align*}
  x_{i,r}^{\pm}\ (i\in I, r\in\Z), \quad \phi_{i,\pm m}^\pm\ (i\in I,
  m\geq 0), \quad k_i^{\pm 1}\ (i\in I).
\end{align*}
We will use the generating series $(i\in I)$: 
$$\phi_i^\pm(z) = \sum_{m\geq 0}\phi_{i,\pm m}^\pm z^{\pm m} =
k_i^{\pm 1}\text{exp}\left(\pm (q_i - q_i^{-1})\sum_{m > 0} h_{i,\pm
    m} z^{\pm m} \right).$$
We also set $\phi_{i,\pm m}^\pm = 0$ for $m < 0$, $i\in I$.

The algebra $U_q(\wh{\Glie})$ has a $\ZZ$-grading defined by
$\on{deg}(e_i) = \on{deg}(f_i) = \on{deg}(k_i^{\pm 1}) = 0$ for $i\in
I$ and $\on{deg}(e_0) = - \on{deg}(f_0) = 1$.  It satisfies
$\on{deg}(x_{i,m}^\pm) = \on{deg}(\phi_{i,m}^\pm) = m$ for $i\in I$,
$m\in\ZZ$. For $a\in\CC^\times$, there is a corresponding automorphism
$$\tau_a : U_q(\wh{\Glie})\rightarrow U_q(\wh{\Glie})$$ 
such that an element $g$ of degree
$m\in\ZZ$ satisfies $\tau_a(g) = a^m g$.

\begin{defi} The Borel algebra $U_q(\bo)$ is 
the subalgebra of $U_q(\wh{\Glie})$ generated by $e_i$ 
and $k_i^{\pm1}$ with $0\le i\le n$. 
\end{defi}
This is a Hopf subalgebra of $U_q(\wh{\Glie})$. 
The algebra $U_q(\mathfrak{b})$ 
contains 
the Drinfeld generators
$x_{i,m}^+$, $x_{i,r}^-$, $k_i^{\pm 1}$, $\phi_{i,r}^+$ 
where $i\in I$, $m \geq 0$ and $r > 0$. When $\Glie = \sw_2$, these
elements generate $U_q(\mathfrak{b})$.

Denote $\tb \subset U_q(\bo)$ the subalgebra generated by $\{k_i^{\pm 1}\}_{i\in I}$. 
Set $\tb^\times=\bigl(\C^\times\bigr)^I$, and endow it with a group
structure by pointwise multiplication.  We define a group morphism
$\overline{\phantom{u}}:P_\Q \longrightarrow \tb^\times$ by setting
$\overline{\omega_i}(j)=q_i^{\delta_{i,j}}$. We shall use the standard
partial ordering on $\tb^\times$:
\begin{align}
\omega\leq \omega' \quad \text{if $\omega \omega'^{-1}$ 
is a product of $\{\ga_i^{-1}\}_{i\in I}$}.
\label{partial}
\end{align}

\subsection{Category $\mathcal{O}$ for representations of Borel
  algebras}    \label{catO}

For a $U_q(\mathfrak{b})$-module $V$ and $\omega\in \tb^\times$, we set
\begin{align}
V_{\omega}=\{v\in V \mid  k_i\, v = \omega(i) v\ (\forall i\in I)\}\,,
\label{wtsp}
\end{align}
and call it the weight space of weight $\omega$. 
For any $i\in I$, $r\in\ZZ$ we have $\phi_{i,r}^\pm (V_\omega)\subset V_\omega$
and $x_{i,r}^\pm (V_{\omega}) \subset V_{\omega \ga_i^{\pm 1}}$.
We say that $V$ is Cartan-diagonalizable 
if $V=\underset{\omega\in \tb^\times}{\bigoplus}V_{\omega}$.

\begin{defi} A series $\Psib=(\Psi_{i, m})_{i\in I, m\geq 0}$ 
of complex numbers such that 
$\Psi_{i,0}\neq 0$ for all $i\in I$ 
is called an $\ell$-weight. 
\end{defi}

We denote by $\tb^\times_\ell$ the set of $\ell$-weights. 
Identifying $(\Psi_{i, m})_{m\geq 0}$ with its generating series we shall
write
\begin{align*}
\Psib = (\Psi_i(z))_{i\in I},
\quad
\Psi_i(z) = \underset{m\geq 0}{\sum} \Psi_{i,m} z^m.
\end{align*}
Since each $\Psi_i(z)$ is an invertible formal power series,
$\tb^\times_\ell$ has a natural group structure.
We have a surjective morphism of groups
$\varpi : P_\ell\rightarrow P_\Q$ given by 
$\Psi_i(0) = q_i^{\varpi(\Psib)(\alpha_i^\vee)}$.
 
\begin{defi} A $U_q(\mathfrak{b})$-module $V$ is said to be 
of highest $\ell$-weight 
$\Psib\in \tb^\times_\ell$ if there is $v\in V$ such that 
$V =U_q(\mathfrak{b})v$ 
and the following hold:
\begin{align*}
e_i\, v=0\quad (i\in I)\,,
\qquad 
\phi_{i,m}^+v=\Psi_{i, m}v\quad (i\in I,\ m\ge 0)\,.
\end{align*}
\end{defi}

The $\ell$-weight $\Psib\in \tb^\times_\ell$ is uniquely determined by $V$. 
It is called the highest $\ell$-weight of $V$. 
The vector $v$ is said to be a highest $\ell$-weight vector of $V$.

\begin{prop}\label{simple} 
For any $\Psib\in \tb^\times_\ell$, there exists a simple 
highest $\ell$-weight module $L(\Psib)$ of highest $\ell$-weight
$\Psib$. This module is unique up to isomorphism.
\end{prop}

The submodule of $L(\Psib)\otimes L(\Psib')$ generated by the tensor
product of the highest $\ell$-weight vectors is of highest
$\ell$-weight $\Psib\Psib'$. In particular, $L(\Psib\Psib')$ is a
subquotient of $L(\Psib)\otimes L(\Psib')$.

\begin{defi}\cite{HJ}
For $i\in I$ and $a\in\CC^\times$, let 
\begin{align}
L_{i,a}^\pm = L(\Psib_{i,a})
\quad \text{where}\quad 
(\Psib_{i,a})_j(z) = \begin{cases}
(1 - za)^{\pm 1} & (j=i)\,,\\
1 & (j\neq i)\,.\\
\end{cases} 
\label{fund-rep}
\end{align}
\end{defi}
We call $L_{i,a}^+$ (resp. $L_{i,a}^-$) a positive (resp. negative)
prefundamental representation in the category $\mathcal{O}$.

\begin{defi}\label{oned}\cite{HJ}
For $\omega\in \tb^\times$, let 
$$[\omega] = L(\Psib_\omega)
\quad \text{where}\quad 
(\Psib_\omega)_i(z) = \omega(i) \quad (i\in I).$$
\end{defi}
Note that the representation $[\omega]$ is $1$-dimensional with a
trivial action of $e_0,\cdots, e_n$. For $\lambda\in P$, we will simply use
the notation $[\lambda]$ for the representation
$[\overline{\lambda}]$.

For $a\in\CC^\times$, the subalgebra $U_q(\mathfrak{b})$ is stable
under $\tau_a$.  Denote its restriction to $U_q(\mathfrak{b})$ by the
same letter.  Then the pullbacks of the $U_q(\mathfrak{b})$-modules
$L_{i,b}^\pm$ by $\tau_a$ is $L_{i,ab}^\pm$.

For $\lambda\in \tb^\times$, we set $D(\lambda )=
\{\omega\in \tb^\times \mid \omega\leq\lambda\}$. The following category $\mathcal{O}$ is introduced in \cite{HJ}.

\begin{defi} A $U_q(\mathfrak{b})$-module $V$ 
is said to be in category $\mathcal{O}$ if:

i) $V$ is Cartan-diagonalizable,

ii) for all $\omega\in \tb^\times$ we have 
$\dim (V_{\omega})<\infty$,

iii) there exist a finite number of elements 
$\lambda_1,\cdots,\lambda_s\in \tb^\times$ 
such that the weights of $V$ are in 
$\underset{j=1,\cdots, s}{\bigcup}D(\lambda_j)$.
\end{defi}

The category $\mathcal{O}$ is a monoidal category. 

\newcommand{\mfr}{\mathfrak{r}}
Let 
$\mfr$
be the subgroup of $\tb^\times_\ell$
consisting of $\Psib$ such that 
$\Psi_i(z)$ is rational for any $i\in I$.

\begin{thm}\label{class}\cite{HJ} Let $\Psib\in\tb^\times_\ell$. 
The simple module $L(\Psib)$ is in category 
$\mathcal{O}$ if and only if $\Psib\in \mfr$. 
\end{thm}

Let $\mathcal{E}$ 
be the additive group of maps $c : P_\Q \rightarrow \ZZ$ 
whose support 
$$\text{supp}(c) = \{\omega\in P_\Q,c(\omega) \neq 0\}$$ 
is contained in 
a finite union of sets of the form $D(\mu)$. 
For $\omega\in P_\Q$, we define $[\omega] = \delta_{\omega,.}\in\mathcal{E}$.
For $V$ in the category $\mathcal{O}$ we define the character of $V$ to be
an element of $\mathcal{E}$
\begin{align}
\chi(V) = \sum_{\omega\in\tb^\times} 
\text{dim}(V_\omega) [\omega]\,.
\label{ch}
\end{align}

As for the category $\mathcal{O}$ of a classical Kac--Moody algebra,
the multiplicity of a simple module in a module of our category
$\mathcal{O}$ is well-defined (see \cite[Section 9.6]{ka}) and we have
the corresponding Grothendieck ring $K_0(\mathcal{O})$ (see also
\cite[Section 3.2]{HL}). Its elements are the formal
sums
$$\chi = \sum_{\Psib\in \mfr} \lambda_{\Psib} [L(\Psib)]$$ 
where the $\lambda_{\Psib}\in\ZZ$ are set so that $\sum_{\Psib\in
  \mfr, \omega\in P_\Q} |\lambda_{\Psib}|
\text{dim}((L(\Psib))_\omega) [\omega]$
is in $\mathcal{E}$.

 We naturally identify $\mathcal{E}$ with 
the Grothendieck ring of the category of representations of
$\mathcal{O}$ with constant $\ell$-weights,
the simple objects of which are the $[\omega]$, $\omega\in P_\Q$. Thus
as in \cite[Section 9.7]{ka} we will regard elements of $\mathcal{E}$
as formal sums
\[
 c = \sum_{\omega\in \text{Supp}(c)} c(\omega)[\omega].
\]
The multiplication is given by $[\omega][\omega'] = [\omega+\omega']$ and $\mathcal{E}$
is regarded as a subring of $K_0(\mathcal{O})$. 

The character defines a ring morphism $\chi : K_0(\mathcal{O})\rightarrow \mathcal{E}$ which is not injective.

\subsection{Monomials and finite-dimensional representations}\label{fdrep}

Following \cite{Fre}, consider the ring of Laurent polynomials $\Yim =
\ZZ[Y_{i,a}^{\pm 1}]_{i\in I,a\in\CC^\times}$ in the indeterminates
$\{Y_{i,a}\}_{i\in I, a\in \C^\times}$.  Let $\mathcal{M}$ be the group of
monomials of $\Yim$. For example, for $i\in I, a\in\CC^\times$, define
$A_{i,a}\in\mathcal{M}$ to be
$$
Y_{i,aq_i^{-1}}Y_{i,aq_i}
\Bigl(\prod_{\{j\in I|C_{j,i} = -1\}}Y_{j,a}
\prod_{\{j\in I|C_{j,i} = -2\}}Y_{j,aq^{-1}}Y_{j,aq}
\prod_{\{j\in I|C_{j,i} =
-3\}}Y_{j,aq^{-2}}Y_{j,a}Y_{j,aq^2}\Bigr)^{-1}\,.
$$
For a monomial 
$m = \prod_{i\in I, a\in\CC^\times}Y_{i,a}^{u_{i,a}}$, 
we consider its `evaluation on $\phi^+(z)$'. 
By definition it is an element 
$m(\phi(z))\in\mfr$  
given by
$$m\bigl(\phi(z))=
\prod_{i\in I, a\in\CC^\times}
\left(Y_{i,a}(\phi(z))\right)^{u_{i,a}}\text{ where }
\Bigl(Y_{i,a}\bigl(\phi(z)\bigr)\Bigr)_j
=\begin{cases}
\displaystyle{q_i\frac{1-a q_i^{-1}z}{1-aq_iz}}& (j=i),\\
1 & (j\neq i).\\
\end{cases}$$
This defines an injective group morphism 
$\mathcal{M}\rightarrow \mfr$. 
We identify a monomial $m\in\mathcal{M}$ with its image in 
$\mfr$. 
Note that $\varpi(Y_{i,a}) = \omega_i$.

Let $\mathcal{C}$ be the category of (type $1$) finite-dimen\-sional
representations of $U_q(\wh{\Glie})$.

A monomial $M\in\mathcal{M}$ is said to be dominant if
$M\in\ZZ[Y_{i,a}]_{i\in I, a\in\CC^\times}$. 
Then $L(M)$ is finite-dimensional.  Moreover, the action of
$U_q(\mathfrak{b})$ can be uniquely extended to an action of the full
quantum affine algebra $U_q(\wh{\Glie})$, and any simple object in the
category $\mathcal{C}$ is of this form. By \cite{CP} and \cite[Remark
3.11]{FH}, for $L(\Psib)$ a finite-dimensional module in the
  category $\mathcal{O}$, there is $M$ as above and
  $\omega\in\tb^\times$ such that $$L(\Psib) \simeq L(M)\otimes [\omega].$$ 
Note that if $\Psib$ is a monomial in the variables
$$\wt{Y}_{i,a} = [-\omega_i]Y_{i,a},$$ 
then $L(\Psib)$ is finite-dimensional. We will also use in the
following notation:
$$\wt{A}_{i,a} =
\Psib_{i,aq_i^{-2}}\Psib_{i,aq_i^2}^{-1}\left(\prod_{j\sim i, r_j >
    1}\Psib_{j,aq_j^{-1}}^{-1}\Psib_{j,aq_j}\right)
\left(\prod_{j\sim i, r_j =
    1}\Psib_{j,aq_i^{-1}}^{-1}\Psib_{j,aq_i}\right) =
[-\alpha_i]A_{i,a}.$$

For $i\in I$, $a\in\CC^\times$ and $k\geq 0$, we have the
Kirillov--Reshetikhin (KR) module
\begin{align}
W_{k,a}^{(i)} = L(Y_{i,a}Y_{i,aq_i^2}\cdots Y_{i,aq_i^{2(k-1)}})\,.
\label{KRmod}
\end{align}
The representations $W_{1,a}^{(i)} = L(Y_{i,a})$ are called
fundamental representations. The simple tensor product of a KR-module
by a one-dimensional representation $[\omega]$, $\omega\in P$, will
also be called a KR-module.

\subsection{Example}\label{dualex}


The prefundamental representations have a relatively simple structure
in comparison to the general simple representations in the category
$\mathcal{O}$.

As an example, let us consider the case $\Glie=B_2$ and the
representation $L_{2,1}^+$. From \cite{HJ} we know the action of a
large number of generators of the Borel algebra on this
representation. For $r > 0$, $\phi_{1,r}^+$, $\phi_{2,r + 1}^+$,
$x_{1,r}^+, x_{2,r}^+$, $x_{1,r}^-$, $x_{2,r+1}^-$ act by $0$ on this
representation; $k_2^{-1}\phi_{2,1}^+$ is the operator
$-\text{Id}$; the root operators $E_{-\alpha_1 -\alpha_2 +
  (r+1)\delta}$, $E_{-2\alpha_1-\alpha_2 + (r+1)\delta}$ act by $0$.

As this representation $L_{2,1}^+$ is constructed in \cite{HJ} as a
limit of finite-dimensional Kirillov-Reshetikhin modules whose
structure is well-known (see \cite{HCrelle, h3} and references
therein), it has a basis
$$L_{2,1}^+ = \bigoplus_{T\in\mathcal{T}} \CC v_T$$
of weight vectors parametrized by the semi-infinite tableaux $T =
(T_{i,j})_{i = 1, 2,j\geq 0}$ with coefficients in the ordered set
$$1\prec 2\prec 0 \prec \overline{1}\prec \overline{2}$$ 
satisfying $T_{i,j}\preceq T_{i,j+1}$, $(T_{i,j},T_{i,j+1})\neq (0,0)$
and ($T_{1,j}\prec T_{2,j}$ or $(T_{1,j},T_{2,j}) = (0,0)$). It is
required in addition that $T_{i,j} = i$ for $j$ is large enough.

The weight of $v_T$ is
$$\omega_T = -\sum_{j\geq 0} \alpha(T_{1,j}) + \beta(T_{2,j})$$
where $\alpha(1,2,0,\overline{1}) = (0,\alpha_1,\alpha_1 +
\alpha_2,\alpha_1 + 2\alpha_2)$ and
$\beta(2,0,\overline{1},\overline{2}) =
(0,\alpha_2,2\alpha_2,2\alpha_2+\alpha_1)$.
In particular, we have the explicit character formula 
$$\chi_2 = \chi(L_{2,1}^+) = \sum_{T\in\mathcal{T}} [\omega_T].$$

\section{The $Q\wt{Q}$-system}    \label{QwtQ}

In this section we prove the $Q\wt{Q}$-system of relations in
$K_0({\mc O})$ (Theorem \ref{relation}). This is one of the main
results of this paper.

\subsection{Statement}

Let $\Glie$ be an arbitrary simple Lie algebra. We start by
introducing the following representations.

\begin{defi} For $i\in I$ and $a\in\CC^\times$, we define the
  representation
$$X_{i,a} = L(\wt{\Psib}_{i,a})$$
where
$$\wt{\Psib}_{i,a} = \Psib_{i,a}^{-1}
 \left(\prod_{j|C_{i,j} = -
     1}\Psib_{j,aq_i}\right)\left(\prod_{j|C_{i,j} =
     -2}\Psib_{j,a}\Psib_{j,aq^2}\right)\left(\prod_{j|C_{i,j} =
     -3}\Psib_{j,aq^{-1}}\Psib_{j,aq}\Psib_{j,aq^3}\right).$$
\end{defi}

Note that if we think of $\Psib_{i,a}$ as an analogue of the
fundamental weight $\omega_i$, then $\wt{\Psib}_{i,a}$ is an analogue
of the weight $\omega_i-\al_i$.

\begin{rem} {\rm (i) If $\g$ is simply-laced, then
$$
X_{i,a} = L(\Psib_{i,a}^{-1}\prod_{j\sim i}\Psib_{j,aq}).
$$
 
(ii) The $\ell$-weight $\Psib_{i,a}$ may be written as an infinite product 
$$\Psib_{i,a} = \wt{Y}_{i,aq_i}\wt{Y}_{i,aq_i^3}\wt{Y}_{i,aq_i^5}\cdots.$$
where $\wt{Y}_{i,a}$ as in Section \ref{fdrep} is the analogue of a
fundamental weight.
 Similarly, one may write $\wt{\Psib}_{i,a}$ as an infinite product
 involving the $\wt{A}_{i,a}$ as in Section \ref{fdrep} which is the
 analogue of a simple root. Indeed,
$$\wt{A}_{i,aq_i^2}^{-1}\wt{Y}_{i,aq_i^3}\wt{A}_{i,aq_i^4}^{-1}
\wt{Y}_{i,aq_i^5}\wt{A}_{i,aq_i^6}^{-1}\wt{Y}_{i,aq_i^7}\cdots$$
$$ = \Psib_{i,a}^{-1} \prod_{K\geq 1}\left(\prod_{j\sim i, r_j > 1}\Psib_{j,aq_i^{2K}q_j^{-1}}\Psib_{j,aq_i^{2K}q_j}^{-1}\right)
\left(\prod_{j\sim i, r_j = 1}\Psib_{j,aq_i^{2K}q_i^{-1}}\Psib_{j,aq_i^{2K}q_i}^{-1}\right)$$
$$= \Psib_{i,a}^{-1}\left(\prod_{j\sim i, r_j > 1}\Psib_{j,aq_i^2q^{-r_j}}\Psib_{j,aq_i^2q_j^{-r_j+2r_i}}\cdots \Psib_{j,aq_i^2q^{r_j-2r_i}}\right)\left(\prod_{j\sim i, r_j = 1}\Psib_{j,aq_i}\right) =\wt{\Psib}_{i,a} .$$
So $\wt{\Psib}_{i,a}$ may be indeed be viewed as an analogue of the
difference $\omega_i - \alpha_i$.
}\qed
\end{rem}

We define the $Q$ and $\wt{Q}$ variables as follows:
\begin{equation}    \label{Q and wtQ}
Q_{i,a} = [L_{i,a}^+] \qquad \text{ and } \qquad \wt{Q}_{i,a} =
[X_{i,aq_i^{-2}}]\chi_i^{-1}\left[-\frac{\alpha_i}{2}\right],
\end{equation}
where 
\begin{equation}\label{chii}\chi_i  = \chi(L_{i,a}^+)\in
  \mathcal{E}\end{equation} 
does not depend on $a$ and is seen as an element of $K_0(\mathcal{O})$. 

Now we state the $Q\wt{Q}$-system which is one of the main results of
this paper.

\begin{thm}\label{relation}
For any $i\in I$, $a\in\CC^\times$ we have the following
$Q\wt{Q}$-system:
\begin{multline}    \label{relation1}
\left[\frac{\alpha_i}{2}\right]Q_{i,aq_i^{-1}}\wt{Q}_{i,aq_i} -
\left[-\frac{\alpha_i}{2}\right] Q_{i,aq_i}\wt{Q}_{i,aq_i^{-1}}
= \\
\left(\prod_{j|C_{i,j} = - 1}Q_{j,a}\right)\left(\prod_{j|C_{i,j} =
    -2}Q_{j,aq^{-1}}Q_{j,aq}\right)\left(\prod_{j|C_{i,j} =
    -3}Q_{j,aq^{-2}}Q_{j,a}Q_{j,aq^2}\right)\ .
\end{multline}
\end{thm}

\begin{rem}    \label{rem 3 parts}

{\rm (i) This $Q\wt{Q}$-system matches \cite[Formula (5.4)]{MRV2},
    with $E$ replaced by $a$ and $\Omega$ replaced by $q^{-2}$.

(ii) In this simply-laced case, the $Q\wt{Q}$-system specializes to
the following:
$$\left[\frac{\alpha_i}{2}\right]Q_{i,aq^{-1}}\wt{Q}_{i,aq} -
\left[-\frac{\alpha_i}{2}\right] Q_{i,aq}\wt{Q}_{i,aq^{-1}} =
\prod_{j\sim i}Q_{j,a}.$$
This matches \cite[Formula (4.6)]{MRV1}.

(iii) An analogous system may be written for the category
$\mathcal{O}^*$ dual to the category $\mathcal{O}$, which was
defined in \cite[Section 3.6]{HJ}.
It suffices to set $Q_{i,a} = [R_{i,a}^+]$ and $\wt{Q}_{i,a} =
[X_{i,aq_i^{-2}}'](\chi_i')^{-1}\left[\frac{\alpha_i}{2}\right]$ 
and to replace $\left[\pm\frac{\alpha_i}{2}\right]$ by
$\left[\mp\frac{\alpha_i}{2}\right]$. Here $R_{i,a}^+$ and $X_{i,a}'$
are defined by $(R_{i,a}^+)^*\simeq L_{i,a}^+$, $(X_{i,a}')^*\simeq
X_{i,a}$ and $\chi_i' = \chi(R_{i,a}^+)$. Indeed the same proof as for
formula \eqref{relation1} gives us a relation in the category
$\overline{\mathcal{O}}$ of the opposite 
Borel as considered in \cite[Section 3.5]{FH} with the representations
$L_{i,a}^+$, $X_{i,a}$ replaced by the simple representations in
$\overline{\mathcal{O}}$ with the same highest $\ell$-weight. Then it
suffices to twist by the involutive automorphism $\wh{\omega}$ as in \cite{FH}.}
\qed\end{rem}

\subsection{First examples}\label{firstex}

Let $\Glie = \sw_2$. We have the following relation in
$K_0(\mathcal{O})$:
$$[L_a^+][L_a^-] - [-\alpha][L_{aq^2}^+][L_{aq^{-2}}^-] = \chi$$
where as above
$$\chi = \chi(L_a^+) = \sum_{r\geq 0}[-r\alpha]\in K_0(\mathcal{O}).$$
does not depend on $a$.
By setting 
$$Q_a = [L_a^+]\text{ and }\wt{Q}_a = [L_{aq^{-2}}^-]\chi^{-1}
\left[-\frac{\alpha}{2}\right]$$ 
we get the $Q\wt{Q}$-relation:
$$\left[\frac{\alpha}{2}\right]Q_{aq^{-1}}\wt{Q}_{aq} -
\left[-\frac{\alpha}{2}\right] Q_{aq}\wt{Q}_{aq^{-1}} = 1$$
which is essentially the ``quantum Wronskian relation'' \cite[Formula
(9)]{BLZ} (see also \cite{BLZ2, BLZ3}).

Let $\Glie = \sw_3$. We consider $6$ families of representations as in
\cite{BHK}: the prefundamental representations $L_{1,a}^+$,
$L_{2,a}^+$, $L_{1,a}^-$, $L_{2,a}^-$ and the new representations
$$X_{1,a} = L(\Psib_{1,a}^{-1}\Psib_{2,aq})\text{ , }X_{2,a} =
L(\Psib_{2,a}^{-1}\Psib_{1,aq}).$$
Let
$$\chi_1 = \chi(L_{1,a}^+) = \sum_{0\leq r\leq s}[-r\alpha_2 -
s\alpha_1] = \text{ and }\chi_2  =  \chi(L_{2,a}^+)=  \sum_{0\leq
  r\leq s}[-r\alpha_1 - s\alpha_2].$$

\begin{rem} {\rm The representations $X_{1,a}$ and $X_{2,a}$ have
      the same character in this case (this is not true for general
      $\g$):
$$\chi(X_{1,a}) = \chi(X_{2,a}) = \sum_{\lambda,\mu\geq
  0}(1+\text{Min}(\lambda,\mu))[-\lambda\alpha_1-\mu\alpha_2],$$ which
is the character of the Verma module of $\sw_3$ of highest weight
$0$. However, the representations $X_{1,a}$, $X_{2,a}$ are not
evaluation modules of this Verma module as the action of $U_q(\wh\bo)$
can not be extended to the full quantum affine algebra.
}
\qed\end{rem}

We view $\chi_1, \chi_2$ as elements of $K_0(\mathcal{O})$. Now we can
define the $Q$ and $\wt{Q}$ variables:
$$Q_{1,a} = [L_{1,a}^+]\text{ , }  \qquad \wt{Q}_{1,a} = [
X_{1,aq^{-2}}]\chi_1^{-1}\left[-\frac{\alpha_1}{2}\right],$$
$$Q_{2,a} = [L_{2,a}^+]\text{ , }  \qquad \wt{Q}_{2,a} = [
X_{2,aq^{-2}}]\chi_2^{-1}\left[-\frac{\alpha_2}{2}\right] .$$
Then get the $Q\wt{Q}$-system in $\text{Frac}(K_0(\mathcal{O}))$ as
in \cite[Formulas (5.4), (5.5)]{BHK}: 
$$\left[\frac{\alpha_1}{2}\right]Q_{1,aq^{-1}}\wt{Q}_{1,aq} -
\left[-\frac{\alpha_1}{2}\right] Q_{1,aq}\wt{Q}_{1,aq^{-1}} =
Q_{2,a}.$$
$$\left[\frac{\alpha_2}{2}\right]Q_{2,aq^{-1}}\wt{Q}_{2,aq} -
\left[-\frac{\alpha_2}{2}\right] Q_{2,aq}\wt{Q}_{2,aq^{-1}} =
Q_{1,a}.$$

Let $\Glie=B_2$. We have given in section \ref{dualex} an explicit
formula for $\chi_2$ and a description of the representation
$L_{2,a}^+$. The corresponding $Q\wt{Q}$-relation is
$$\left[\frac{\alpha_2}{2}\right]Q_{2,aq^{-1}}\wt{Q}_{2,aq} -
\left[-\frac{\alpha_2}{2}\right] Q_{2,aq}\wt{Q}_{2,aq^{-1}}
= Q_{1,aq^{-1}}Q_{1,aq}$$
where the $Q_{i,a}$ are classes of prefundamental representations and
$\wt{Q}_{2,a} = [X_{2,aq^{-2}}]\chi_2^{-1}
\left[-\frac{\alpha_2}{2}\right]$.

\begin{rem}\label{exbf} Vladimir Bazhanov drew out attention to a system of
  relations in the case of $\sw_{n+1}$ which were introduced in
  \cite[Equation (1.3)]{BFLMS} in the context of finite-dimensional
  representations of the corresponding Yangians, which is equivalent
  to ours (for example, the analogues of the prefundamental
  representations for the Yangians of $sl_{n+1}$ are described in
  \cite{BFLMS} and called there ``partonic'' representations).
  Moreover, an explicit comparison between this system and the
  $Q\tilde{Q}$-system was kindly given to us by the referee which we
  now present (it is worth mentioning that in this case it coincides
  with the Hirota equations, according to \cite{BFLMS}). This is not
  immediately clear as the system of \cite{BFLMS} involves a seemingly
  different set of variables ${\bf Q}_{j_1,\cdots, j_i}$ with
  $j_1,\cdots, j_i\in\{1,\cdots, n\}$.  Let us pick what can be called
  a path in the Hasse diagram
	$$\mathcal{P} : \emptyset\subset\{\iota_1\}\subset
        \{\iota_1,\iota_2\}\subset\cdots \subset
        \{\iota_1,\cdots,\iota_n\} = \{1,2,\cdots,n\}.$$
Then \cite[Equation (1.3)]{BFLMS} with $I = \emptyset$ and $a =
\iota_1$, $b = \iota_2$ is the type $A$ $Q\tilde{Q}$-system with $i =
1$ after the identification of $Q_1$, $Q_2$, $\tilde{Q}_1$ with ${\bf
  Q}_{\iota_1}$, ${\bf Q}_{\iota_1, \iota_2}$, ${\bf
  Q}_{{\iota_1,\iota_2}\setminus\{\iota_1\}}$ (up to a multiplication
by a weight).
More generaly, if we consider a second path
$$\tilde{\mathcal{P}} : \emptyset\subset \{\tilde{\iota}_1\} =
\{\iota_1,\iota_2\}\setminus\{\iota_1\}
\subset \{\tilde{\iota}_1,\tilde{\iota_2}\} =
\{\iota_1,\iota_2,\iota_3\}\setminus\{\iota_2\}\subset\cdots\subset
\{1,\cdots,n\},$$
then \cite[Equation (1.3)]{BFLMS} with $I =
\{\iota_1,\cdots,\iota_{i-1}\}$ and $a = \iota_i$, $b = \iota_{i+1}$
is the $Q\tilde{Q}$-system at $i$ after the identification of $Q_j$,
$\tilde{Q}_j$ with ${\bf Q}_{\iota_1,\cdots,\iota_j}$, ${\bf
  Q}_{\tilde{\iota}_1,\cdots,\tilde{\iota}_j}$ (up to a multiplication
by a weight).
\end{rem}

\subsection{$Q\wt{Q}$-system for the twisted quantum affine
  algebras}    \label{QQ twisted}

There is also a $Q\wt{Q}$-system for the twisted quantum affine
algebras. To explain this, we use the notation of \cite[Section
2.4]{HIMRN} (except that the Lie algebra denoted by $\g$ in
\cite{HIMRN} will now be denoted by $\g'$, and $I$ will be denoted by
$I'$). Let $\sigma$ be an automorphism of the Dynkin diagram of a
simply-laced simple finite-dimensional Lie algebra $\Glie'$; that is,
a bijection $\sigma : I'\rightarrow I'$ of the set $I'$ of nodes of the
Dynkin diagram of $\Glie$ such that $C_{\sigma(i),\sigma(j)} =
C_{i,j}$ for any $i,j\in I'$, where $C$ is the Cartan matrix of
$\Glie'$. Let $r$ be the order of $\sigma$. Consider the twisted case,
so that $r\in\{2,3\}$ (in fact, $\Glie$ must be of type $A_n$ ($n\geq
2$), $D_n$ ($n\geq 4$), or $E_6$). Let $I_{\sigma}$ denote the set of
orbits of $\sigma$ and for $i\in I'$ we denote by $\overline{i}\in
I_{\sigma}$ the orbit of $i$.

Using the Cartan generators of $\gt$, we obtain an automorphism of
$\gt$ of the same order, which we also denote by $\sigma$. The Lie
algebra $\gt$ decomposes into a direct sum of eigenspaces of $\sigma$:
$$
\gt = \bigoplus_{\ol{i} \in \Z/r\Z} \gt_{\ol{i}},
$$
where $\gt_{\ol{0}}$ is the simple Lie algebra corresponding to the
quotient of the Dynkin diagram of $\gt$ by the action of the
automorphism. The twisted affine Kac--Moody algebra $\ghat$ is defined
as the universal central extension of the twisted loop algebra
$$
{\mc L}_\sigma \g = \bigoplus_{n \in \Z} \gt_{\ol{n}} \otimes z^n
$$
Note that its constant part is the simple Lie algebra $\gt_{\ol{0}}$,
the $\sigma$-invariants of $\g'$. The nodes of the Dynkin diagram of
$\gt_{\ol{0}}$ are naturally parametrized by $I_\sigma$.

There is a quantum affine algebra $U_q(\ghat)$ attached to the twisted
affine algebra $\ghat$, whose finite-dimensional representations were
studied by several authors, see \cite{HIMRN} and references therein. 
This algebra has a Borel subalgebra, and one
can define the corresponding category ${\mc O}$ in the same way as in
\cite{HJ} (see Section \ref{catO}). Though this category has not been
studied in the twisted case, it is natural to conjecture that it
contains analogues of the representations $L^+_{i,a}$ and $X_{i,a}$
defined in \cite{HJ} and the present paper, respectively, in the
untwisted case.

More precisely, for each $\overline{i}\in I_{\sigma}$, let us choose a
representative $i\in I'$ in such a way that
that $$(C_{i,j},C_{i,\sigma(j)},\cdots,C_{i,\sigma^{M-1}(j)})\neq
(0,\cdots,0)\Rightarrow C_{i,j} = -1.$$ We fix such a choice and
identify $\overline{i}$ and $i$ using this choice. Hence $C_{i,j}$ is
well-defined for any $i,j\in I_\sigma$. Then we expect that in the
category ${\mc O}$ in the twisted case there are representations
$L^+_{i,a}$ and $X_{i,a}$ for all $i\in I_\sigma$ and $a \in
\C^\times$.

Now let us define $Q_{i,a}$ and $\wt{Q}_{i,a}$ by formulas \eqref{Q
  and wtQ}, and for each $i\in I_{\sigma}$, set $q_i = q^{d_i}$
where
$$d_i = 
\begin{cases}
r    \text{ if $C_{i,\sigma(i)} = 2$,}
\\1  \text{ if $C_{i,\sigma(i)} = 0$,}
\\1/2\text{ if $C_{i,\sigma(i)} = -1$.}
\end{cases}$$

\begin{conj}    \label{QQ t}
  The variables $Q_{i,a}$, $\wt{Q}_{i,a}, i\in I_\sigma, a\in\CC^*$
  satisfy the following $Q\wt{Q}$-system:
\begin{equation*}
\left[\frac{\alpha_i}{2}\right]Q_{i,aq_i^{-1}}\wt{Q}_{i,aq_i}
  - \left[-\frac{\alpha_i}{2}\right] Q_{i,aq_i}\wt{Q}_{i,aq_i^{-1}}
\end{equation*}
\begin{equation}    \label{QQ tw}
\begin{split}
  = \begin{cases} \left(\prod_{ j\sim i, d_j = r}Q_{j,a}\right)
    \left(\prod_{j\sim i,d_j \neq r\text{ and } b,b^r =
        a}Q_{j,b}\right)&\text{ if $d_i = r$,} \\ \left(\prod_{j\sim
        i, d_j = r}Q_{j,a^r}\right)\left(\prod_{j\sim i,d_j \neq
        r}Q_{j,a}\right)&\text{ if $d_i = 1$,} \\Q_{i,-a} \times\left(
      \prod_{j\sim i}Q_{j,a}\right)&\text{ if $d_i =
      1/2$,}\end{cases}\end{split}\end{equation}
where the products run on the $j\in I_\sigma$, $j\sim i$ means
$C_{i,j} < 0$ as above and $\alpha_i$ is a simple root of
$\Glie^\sigma$.
\end{conj}

We note that we have obtained the system \eqref{QQ tw} by a kind of
``folding'' of the $Q\wt{Q}$-system for the simply-laced Lie algebra
$\g'$ (analogously to how the $T$-system in the twisted case was
written in \cite{KS2} and established in \cite{HIMRN} by ``folding''
the $T$-system in the untwisted case).

We expect that the proof of this conjecture can be obtained along the
same lines as our proof in Section \ref{proof} of the $Q\wt{Q}$-system
in the case of untwisted quantum affine algebras. This will be discussed
in another paper.

\subsection{$QQ^*$-system}    \label{QQstar}

Another system of relations in $K_0({\mc O})$ was established in
\cite[Section 6.1.3]{HL}. It arises naturally in the context of
cluster algebras, as the first step of the Fomin--Zelevinsky mutation
relations. Namely, for $i\in I$ and $a\in\CC^\times$ we have
$$Q_{i,a}Q_{i,a}^*  = \prod_{j,C_{j,i}\neq 0}Q_{j,aq^{-d_jC_{j,i}}} 
+ [-\alpha_i]\prod_{j,C_{j,i}\neq 0}Q_{j,aq^{d_j C_{j,i}}}$$
where 
$$Q_{i,a}^* = [L(\Psib_{i,a}^{-1}\prod_{j,C_{j,i} \neq
  0}\Psib_{j,aq^{-d_jC_{j,i}}})]\text{ , }Q_{i,a} = [L_{i,a}^+].$$ This
  is the relation \cite[Equation 6.14]{HL}. Note that $Q_{i,a}^*$ is
  the mutated cluster variable obtained from $Q_{i,a}$ in a certain
  subring of $K_0({\mc O})$, which was introduced and proved to be a
  cluster algebra in \cite{HL}.

To distinguish them from
  the relations of the $Q\wt{Q}$-system, we call these relations the
  $QQ^*$-system.

  An analogous system of relations was also established in
  \cite[Example 7.8]{HL} in terms of the negative prefundamental
    representations. Taking the duals in those relations, we get the
    $QQ^*$-system in the Grothendieck ring of the dual category
    $\mathcal{O}^*$, with the variables $Q_{i,a} = [R_{i,a^{-1}}^+]$,
    $Q_{i,a}^* = L(\Psib_{i,a^{-1}}^{-1}\prod_{j,C_{j,i} \neq
      0}\Psib_{j,a^{-1}q^{d_jC_{j,i}}})$ and with $[-\alpha_i]$
    replaced by $[\alpha_i]$.

\subsection{$Q\wt{Q}$-system for the quantum ${\mathfrak g}{\mathfrak
    l}_1$ toroidal algebra}    \label{toroidal}

An analogue of the $Q\wt{Q}$-system holds for the quantum ${\mathfrak
  g}{\mathfrak l}_1$ toroidal algebra $\mathcal{E}$ as well. The
quantum parameters of $\mathcal{E}$ are $q_1$, $q_2$, $q_3$ satisfying
$(q_1q_3)^{-1} = q_2 = q^2$ and we use the notations of \cite{FJMM}
except that we replace their spectral parameter $z$ by $z^{-1}$ for
consistency.  For $a\in\CC^\times$ we introduce the representation $L_a$ of
highest $\ell$-weight
$(1-az)^{-1}(1-q_1^{-1}az)(1-q_3^{-1}az)$. $L_a^+$ denotes the
prefundamental representation of highest $\ell$-weight $(1- az)$.

If we set as above $Q_a = [L_a^+]$ and $\wt{Q}_a =
[X_{aq^{-2}}][-\frac{\alpha}{2}]\chi^{-1}$, where $\chi =
\chi(L_a^+)$, then we obtain the following $Q\wt{Q}$-system:
\begin{equation}
[\frac{\alpha}{2}]Q_{aq^{-1}}\wt{Q}_{aq} -
  [-\frac{\alpha}{2}]Q_{aq}\wt{Q}_{aq^{-1}}
=Q_{aq^{-1}q_1^{-1}}Q_{aq^{-1}q_3^{-1}},
\end{equation}
where  $[\pm\frac{\alpha}{2}]$ is the class of the one-dimensional
representation corresponding to $t^{\pm \frac{1}{2}}$.

The proof is analogous to the proof of the $Q\wt{Q}$-system presented
in the next section. One needs to use the theory of $q$-characters of
the category ${\mc O}$ of the quantum
${\mathfrak g}{\mathfrak l}_1$ toroidal algebra, which has recently
been developed in \cite{FJMM} in parallel with the theory for the
quantum affine algebras \cite{Fre,HJ,FH}.

Note that the $QQ^*$-system for quantum affine algebras discussed in
Section \ref{QQstar} also has an analogue for the ${\mathfrak
  g}{\mathfrak l}_1$ quantum toroidal algebra. This is the system of
relations obtained in \cite[Formula (4.24)]{FJMM}.

\begin{rem} We were informed by Michio Jimbo that the above
  $Q\tilde{Q}$-system for the quantum ${\mathfrak g}{\mathfrak l}_1$
  toroidal algebra was known to the authors of \cite{FJMM}.
\qed\end{rem}

\section{Proof of the $Q\wt{Q}$-system}    \label{proof}

In this section we prove the $Q\wt{Q}$-system \eqref{relation1} stated
in Theorem \ref{relation}. One of the crucial tools used in the proof
is the theory of $q$-characters.

\subsection{$q$-characters}\label{qchar}

For a $U_q(\mathfrak{b})$-module $V$ and $\Psib\in\tb_\ell^*$, 
the linear subspace
\begin{align}
V_{\Psibs} =
\{v\in V\mid
\exists p\geq 0, \forall i\in I, 
\forall m\geq 0,  
(\phi_{i,m}^+ - \Psi_{i,m})^pv = 0\}
\label{l-wtsp} 
\end{align}
is called the $\ell$-weight space of $V$ of $\ell$-weight $\Psib$.

\begin{thm}\cite{HJ} For $V$ in category $\mathcal{O}$, $V_{\Psib}\neq
  0$ implies $\Psib\in\mfr$.
\end{thm}

Given a map $c : \mfr\rightarrow \ZZ$, consider its support
\[
 \text {supp}(c) = \{\Psib \in  \mfr \mid c(\Psib) \not = 0\}.
\]
Let
$\mathcal{E}_\ell$ 
be the additive group of maps
$c : \mfr\rightarrow \ZZ$  
such that $\varpi(\text{supp}(c))$ is contained in a finite union
of sets of the form $D(\mu)$, and such that 
for every $\omega\in P_\Q$, 
the set $\text{supp}(c) \cap \varpi^{-1}(\{\omega\})$
is finite.
The map $\varpi$ is naturally extended to a surjective homomorphism 
$\varpi : \mathcal{E}_\ell\rightarrow \mathcal{E}$.

For $\Psib\in\mfr$, we define $[\Psib] = \delta_{\Psibs,.}\in\mathcal{E}_\ell$.

Let $V$ be a $U_q(\mathfrak{b})$-module in category $\mathcal{O}$. 
We define \cite{Fre, HJ} the $q$-character of $V$ as
\begin{align}
\chi_q(V) = 
\sum_{\Psibs\in\mfr}  
\mathrm{dim}(V_{\Psibs}) [\Psib]\in \mathcal{E}_\ell\,.
\label{qch}
\end{align}

\begin{example}\label{onedim} For $\omega\in\tb^\times$, the $q$-character
  of the $1$-dimensional representation
  $[\omega]$ is just its $\ell$-highest weight $\chi_q([\omega]) =
  [\omega]$. That is why the use of the same notation $[\omega]$ will
  not lead to confusion.
\end{example}

Note that we have $\chi(V) = \varpi(\chi_q(V))$ for $V$ a representation in the category $\mathcal{O}$.

By \cite[Theorem 3]{Fre} and \cite[Proposition
3.12]{HJ}, we have the following.
  
\begin{prop} The $q$-character morphism 
$$
\chi_q : \text{Rep}(U_q(\mathfrak{b}))\rightarrow
\mathcal{E}_\ell,\quad [V]\mapsto \chi_q(V),
$$
is an injective ring morphism.
\end{prop}

It is proved in \cite{Fre} that a finite-dimensional
$U_q(\wh{\Glie})$-module $V$ satisfies 
$V = \bigoplus_{m\in\mathcal{M}} V_{m\left(\phi(z)\right)}$.
In particular, $\chi_q(V)$ can be viewed  
as an element of $\Yim$. It is proved in \cite{Fre, Fre2} that
if moreover $V = L(m)$ is simple, then 
$$\chi_q(L(m))\in m (1 + \ZZ[A_{i,a}^{-1}]_{i\in I, a\in\CC^\times}).$$

\begin{thm}\label{formuachar}
(i)  For any $a\in\CC^\times$, $i\in I$ we have
$$\chi_q(L_{i,a}^+) = \left[\Psib_{i,a}\right] \chi(L_{i,a}^+) = \left[\Psib_{i,a}\right] \chi(L_{i,a}^-).$$

(ii)  For any $a\in\CC^\times$, $i\in I$ we have
$$\chi_q(L_{i,a}^-)\in \left[\Psib_{i,a}^{-1}\right] (1 + A_{i,a}^{-1}\ZZ[[A_{j,b}^{-1}]]_{j\in I, b\in\CC^\times}).$$
\end{thm}

\begin{rem}{\rm
(i) The statement (i) is proved in \cite{HJ, FH}. 

(ii) The statement (ii) is proved in \cite{HJ}: indeed it is established there that $\left[\Psib_{i,a}\right]\chi_q(L_{i,a}^-)$
is a certain limit of $q$-characters of KR modules as a formal power series in the $A_{j,b}^{-1}$.
It is of the form written in the Theorem by \cite[Lemma 4.4]{HCrelle}.

(iii) As a consequence the $\chi_i\in\mathcal{E}$ defined in formula (\ref{chii}) is equal to
$$\chi_i = \chi(L_{i,a}^+) = \chi(L_{i,a}^-) = \left[\Psib_{i,a}^{-1}\right]\chi_q(L_{i,a}^+).$$}
\qed\end{rem}

\begin{example}\label{ex-calcul}
{\rm
In the case $\wh{\Glie} = \widehat{\sw}_2$, we have:
\[
\chi_q(L_{1,a}^+) = [(1 - za)]\sum_{r\geq 0} [-2r\omega_1]\text{ , }\chi_q(L_{1,a}^-) = \left[\frac{1}{(1-za)}\right]\sum_{r\geq 0}
  A_{1,a}^{-1}A_{1,aq^{-2}}^{-1}\cdots A_{1,aq^{-2(r-1)}}^{-1}. 
\]
}
\end{example}

\begin{example}
{\rm
In the case of $\Glie=B_2$ (see section \ref{dualex}), we have:

$$\chi_q(L_{2,a}^+) = [(1 - za)]\sum_{T\in\mathcal{T}} [\omega_T],$$
$$\chi_q(L_{2,a}^-) =
\left[\frac{1}{(1-za)}\right]\sum_{T\in\mathcal{T}}\prod_{j\geq
  0}(A_j(T_{1,j})B_j(T_{2,j}))^{-1},$$
where $A_j(1,2,0,\overline{1}) =
(1,A_{1,aq^{-4j+2}},A_{1,aq^{-4j+2}}A_{2,aq^{-4j+4}},A_{1,aq^{-4j+2}}
A_{2,aq^{-4j+4}}A_{2,aq^{-4j+2}})$
and $B_j(2,0,\overline{1},\overline{2}) =
(1,A_{2,aq^{-4j}},A_{2,aq^{-4j}}A_{2,aq^{-4j-2}},A_{2,aq^{-4j}}A_{2,aq^{-4j-2}}A_{1,aq^{-4j}})$.
}
\end{example}

\begin{thm}\cite{FH}\label{stensor} Any tensor product of positive
  (resp. negative) prefundamental representations $L_{i,a}^+$
  (resp. $L_{i,a}^-$)
  is simple. 
\end{thm}

\subsection{Examples} Let us explain the examples from Section
\ref{firstex} in terms of $q$-characters.

For $\Glie = \sw_2$, the relations follow directly from the
$q$-character explicit formulas given in Example \ref{ex-calcul}.

For $\Glie = \sw_3$, we can prove the following explicit $q$-character
formula (see the general result in Proposition \ref{genfor}):
$$\chi_q(X_{1,a}) = [\Psib_{1,a}^{-1}\Psib_{2,aq}] \chi_2  \sum_{r\geq
  0}(A_{1,a}A_{1,aq^{-2}}\cdots A_{1,aq^{-2(r-1)}})^{-1}$$
and an analog formula for $\chi_q(X_{2,a})$.

For $\Glie = B_2$, we can prove the following explicit $q$-character
formula:
$$\chi_q(X_{2,a}) = [\Psib_{2,a}^{-1}\Psib_{1,a}\Psi_{1,aq^2}] \chi_2  \sum_{r\geq
  0}(A_{2,a}A_{2,aq^{-2}}\cdots A_{2,aq^{-2(r-1)}})^{-1}.$$
This gives some insights on the structure of the representation
$X_{2,a}$: it has a basis of the $\ell$-weight vectors 
$$X_{2,a} = \bigoplus_{T\in\mathcal{T},r\geq 0} \CC v_{T,r}$$
where $v_{T,r}$ has $\ell$-weight 
$$[\Psib_{2,a}^{-1}\Psib_{1,a}\Psib_{1,aq^2}][-\omega_T]
(A_{2,a}A_{2,aq^{-2}}\cdots A_{2,aq^{-2(r-1)}})^{-1}.$$ 

\subsection{A $q$-character formula}

Our Theorem \ref{relation} is a consequence of the following

\begin{prop}\label{genfor} For any $i\in I$, $a\in\CC^\times$, we have
\begin{equation}\label{adem}\chi_q(X_{i,a} ) = [\wt{\Psib}_{i,a}]
  \chi_{i,a}  \prod_{j\neq i}\chi_j^{-C_{i,j}},\end{equation}
where
$$\chi_{i,a} = \sum_{r\geq 0}(A_{i,a}A_{i,aq_i^{-2}}\cdots
A_{i,aq_i^{-2(r-1)}})^{-1}\in\mathcal{E}_\ell.$$
\end{prop}

\begin{rem}{\rm This explicit $q$-character formula implies
$$\chi(X_{i,a}) = \chi_i \prod_{j\neq i}\chi_{j}^{-C_{i,j}}.$$}
\qed\end{rem}

By using the automorphism $\tau_a$, it suffices to prove the formula for $a = 1$. 

For $m\in\ZZ$ we denote by $[m]$ its integer part.

For $N \leq 0 < M$ let us set
$$\wt{\Psib}_i^{(N,M)} = \wt{\Psib}_{i,1} \Psib_{i,q_i^{-2N}}
 \left(\prod_{j|C_{i,j} = - 1}\Psib_{j,q^{r_i +
       2r_j[1+(M-r_i)/(2r_j)]}}^{-1}\right) \cdot$$
$$\cdot
\left(\prod_{j|C_{i,j} = -2}\Psib_{j,q^{4[1+M/4]}}^{-1}\Psib_{j,q^{6 +
      4[(M-2)/4]}}^{-1}\right) \cdot$$
$$\cdot\left(\prod_{j|C_{i,j} = -3}\Psib_{j,q^{5+6[(M+1)/6]}}^{-1}\Psib_{j,q^{7 +6[(M-1)/6]}}^{-1}\Psib_{j,q^{8+6[(M-3)/6]}}^{-1}\right)$$
$$=(\wt{Y}_{i,q_i^{-1}}\wt{Y}_{i,q_i^{-3}}\cdots \wt{Y}_{i,q_i^{1 - 2N}})
 \left(\prod_{j|C_{i,j} = - 1}\wt{Y}_{j,q^{r_i +
       r_j}}\wt{Y}_{j,q^{r_i+3r_j}}\cdots \wt{Y}_{j,q^{r_i + 2
       r_j[(M-r_i)/(2r_j)]+r_j}}\right) \cdot$$
$$\cdot\left(\prod_{j|C_{i,j} =
    -2}(\wt{Y}_{j,q^{2}}\wt{Y}_{j,aq^6}\cdots \wt{Y}_{j,q^{4[M/4] +
      2}})(\wt{Y}_{j,q^4}\wt{Y}_{j,q^8}\cdots
  \wt{Y}_{j,q^{4+4[(M-2)/4]}})\right)\cdot$$
$$\cdot \left(\prod_{j|C_{i,j} = -3}
(\wt{Y}_{j,q^{2}}\wt{Y}_{j,q^8}\cdots \wt{Y}_{j,q^{2+6[(M+1)/6]}})
(\wt{Y}_{j,q^4}\wt{Y}_{j,q^{10}}\cdots \wt{Y}_{j,q^{4 +6[(M-1)/6]}})
\right. \cdot
$$
$$
\cdot \left.(\wt{Y}_{j,q^6}\wt{Y}_{j,q^{12}}\cdots
  \wt{Y}_{j,q^{6[1+(M-3)/6]}}) \right).
$$
As $\wt{\Psib}_i^{(N,M)}$ is expressed as a product of variables $\wt{Y}_{j,b}$, the
representation $L(\wt{\Psib}_i^{(N,M)})$ is finite-dimensional (see Section \ref{fdrep}). We
will also consider the $\ell$-weight $\wt{\Psib}_i^{(M)}$  obtained
from $\wt{\Psib}_i^{(N,M)}$ by removing the factors depending on
$N$, that is 
$$\wt{\Psib}_i^{(M)} = \wt{\Psib}_i^{(N,M)} \Psib_{i,q_i^{2N}}^{-1}.$$
As discussed in Section \ref{qchar}, we have
$$\chi_q(L(\wt{\Psib}_i^{(N,M)}))\in
[\wt{\Psib}_i^{(N,M)}]\ZZ[A_{j,q^r}^{-1}]_{j\in I, r\in\ZZ}.$$
As moreover it follows from Theorem \ref{formuachar} that 
$$\chi_q(\Psib_{i,q_i^{2N}}^{-1})\in [\Psib_{i,q_i^{2N}}^{-1}]\ZZ[[A_{j,q^r}^{-1}]]_{j\in I, r\in\ZZ},$$
we get
$$\chi_q(L(\wt{\Psib}_i^{(M)}))\in
[\wt{\Psib}_i^{(M)}]\ZZ[[A_{j,q^r}^{-1}]]_{j\in I, r\in\ZZ}.$$
Following \cite{HL0}, let us consider the truncated $q$-characters
$$\chi_q^{< M}(L(\wt{\Psib}_i^{(N,M)}))\in\mathcal{E}_\ell \text{ and }\chi_q^{< M}(L(\wt{\Psib}_i^{(M)}))\in\mathcal{E}_\ell$$ 
which are the sum (with multiplicity)
of the $\ell$-weights $m$ occurring in $\chi_q(L(\wt{\Psib}_i^{(N,M)}))$
(resp. $\chi_q(L(\wt{\Psib}_i^{(M)}))$) such that
$$m(\wt{\Psib}_i^{(N,M)})^{-1} \in \ZZ[A_{i,q^r}]_{i\in I , r< M}.$$

\begin{lem} We have:
$$\chi_q^{< M}(L(\wt{\Psib}_i^{(N,M)})) = [\wt{\Psib}_i^{(N,M)}]
\sum_{0\leq r\leq -N+1}(A_{i,1}A_{i,q_i^{-2}}\cdots A_{i,q_i^{-2(r-1)}})^{-1},$$
$$\chi_q^{< M}(L(\wt{\Psib}_i^{(M)})) = [\wt{\Psib}_i^{(M)}]\chi_{i,1}.$$
\end{lem}

\begin{rem}{\rm The first formula proves a particular case of \cite[Conjecture 7.15]{HL}.}\qed\end{rem}

\begin{proof} For the first formula, note that
  $L(\wt{\Psib}_i^{(N,M)})$ is a subquotient of
$$L(\wt{\Psib}_i^{(N,0)})\otimes L(\wt{\Psib}_i^{(0,M)}),$$
where we set 
$$\wt{\Psib}_i^{(0,M)} = \wt{\Psib}_i^{(N,M)}(\wt{\Psib}_i^{(N,0)})^{-1}.$$
Here $L(\wt{\Psib}_i^{(N,0)})$ is a KR-module and $
L(\wt{\Psib}_i^{(0,M)})$ is a tensor product of KR-modules which is
also simple (we can argue as in \cite[Proposition 5.3]{HCrelle}). Then
it is proved in \cite[Lemma 4.4]{HCrelle} that $\wt{\Psib}_i^{(0,M)}$
is the only $\ell$-weight in $\chi_q(L(\wt{\Psib}_i^{(0,M)}))$ which
may occur in $\chi_q^{< M}(L(\wt{\Psib}_i^{(0,M)}))$; that is,
$$\chi_q^{< M}(L(\wt{\Psib}_i^{(0,M)})) = [\wt{\Psib}_i^{(0,M)}].$$
Consequently the $\ell$-weights occurring in $\chi_q^{<
  M}(L(\wt{\Psib}_i^{(N,M)}))$ are of the form 
$$\wt{\Psib}_i^{(N,M)}\Psib$$ 
where $(\wt{\Psib}_i^{(0,M)})^{-1}\Psib$ is an $\ell$-weight occurring
in $\chi_q(L(\wt{\Psib}_i^{(N,0)}))$. As $L(\Psib_i^{(N,0)})$ is a KR
modules, it is proved in \cite[Section 4]{HCrelle} that its
$q$-character can be computed by using the algorithm introduced in
\cite[Section 5.5]{Fre2}.  We also have precise information on the
monomials occurring in ts $q$-character in \cite[Lemma
5.5]{HCrelle}. In particular, we have the following: suppose that
$\Psib$ is not of the form
$$(A_{i,1}A_{i,q_i^{-2}}\cdots A_{i,q_i^{-2(r-1)}})^{-1}\text{ for
  some }0\leq r\leq -N+1,$$ that is it is not in the set denoted by
$\mathcal{B}'$ in \cite[Lemma 5.5]{HCrelle}. Then there is an
$\ell$-weight $\Psib'$ occurring in $\chi_q^{<
  M}(L(\wt{\Psib}_i^{(N,M)}))$ whose weight is of the form
$-r\alpha_i-\alpha_j$ for some $0\leq r\leq -N+1$ and some $j\sim i$
(in \cite[Lemma 5.5]{HCrelle} this is stated with $j\neq i$, but the
Frenkel-Mukhin algorithm mentioned above gives immediately that
necessarily $j\sim i$).  Let $r$ minimal with this property. There is
$\alpha\in\ZZ$ such that
$$\wt{\Psib}_i^{(N,M)}  (A_{i,1}A_{i,q_i^{-2}}\cdots
A_{i,q_i^{-2(r-1)}})^{-1}A_{j,q^\alpha}^{-1}$$
occurs as an $\ell$-weight in $\chi_q(L(\wt{\Psib}_i^{(N,M)}))$. But
such an $\ell$-weight
satisfies exactly the hypothesis of \cite[Theorem 5.1]{H} which gives
sufficient conditions
so that a monomial do not occur in the $q$-character of a simple
module (here the $i$ in
\cite[Theorem 5.1]{H} is $j$, $m$ is $\wt{\Psib}_i^{(N,M)}  (A_{i,1}A_{i,q_i^{-2}}\cdots A_{i,q_i^{-2(r-1)}})^{-1}A_{j,q^\alpha}^{-1}$ and $M$ is $m A_{j,q^\alpha}$ up to a constant $\ell$-weight multiple). Hence we get a contradiction.

For the second formula, it follows from \cite[Theorem 6.1]{HJ} generalized in \cite[Theorem 7.6]{HL} 
that we can take the limit $N\rightarrow -\infty$, that is 
$[\wt{\Psib}_i^{(N,M)}]^{-1}\chi_q^{< M}(L(\wt{\Psib}_i^{(N,M)}))$ converges to 
$[\wt{\Psib}_i^{(M)}]^{-1}\chi_q^{< M}(L(\wt{\Psib}_i^{(M)}))$ as a formal power series
in the $A_{j,b}^{-1}$.
\end{proof}

Consider the partial ordering $\preceq$ is defined on $\mathcal{E}_\ell$ so that $\chi\preceq \chi'$ if the 
coefficients of $\chi$ are lower than those of $\chi'$.

\begin{lem}We have $\chi_q(X_{i,1} ) \preceq [\wt{\Psib}_{i,1}]  \chi_{i,1}  \prod_{j\neq i}\chi_j^{-C_{i,j}}$.
\end{lem}

\begin{proof} $X_{i,1}$ is a subquotient of 
$$L(\wt{\Psib}_i^{(M)})\otimes L(\wt{\Psib}_{i,1}(\wt{\Psib}_i^{(M)})^{-1}).$$
By Theorem \ref{stensor}, $L(\wt{\Psib}_{i,1}(\wt{\Psib}_i^{(M)})^{-1})$ is a simple tensor product of positive prefundamental representations and
$$\chi_q(L(\wt{\Psib}_{i,1}(\wt{\Psib}_i^{(M)})^{-1})) = [\wt{\Psib}_{i,1}(\wt{\Psib}_i^{(M)})^{-1}] \prod_{j\neq i}\chi_j^{-C_{i,j}}.$$
This implies
$$[\wt{\Psib}_{i,1}^{-1}]\chi_q(X_{i,1})\preceq [\wt{\Psib}_{i,1}^{-1}]\chi_q(L(\wt{\Psib}_i^{(M)})) \chi_q(L(\wt{\Psib}_{i,1}(\wt{\Psib}_i^{(M)})^{-1}))$$
$$ =   [(\wt{\Psib}_i^{(M)})^{-1}]\chi_q(L(\wt{\Psib}_i^{(M)})) \prod_{j\neq i}\chi_j^{-C_{i,j}}.$$
This is true for any $M> 0$. For each $\ell$-weight $\Psib$ in the left term, there is $M$ such that no $A_{j,q^r}^{-1}$ with $r \geq M$ occurs as a factor in $\Psib$. So this $\ell$-weight $\Psib$ occurs only 
in the product with the truncated $q$-character
$$[(\wt{\Psib}_i^{(M)})^{-1}]\chi_q^{< M}(L(\wt{\Psib}_i^{(M)})) \prod_{j\neq i}\chi_j^{-C_{i,j}} = \chi_{i,1} \prod_{j\neq i}\chi_j^{-C_{i,j}}.$$
\end{proof}

To conclude, we prove

\begin{lem}We have $\chi_q(X_{i,1} ) \succeq [\wt{\Psib}_{i,1}  \chi_{i,1}]  \prod_{j\neq i}\chi_j^{-C_{i,j}}$.
\end{lem}

\begin{proof}
Consider the representation 
$$X_{i,1}\otimes L(\wt{\Psib}_{i,1}^{-1}\Psib_{i,1}^{-1}).$$ 
It admits $L(\Psib_{i,1}^{-1})$ as a simple constituent. By Theorem \ref{stensor}, $L(\wt{\Psib}_{i,1}^{-1}\Psib_{i,1}^{-1})$ is a 
tensor product of negative prefundamental representations. 
Let $\Psib'$ be an $\ell$-weight occurring in 
$$\chi_{i,1}\preceq \Psib_{i,1}\chi_q(L(\Psib_{i,1}^{-1})).$$ 
Hence by (ii) in Theorem \ref{formuachar}, $\Psib'$ is a product $M\wt{\Psib}_{i,1}^{-1}$ where $M$ is an $\ell$-weight of $\chi_q(X_{i,1})$. 
We get 
$$\chi_q(X_{i,1})\succeq [\wt{\Psib}_{i,1}]\chi_{i,1}.$$
Note that for $\Psib'$ an $\ell$-weight in $\chi_{i,1}$, the product $\Psib_{i,1}\wt{\Psib}_{i,1}\Psib'$ is a monomial in the $\Psib_{j,a}$. So  by Theorem \ref{stensor} and (i) in Theorem \ref{formuachar}, $L(\Psib_{i,1}\wt{\Psib}_{i,1}\Psib')$ is a 
 simple tensor products of positive prefundamental representations and
$$\chi_q(L(\Psib_{i,1}\wt{\Psib}_{i,1}\Psib')) = [\Psib_{i,1}\wt{\Psib}_{i,1}\Psib']\chi_i \prod_{j\neq i}\chi_j^{-C_{i,j}}.$$
 Hence these simple modules are simple constituents of 
$$[X_{i,1}\otimes L_{i,1}^+],$$ that is
$$[\Psib_{i,1}]\chi_i \chi_q(X_{i,1}) = \chi_q(X_{i,1})\chi_q(L_{i,1}^+)\succeq \chi_{i,1} [\Psib_{i,1}\wt{\Psib}_{i,1}] \chi_i \prod_{j\neq i}\chi_j^{-C_{i,j}},$$
which implies the result.
\end{proof}

\subsection{Completion of the proof of Theorem \ref{relation}}

We can now complete the proof of Theorem \ref{relation}.

Note that $C_{i,j} < -1$ implies $r_i = 1$.  It suffices to prove that
$$
 \left(\prod_{j|C_{i,j} = - 1}\Psib_{j,a}\right)\left(\prod_{j|C_{i,j} = -2}\Psib_{j,aq^{-1}}\Psib_{j,aq}\right)\left(\prod_{j|C_{i,j} = -3}\Psib_{j,aq^{-2}}\Psib_{j,a}\Psib_{j,aq^2}\right) \chi_{i,aq_i^{-1}}$$ 
$$= [-\alpha_i] \Psib_{i,aq_i}\Psib_{i,aq_i^{-3}}^{-1}
 \left(\prod_{j|C_{i,j} = - 1}\Psib_{j,aq_i^{-2}}\right)\left(\prod_{j|C_{i,j} = -2}\Psib_{j,aq^{-3}}\Psib_{j,aq^{-1}}\right)$$
$$\times\left(\prod_{j|C_{i,j} = -3}\Psib_{j,aq^{-4}}\Psib_{j,aq^{-2}}\Psib_{j,a}\right)\chi_{i,aq_i^{-3}}$$ 
$$+  \left(\prod_{j|C_{i,j} = - 1}\Psib_{j,a}\right)\left(\prod_{j|C_{i,j} = -2}\Psib_{j,aq^{-1}}\Psib_{j,aq}\right)\left(\prod_{j|C_{i,j} = -3}\Psib_{j,aq^{-2}}\Psib_{j,a}\Psib_{j,aq^2}\right),$$
that is
$$\chi_{i,aq_i^{-1}}
= 1  +
 [-\alpha_i] \Psib_{i,aq_i}\Psib_{i,aq_i^{-3}}^{-1}\chi_{i,aq_i^{-3}}$$
\begin{equation}\label{term}\times \left(\prod_{j|C_{i,j} = - 1}\Psib_{j,a}^{-1}\Psib_{j,aq_i^{-2}}\right)\left(\prod_{j|C_{i,j} = -2}\Psib_{j,aq}^{-1}
\Psib_{j,aq^{-3}}\right)\left(\prod_{j|C_{i,j} = -3}
\Psib_{j,aq^2}^{-1}
\Psib_{j,aq^{-4}}\right)
.\end{equation}
Note that $A_{i,aq_i^{-1}} [-\alpha_i] \Psib_{i,aq_i}^{-1}\Psib_{i,aq_i^{-3}}$ is equal to
$$
\left(\prod_{j|C_{j,i} = - 1}\Psib_{j,aq^{-r_i-r_j}}\Psib_{j,aq^{r_j-r_i}}^{-1}\right)
\left(\prod_{j|C_{j,i} = -2}\Psib_{j,aq^{-r_i-2}}\Psib_{j,aq^{2-r_i}}^{-1}\right)$$
$$\times\left(\prod_{j|C_{j,i} = -3}\Psib_{j,aq^{-r_i-3}}\Psib_{j,aq^{3-r_i}}^{-1}\right)
.$$ 
This is exactly the last factor in Equation (\ref{term}):

if $r_j = 1$, then $C_{i,j} = -1$ and both factors are equal to $\Psib_{j,a}^{-1}\Psib_{j,aq_i^{-2}}$.

if $r_j = 2$ and $r_i= 1$, then $C_{i,j} = -2$, $C_{j,i} = -1$ and both factors are equal to $\Psib_{j,aq}^{-1}\Psib_{j,aq^{-3}}$.

if $r_j = r_i = 2$, then $C_{i,j} = C_{j,i} = -1$ and both factors are equal to $\Psib_{j,a}^{-1}\Psib_{j,aq^{-4}}$.

if $r_j = 3$ and $r_i = 1$, then $C_{i,j} = -3$, $C_{j,i} = -1$ and both factors are equal to $\Psib_{j,aq^2}^{-1}\Psib_{j,aq^{-4}}$.

\noindent We get the desired result because
$$\chi_{i,aq_i^{-1}} = 1 +  A_{i,aq_i^{-1}}^{-1}\chi_{i,aq_i^{-3}}.$$
\qed

\section{Bethe Ansatz}    \label{bethe}

We now derive the Bethe Ansatz equations from the $Q\wt{Q}$-system
\eqref{relation1}, following \cite{MRV1,MRV2}. We focus of the case
of untwisted affine algebras, but one can obtain the Bethe Ansatz
equations for the twisted affine algebras from the $Q\wt{Q}$-system
\eqref{QQ tw} in a similar way.

Suppose that we have an action of the commutative algebra $K_0({\mc
  O})$ on a vector space $V$, and let $v$ be one of its joint
eigenvectors. Then we obtain an algebra homomorphism from $K_0({\mc
  O})$ to $\C$. Let us denote the values of the elements $Q_{i,u}$ and
$\wt{Q}_{i,u}$ of $K_0({\mc O})$ under this homomorphism by ${\mb
  Q}_{i}(u)$ and $\wt{\mb Q}_{i}(u)$, respectively. Depending on the
space $V$, these functions will have different analytic properties.

In addition, under any homomorphism from $K_0({\mc O})$ to $\C$, we
have
$$
\left [\pm\frac{\alpha_i}{2}\right] \mapsto v_i^{\pm 1}
$$
for some $v_i \in \C^\times$, for all $i \in I$. (Note that what we
denoted by $v_i$ in \cite{FH} corresponds to $v_i^2$ here; however,
that $v_i$ was a formal variable in \cite{FH}, whereas here it is a
non-zero complex number.)

The relations in \eqref{relation} then give rise to algebraic
relations between these functions:
\begin{equation}    \label{Qsyst}
v_i {\mb Q}_{i}(uq_i^{-1})\wt{{\mb
    Q}}_{i}(uq_i) - v_i^{-1} {\mb
  Q}_{i}(uq_i)\wt{{\mb Q}}_{i}(uq_i^{-1})
\end{equation}
$$=
\left(\prod_{j|C_{i,j} = - 1}{\mb Q}_{j}(u)\right)\left(\prod_{j|C_{i,j} =
    -2}{\mb Q}_{j}(uq^{-1}){\mb Q}_{j}(uq)\right)\left(\prod_{j|C_{i,j} =
    -3}{\mb Q}_{j}(uq^{-2}){\mb Q}_{j}(u){\mb Q}_{j}(uq^2)\right) .$$

Now suppose that $w$ is a zero of ${\mb Q}_i(u)$ that is not a zero of
$\wt{\mb Q}_i(u)$ and that the terms in formula \eqref{Qsyst} have no
poles when $u=wq_i^{\pm 1}$ (we will refer to this as a genericity
condition). Substituting $u=wq_i^{\pm 1}$ into \eqref{Qsyst} and
taking the ratio of the resulting equations, we obtain:
\begin{equation}    \label{BAE}
v_i^{-2} \prod_{j \in I}
\frac{\mb{Q}_j(wq^{B_{ij}})}{{\mb Q}_j(wq^{-B_{ij}})} = -1,
\end{equation}
where $(B_{ij})$ is the symmetrized Cartan matrix, $B_{ij} =
(\al_i,\al_j)$. These are the Bethe Ansatz equations.

Thus, under the genericity condition, the zeros of ${\mb Q}_i(u)$ must
satisfy the Bethe Ansatz equations \eqref{BAE}.

In Section 5.6 of \cite{FH} (see also Section 6 of \cite{Fre}) we
obtained these equations in the case that $V$ is the tensor product of
irreducible finite-dimensional representations of $U_q(\ghat)$ and the
action of $K_0({\mc O})$ on $V$ is obtained using the standard
transfer-matrix construction.

If we switch to the dual category ${\mc O}^*$, so that $Q_{i,u}$
becomes $[R_{i,u}^+]$, as explained in Remark \ref{rem 3 parts},(iii),
then the corresponding Bethe Ansatz equation \eqref{BAE} is equivalent
to formula (5.8) of \cite{FH} (note that an overall minus sign is
missing in that formula). This case is special in that for a given
$V$, any eigenvalue of the transfer-matrix of $R_{i,a}^+$ on $V$ has
the form ${\mb Q}_i(u) = f_i(u) Q_i(u)$, where $f_i(u)$ is a universal
factor that depends only on $V$ and $i$, and $Q_i(u)$ is a polynomial
(this is an analogue of the Baxter polynomial). Thus, the analytic
behavior of ${\mb Q}_i(u)$ has a very special form in this case.

Though we did not prove it in \cite{FH}, we do expect that the
eigenvalues of the transfer-matrix of $L^+_{i,u}$ have the same
general form as those of the transfer-matrix of $R_{i,u}^+$. If this
is indeed the case, then equations (5.8) of \cite{FH} may also be
viewed as the equations on the zeros of the generalized Baxter
polynomials occurring in the eigenvalues of the transfer-matrix of
$L^+_{i,u}$ on $V$.

However, the derivation of the Bethe Ansatz equations presented in
\cite{FH} (following the analytic Bethe Ansatz method
\cite{R1,R2,R3,BR,KS}) is much less direct than the derivation
presented in this section. Indeed, the argument of \cite{FH} started
with the formula expressing the eigenvalues of the transfer-matrix of
a finite-dimensional representation $W$ of $U_q(\ghat)$ in terms of the
eigenvalues of the transfer-matrices of $R_{i,u}^+$ (or $L^+_{i,u}$),
see Theorem 5.11 of \cite{FH} (these are the analogues of Baxter's
$TQ$-relation). If we make a specific assumption about how poles get
canceled in this formula (namely, that the cancellation happens
between the terms in the formula corresponding to the monomials $M$
and $M A_{i,aq_i}^{-1}$ from the $q$-character), then we obtain the
above Bethe Ansatz equations \eqref{BAE}, see Section 5.8 of \cite{FH}
for details. In contrast, in our present argument we immediately get
the Bethe Ansatz equations under a mild genericity condition.

\medskip

The non-local quantum KdV Hamiltonians give us (conjecturally, see
Section \ref{nlham} below) another way to construct an action of
$K_0({\mc O})$, as explained in Section \ref{qkdv}. In this case, $V$
is a graded component in a representation of a ${\mc W}$-algebra
associated to $\g$; for example, a Fock representation. Under the same
genericity condition, the zeros of the corresponding eigenvalues
${\mb Q}_i(u)$ satisfy Bethe Ansatz equations \eqref{BAE} (with
specific values of $v_i$). Note that for $\wh{\sw}_2$ this was shown
in \cite{BLZ} using the quantum Wronskian relation, to which the
$Q\wt{Q}$-system reduces in the case of $\wh\sw_2$ (see Section
\ref{firstex}). In this case, the function ${\mb Q}_1(u)$ is expected
to be an entire function of $u$, see \cite{BLZ4,BLZ}.

\medskip

We close this section with two remarks. First, the Bethe Ansatz
equations \eqref{BAE} can be derived, in a similar fashion, from the
$QQ^*$-system of \cite{HL} (see Section \ref{QQstar}), under an
assumption that is similar to the above genericity condition.

Second, for the quantum ${\mathfrak g}{\mathfrak l}_1$ toroidal
algebra, a system of relations in $K_0({\mc O})$ was established in
\cite{FJMM}. It could be viewed as an analogue of the $QQ^*$-system of
\cite{HL} (for the analogue of the $Q\wt{Q}$-system, see Section
\ref{toroidal} above). In \cite{FJMM}, Bethe Ansatz equations were
derived from that system in a similar fashion. However, the authors of
\cite{FJMM} went a step further: they proved that the analogue of the
above genericity assumption is in fact not necessary. This gives us
hope that the genericity assumption can be dropped in the affine case
as well, for both the $QQ^*$-system and the $Q\wt{Q}$-system.

Finally, we can derive the Bethe Ansatz equations of \cite{FJMM} from
the $Q\wt{Q}$-system for the quantum ${\mathfrak g}{\mathfrak l}_1$
toroidal algebra from Section \ref{toroidal}, under the genericity
assumption. Namely, if $w$ is a zero of ${\mb Q}(z)$ which is not a zero of
$\wt{{\mb Q}}(z)$, then we get the following Bethe Ansatz equation:
$$
{\mb Q}(wq_1){\mb Q}(wq_2){\mb Q}(wq_3) + [\alpha]{\mb
  Q}(wq_1^{-1}){\mb Q}(wq_2^{-1}){\mb Q}(wq_3^{-1}) = 0.
$$
It coincides with the Bethe Ansatz equation obtained in a different
way in \cite{FJMM}.

\section{Classical KdV system}    \label{ckdv}

In the rest of this paper, we discuss the affine opers that should
encode the eigenvalues of the quantum $\ghat$-KdV Hamiltonians
according to the conjecture of \cite{FF:sol}. We start by recalling
the definition of the classical KdV systems and opers.

\subsection{Drinfeld--Sokolov reduction and opers}    \label{untw}

The phase space of the classical $\ghat$-KdV system is obtained from a
certain space of first order differential operators by Hamiltonian
reduction, which is called the Drinfeld--Sokolov reduction
\cite{DS}. We will first discuss the case of an untwisted affine
algebra $\ghat$, by which we mean the universal central extension of
the formal loop algebra $\g\ppart$:
$$
0 \longrightarrow \C {\mathbf 1} \longrightarrow \ghat \longrightarrow
\g\ppart \longrightarrow 0
$$
(we are slightly abusing notation here, because in our discussion of
the quantum affine algebras $\ghat$ stands for the Laurent polynomial
version). The commutation relations read: $[{\mathbf 1},A(t)]=0$ and
$$
[A(t),B(t)] = [A(t),B(t)] - \on{Res}_{t=0} \kappa_0(A(t),dB(t)),
$$
where $\kappa_0$ is the invariant inner product on $\g$ normalized in
the standard way, so that the square length of the maximal root is
equal to 2.

Consider the space of differential operators
\begin{equation}    \label{eta L}
\pa_t + A(t), \qquad A(t) \in \g\ppart,
\end{equation}
The inner product
$$
\langle A(t),B(t) \rangle = \on{Res}_{t=0} \kappa_0(A(t),B(t)) dt
$$
enables us to identify $\g\ppart$ with its dual space. It is known
(see, e.g., \cite{FB}, Ch. 16.4) that under this identification, the
space of differential operators \eqref{eta L} may be identified with a
hyperplane in the dual space to $\ghat$ that consists of all linear
functionals on $\ghat$ taking value 1 on the central element ${\mathbf
  1}$. The standard Kirillov--Kostant Poisson structure on the dual
space to $\ghat$ restricts to a Poisson structure on the
hyperplane. So do the coadjoint actions of the group $G\ppart$ and its
Lie algebra $\g\ppart$, and when written in terms of the operators
\eqref{eta L}, they become the gauge actions of $G\ppart$ and
$\g\ppart$, respectively.

Fix the Cartan decomposition
$$
\g = \n_+ \oplus \h \oplus \n_-,
$$
where $\n_+$ and $\n_-$ are the upper and lower nilpotent subalgebras
of $\g$, respectively, and $\h$ is the Cartan subalgebra. The above
inner product on $\g\ppart$ identifies the dual space to $\n_+\ppart$
with $\n_-\ppart$. Let $f_i, i=1,\ldots,n$, be generators of $\n_-$
corresponding to negative simple roots of $\g$. Consider the
Hamiltonian reduction of the space of the operators \eqref{eta L} with
respect to the gauge (that is, coadjoint, hence Poisson) action of the
Lie algebra $\n_+\ppart$ and its character (that is, a one-point
coadjoint orbit in $\n_-\ppart$) corresponding to the element
\begin{equation} \label{olp} \ol{p}_{-1} = \sum_{i=1}^n f_i \; \in
  \; \n_- \subset \n_-\ppart = \n_+\ppart^*.
\end{equation}
This is the {\em Drinfeld--Sokolov reduction} \cite{DS}.

The reduced phase space of the Drinfeld--Sokolov reduction is
therefore the quotient of the space $\wt{\mc M}(\ghat)$ of operators
of the form
\begin{equation}    \label{reduced}
\pa_t + \ol{p}_{-1} + {\mb v}(t), \qquad {\mb v}(t) \in \bb_+\ppart,
\end{equation}
where $\bb_+ = \h \oplus \n_+$ is the Borel subalgebra of $\g$,
under the gauge action of the loop group $N_+\ppart$.

According to \cite{DS}, the action of $N_+\ppart$ on $\wt{\mc
  M}(\g)$ is free. The resulting quotient space
$$
{\mc M}(\g) = \wt{\mc M}(\g)/N_+\ppart
$$
is called the space of $\g$-{\em opers} on the punctured disc
$D^\times = \on{Spec} \C\ppart$ (for a general curve, the space of
$\g$-opers has also been defined by Beilinson and Drinfeld
\cite{BD,BD:opers}). The Poisson algebra of local functionals on ${\mc
  M}(\g)$ is known as the {\em classical ${\mc W}$-algebra}. We denote
it by ${\mc W}(\g)$.

For example, for $\g = \sw_2$ we have
$$
\ol{p}_{-1} = f_1 = \begin{pmatrix} 0 & 0 \\
  1 & 0 \end{pmatrix},
$$
and so ${\mc M}(\sw_2)$ is the quotient of the space of operators of
the form
$$
\pa_t + \begin{pmatrix} a(t) & b(t) \\
  1 & -a(t) \end{pmatrix}, \qquad a(t),b(t) \in \C\ppart,
$$
by the upper triangular gauge transformations depending on $t$. It is
easy to see that each gauge equivalence class contains a unique
operator of the form
$$
\pa_t + \begin{pmatrix} 0 & v(t) \\
  1 & 0 \end{pmatrix}, \qquad v(t) \in \C\ppart,
$$
and hence we may identify ${\mc M}(\sw_2)$ with the space of such
operators, or, equivalently, with the space of second order
differential operators
$$
\pa_t^2 - v(t), \qquad v(t) \in \C\ppart.
$$

Likewise, the space ${\mc M}(\sw_r)$ may be identified with the
space of $n$th order differential operators
$$
\pa_t^r - v_1(t) \pa_t^{r-2} + \ldots + (-1)^r v_{r-2}(t) \pa_t - (-1)^r
v_{r-1}(t).
$$
In a similar way, for Lie algebras of types $B$ and $C$ one can
identify $\g$-opers with self-adjoint and anti-self adjoint scalar
differential operators, and for type $D$, pseudo-differential
operators of a special kind \cite{DS}. However, there is no such
uniform identification for a general Lie algebra $\g$. The best we can
do in general is to choose special representatives in the
$N_+\ppart$-gauge equivalence classes on the space of operators of the
form \eqref{reduced} in the following way.

Recall the element $\ol{p}_{-1} \in \n_-$ given by formula
\eqref{olp}. There exists a unique element of $\n_+$ of the form
$$
\ol{p}_1 = \sum_{i=1}^n c_i e_i, \qquad c_i \in \C,
$$
where $e_i, i=1,\ldots,n$, are generators of $\n_+$, such that
$\ol{p}_1, \ol{p}_{-1}$, and $\ol{p}_0 = [\ol{p}_1,\ol{p}_{-1}]$ form
an $\sw_2$ triple. The element $\frac{1}{2} \ol{p}_0 \in \h$ then
defines the principal grading on $\g$ such that $\deg \ol{p}_1 = 1,
\deg \ol{p}_{-1} = -1$. Let
$$
V_{\can} = \bigoplus_{i \in E} V_{\can,i}
$$
be the space of $\on{ad} \ol{p}_1$-invariants in $\n_+$, decomposed
according to the principal grading. Here
$$
E = \{ d_1,\ldots,d_n \}
$$
is the set of exponents of $\g$. Then $\ol{p}_1$ spans
$V_{\on{can},1}$. Choose a linear generator $\ol{p}_j$ of
$V_{\can,d_j}$ (if the multiplicity of $d_j$ is greater than one,
which happens only in the case $\ghat=D^{(1)}_{2n}, d_j=2n$, then we
choose linearly independent vectors in $V_{\on{can},d_j}$). The
following result is due to Drinfeld and Sokolov \cite{DS} (see also
\cite{BD,BD:opers}).

\begin{lem}    \label{free}
  The gauge action of $N_+\ppart$ on the space $\wt{\mc M}(\g)$ is
  free, and each gauge equivalence class contains a unique operator of
  the form $\pa_t + \ol{p}_{-1} + {\mathbf v}(t)$, where ${\mathbf
    v}(t) \in V_{\can}\ppart$, so that we can write
\begin{equation} \label{coeff fun}
{\mathbf v}(t) = \sum_{j=1}^n v_j(t) \cdot \ol{p}_j, \qquad v_j(t) \in
\C\ppart.
\end{equation}
\end{lem}

Thus, each point of the reduced phase space ${\mc M}(\g)$ of the
Drinfeld--Sokolov reduction is canonically represented by an operator
$\pa_t + \ol{p}_{-1} + {\mathbf v}(t)$, where ${\mathbf v}(t)$ is of
the form \eqref{coeff fun}.

\subsection{Spectral parameter}    \label{mon}

Now we insert the ``spectral parameter'' $z$ into our operators. This
means that we go from $\g$ to $\g\zpart$, and from $\g\ppart$ to
$\g\zpart\ppart$. Let
$$
f_0 = e_{\theta} z \; \in \g\zpart, \quad \on{where} \quad e_{\theta}
\in \n_+ \subset \g
$$
is a non-zero element in the one-dimensional weight subspace of the
nilpotent subalgebra $\n_+$ of $\g$ corresponding to the maximal root
(this is a highest weight vector in the adjoint representation of
$\g$). For instance, if $\g=\sw_r$, we can take as $e_{\theta}$ the
matrix with 1 in the upper right corner and 0 in all other
places.

Note that $f_i, i=0,\ldots,n$, are the generators of the lower
nilpotent subalgebra $\wt\n_-$ of $g\zpart$, which consists of all
elements of $\g[z]$ whose value at $z=0$ is in $\n_- \subset
\g$. Therefore,
$$
p_{-1} = \sum_{i=0}^n f_i 
$$
may be viewed as a ``principal nilpotent element'' of $\wt\n_-$.

Now we shift the operators \eqref{eta L} by $f_0 = e_\theta z$. We
then obtain the following operators:
\begin{equation}    \label{eta L1}
\pa_t + A(t) + e_\theta z, \qquad A(t) \in \g\ppart,
\end{equation}
Since $f_0 = e_{\theta} z$ is stable under the action of $N_+\ppart$,
this shift does not change the gauge action of $N_+\ppart$. Therefore,
we can identify the reduced phase space ${\mc M}(\g)$ of the
Drinfeld--Sokolov reduction with the quotient of the space $\wt{\mc
  M}(\g)$ of operators of the form
\begin{equation}    \label{reduced1}
\pa_t + p_{-1} + {\mb v}(t), \qquad {\mb v}(t) \in \bb_+\ppart,
\end{equation}
by the gauge action of the loop group $N_+\ppart$.

According to Lemma \ref{free}, each gauge equivalence class contains a
unique operator of the form
\begin{equation}    \label{untw form}
\pa_t + p_{-1} + {\mathbf v}(t), \qquad {\mathbf v}(t) \in
V_{\can}\ppart.
\end{equation}
We will denote this quotient by ${\mc M}(\ghat)$. It is isomorphic to
the space ${\mc M}(\g)$, but its elements are differential operators
with ``spectral parameter'' $z$ that we need to construct the
$\ghat$-KdV Hamiltonians.

For example, ${\mc M}(\wh\sw_2)$ is the quotient of the space of
operators of the form
$$
\pa_t + \begin{pmatrix} a(t) & b(t) + z \\
  1 & -a(t) \end{pmatrix}, \qquad a(t),b(t) \in \C\ppart,
$$
by the upper triangular gauge transformations depending on $t$ (but
not on $z$). Each gauge equivalence class contains a unique operator of
the form
$$
\pa_t + \begin{pmatrix} 0 & v(t) + z\\
  1 & 0 \end{pmatrix}, \qquad v(t) \in \C\ppart,
$$
and hence we may identify ${\mc M}(\wh\sw_2)$ with the space of such
operators, or, equivalently, with the space of second order
differential operators with spectral parameter
$$
\pa_t^2 - v(t) - z, \qquad v(t) \in \C\ppart.
$$

Likewise, the space ${\mc M}(\wh\sw_r)$ may be identified with the
space of $n$th order differential operators with spectral parameter
$$
\pa_t^r - v_1(t) \pa_t^{r-2} + \ldots + (-1)^r v_{n-2}(t) \pa_t - (-1)^r
v_{r-1}(t) - (-1)^r z.
$$

The reason why inserting the spectral parameter is important is that
after we do that we can define the $\ghat$-KdV Hamiltonians. These are
Poisson commuting functions (more properly, functionals) on the
reduced phase space ${\mc M}(\ghat) = {\mc M}(\g)$. There are two
types of $\ghat$-KdV Hamiltonians: local and non-local, and they are
both constructed using the formal monodromy matrix of the operators
\eqref{reduced} specialized to different representations of
$\g$. Because our operators now depend on $z$, the monodromy matrix
depends on $z$ as well, and this enables us to take the coefficients
of its expansion.

More precisely, let $M_V(z) \in G$ be the monodromy matrix of the
operator \eqref{eta L1} specialized to an irreducible representation
$V$ of $\g$ (see Section 3.2 of \cite{FF:sol} for the precise
definition). For any $\g$-invariant function $\varphi$ on $V$ the
corresponding function $\wt{H}_\varphi(z) = \varphi(M_V(z))$ on
$\wt{\mc M}(\ghat)$ is invariant under the gauge action of $N_+\ppart$
and hence gives rise to a well-defined function $H_\varphi(z)$ on the
quotient ${\mc M}(\ghat)$.

The asymptotic expansion of $H_\varphi(z)$ at $z=\infty$ yields the
local $\ghat$-KdV Hamiltonians, which generate the $\ghat$-KdV
hierarchy of commuting Hamiltonian flows on the Poisson manifold ${\mc
  M}(\ghat)$ (see \cite{DS}). These Hamiltonians have the form
\begin{equation}    \label{class local}
H_s = \int P_s(v_j(t),v'_j(t),\ldots) dt, \qquad s = d_i + Nh, \quad N
\in \Z_+,
\end{equation}
where $d_i \in E$ is an exponent of $\ghat$ and $h$ is the Coxeter
number. The integrand $P_s$ is a differential polynomial of degree
$s+1$, where we set $\deg v^{(m)}_j = d_j+m+1$.

On the other hand, the $z$-expansion of $H_\varphi(z)$ at $z=0$ yields
Poisson commuting non-local $\ghat$-KdV Hamiltonians.

Poisson commutativity of these Hamiltonians is easily proved from the
commutativity of the functions $H_\varphi(z)$ (see, for example,
\cite{RS,RSF,Reyman}).

\subsection{Miura transformation}

A convenient way to compute the higher order terms in the
$z$-expansion of $\varphi(M_V(z))$ is to realize the variables of the
KdV hierarchy in terms of the variables of the modified KdV (mKdV)
hierarchy. This provides a kind of ``free field realization,'' also
known as the Miura transformation, for the commuting Hamiltonians.

Consider the space $\ol{\mc M}(\ghat)$ of operators of the form
\begin{equation}    \label{u t}
\pa_t + p_{-1} + {\mb u}(t), \qquad {\mb u}(t) \in \h\ppart.
\end{equation}
The natural map $\ol{\mc M}(\ghat) \to {\mc M}(\ghat)$ given by the
composition of the inclusion $\ol{\mc M}(\ghat) \to \wt{\mc M}(\ghat)$
and the projection $\wt{\mc M}(\ghat) \to {\mc M}(\ghat)$ is called
the Miura transformation (see \cite{DS}). It is a Poisson map with
respect to the Heisenberg--Poisson structure on $\ol{\mc M}(\ghat)$
and the Drinfeld--Sokolov Poisson structure on ${\mc
  M}(\ghat)$. Therefore it gives rise to an embedding of the classical
${\mc W}$-algebra ${\mc W}(\g)$ into the Heisenberg--Poisson algebra
of functions on $\ol{\mc M}(\ghat)$.

Because the operator \eqref{u t} has such a simple structure, it is
easier to compute the monodromy matrix $M_V(z)$, and hence the
functions $\varphi(M_V(z))$, for it rather than for the operators of
the form \eqref{reduced1}. The coefficients of the asymptotic
expansion of the function $\varphi(M_V(z))$ are the local Hamiltonians
of the modified KdV (or mKdV) hierarchy associated to $\ghat$. They are
connected to the above $\ghat$-KdV Hamiltonians by the Miura
transformation. On the other hand, the coefficients in the
$z$-expansion of $\varphi(M_V(z))$ are the non-local $\ghat$-mKdV
Hamiltonians.

For example, in the case when $\g=\sw_2$ the operator \eqref{u t} has
the form
$$
\pa_t + \begin{pmatrix} u(t) & z \\ 1 & -u(t) \end{pmatrix}.
$$
The coefficients in the $z$-expansion of the trace of the monodromy of
this operator are written down explicitly in \cite{BLZ1}. They are
given by multiple integrals of $\exp(\pm 2\phi(t))$, where $\phi(t)$
is the anti-derivative of $u(t)$, that is, $u(t) = \phi'(t)$ (these
are classical screening operators, see \cite{FF:laws}).

Similar formulas may be obtained for other affine Kac--Moody algebras.

\subsection{Twisted affine algebras}    \label{twisted aff}

Finally, we consider the case of a twisted affine Kac--Moody algebra
$\ghat$. We recall that it is constructed from a finite-dimensional
simple Lie algebra $\gt$ whose Dynkin diagram has an automorphism of
order $r=2$ or $3$. Using the Cartan generators of $\gt$, we obtain
an automorphism of $\gt$ of the same order denoted by $\sigma$. The
Lie algebra $\gt$ decomposes into a direct sum of eigenspaces of
$\sigma$:
$$
\gt = \bigoplus_{\ol{i} \in \Z/r\Z} \gt_{\ol{i}},
$$
where $\gt_{\ol{0}}$ is the simple Lie algebra corresponding to the
quotient of the Dynkin diagram of $\gt$ by the action of the
automorphism. The twisted affine Kac--Moody algebra $\ghat$ is defined
as the universal central extension of the twisted loop algebra ${\mc
  L}_\sigma \g$, which is the completion of the Lie algebra
$$
\bigoplus_{n \in \Z} \gt_{\ol{n}} \otimes z^n
$$
in $\gt\zpart$.

Let $f_i, i=1,\ldots,n$, be the generators of the lower nilpotent
subalgebra of $\gt_{\ol{0}}$, and $f_0 = e_{\theta_0} z$, where
$e_{\theta_0}$ is a non-zero generator of the one-dimensional highest
weight subspace of the $\g_{\ol{0}}$-module $\g_{\ol{1}}$ (this is the
twisted affine algebra analogue of the element $e_\theta z$). We
denote its weight (from the point of view of the Cartan subalgebra of
$\gt_0$) by $\theta_0$. The element
$$
p_{-1} = \sum_{i=0}^n f_i
$$
is then the ``principal nilpotent element'' of the twisted affine
Kac--Moody algebra $\ghat$. Therefore the analogues of the operators
\eqref{reduced1} are the operators of the form
\begin{equation}    \label{reduced tw}
\pa_t + p_{-1} + {\mb v}(t),
  \qquad {\mb v}(t) \in \bb_{\ol{0},+}\ppart,
\end{equation}
where $\bb_{\ol{0},+} = \bb'_+ \cap \gt_{\ol{0}}$ is the Borel
subalgebra of $\gt_{\ol{0}}$. We denote the space of operators
\eqref{reduced tw} by $\wt{\mc M}(\ghat)$.

Next, we take the quotient of the space $\wt{\mc M}(\ghat)$ by the
gauge action of the loop group $N_{\ol{0},+}\ppart$, where
$N_{\ol{0},+}$ is the unipotent Lie group corresponding to the Lie
algebra $\n_{\ol{0},+} = \n'_+ \cap \gt_{\ol{0}}$.

Since $f_0$ is invariant under the action of $N_{\ol{0},+}\ppart$, we
can remove $f_0$ from \eqref{reduced tw} (in the same way as in the
untwisted case); that is, we can replace $p_{-1}$ by the
element $\ol{p}_{-1} = \sum_{i=1}^n f_i$ of $\gt_{\ol{0}}$. The
resulting space consists of the operators of the form
\begin{equation}    \label{reduced tw1}
\pa_t + \ol{p}_{-1} + {\mb v}(t),
  \qquad {\mb v}(t) \in \bb_{\ol{0},+}\ppart
\end{equation}
and hence coincides with the space $\wt{\mc M}(\gt_{\ol{0}})$ that we
considered in Section \ref{untw} when we discussed the untwisted
case. Therefore the quotient of $\wt{\mc M}(\ghat)$ by the gauge
action of $N_{\ol{0},+}\ppart$ is nothing but the reduced space ${\mc
  M}(\gt_{\ol{0}})$ arising in the Drinfeld--Sokolov reduction of the
untwisted affine algebra $\wh\gt_{\ol{0}}$ which is the central
extension of the loop algebra $\gt_{\ol{0}}\zpart$. The reduced phase
space is therefore the same as the one which we get in the case of the
untwisted affine algebra $\wh\gt_{\ol{0}}$, and so the corresponding
Poisson algebra is nothing but the classical ${\mc W}$-algebra ${\mc
  W}(\gt_{\ol{0}})$.

However, we now insert the spectral parameter differently, by
shifting our operators by $e_{\theta_0} z$ (rather than $e_\theta
z$). Then, using Lemma \ref{free}, we can realize the space ${\mc
  M}(\gt_{\ol{0}})$ as the space of differential operators of the form
\begin{equation} \label{tw form} \pa_t + \ol{p}_{-1} + e_{\theta_0} z
  + {\mathbf v}(t), \qquad {\mathbf v}(t) \in V_{\can}\ppart.
\end{equation}
That is to say, we ``insert'' the element $f_0 = e_{\theta_0} z \in
\gt_{\ol{1}} z$ of the twisted loop algebra ${\mc L}_\sigma \gt$
rather than the element $e_\theta z \in \gt_{\ol{0}} z$ of the
untwisted loop algebra $\g_{\ol{0}}\zpart$ (see formula \eqref{untw
  form}).

We then construct the invariants of the monodromy matrices for
these operators in the same way as in the untwisted case. Note,
however, that these monodromy matrices are different from the
monodromy matrices for the operators \eqref{untw form} corresponding
to the untwisted affine algebra $\wh{\gt_{\ol{0}}}$, even though both
may be viewed as functionals on the same reduced phase space ${\mc
  M}(\gt_{\ol{0}})$. The expansions of the invariants of the monodromy
matrices at $z = \infty$ and $z=0$ give rise to the local and
non-local classical Hamiltonians, respectively, for the KdV system
corresponding to the twisted affine algebra $\ghat$ we started
with. These Hamiltonians are different from the those defined for the
operators \eqref{untw form} corresponding to the untwisted affine
algebra $\wh{\gt_{\ol{0}}}$, even though in both cases they are
functionals on the same reduced phase space ${\mc M}(\gt_{\ol{0}})$.

\section{Quantum KdV system}    \label{qkdv}

\subsection{Local Hamiltonians}

KdV Hamiltonians can be quantized. First, let's consider the problem
of quantization of the classical local $\ghat$-KdV Hamiltonians. They
are Poisson commuting elements of the classical ${\mc W}$-algebra
${\mc W}(\g)$; namely, the Poisson algebra of local functionals on the
space ${\mc M}(\ghat)$ discussed in the previous section. It has been
shown in \cite{FF:laws} that ${\mc W}(\g)$ can be quantized; that is,
there exists a one-parameter associative algebra ${\mc W}_\beta(\g)$
whose limit as $\beta \to 0$ is a commutative algebra with a natural
Poisson structure that is isomorphic to ${\mc W}(\g)$. For example,
${\mc W}_\beta(\sw_2)$ is a completed enveloping algebra of the
Virasoro algebra.

More precisely, in \cite{FF:laws} it was shown that ${\mc
  W}_\beta(\g)$ may be defined as a subalgebra in a Heisenberg
algebra. Here by ``Heisenberg algebra'' we mean a completion of the
universal enveloping algebra of the Heisenberg Lie subalgebra $\hhat$
of $\ghat$ (central extension of the formal loop algebra $\h\ppart$)
in which the central element is identified with the identity. The
${\mc W}$-algebra ${\mc W}_\beta(\g)$ is defined for generic $\beta$
(i.e., such that $\beta^2$ is not a rational number) as the
intersection of the kernels of the so-called screening operators,
which depend on $\beta$, associated to the simple roots of $\g$ (see
\cite[Section 4.6]{FF:laws} or \cite[Section 15.4.11]{FB}). The
resulting algebra is then extended to all complex values of
$\beta$. This implies, in particular, that each Fock representation
$\pi_\mu$ of $\hhat$ (where $\mu \in \h^*$ is the highest weight) is
naturally a module over ${\mc W}_\beta(\g)$. In fact, ${\mc
  W}_\beta(\g)$ is first defined as a vertex subalgebra of the
Heisenberg vertex algebra $\pi_0$, and then as an associative algebra
corresponding to this vertex algebra (see \cite{FF:laws} and
\cite[Chapter 15]{FB}).

In the limit $\beta \to 0$ we obtain an embedding of ${\mc W}(\g)$
(viewed as a Heiseberg--Poisson algebra) into the Heisenberg--Poisson
algebra of functions on the space $\ol{\mc M}(\ghat)$ discussed in the
previous section. This embedding is induced by the Miura
transformation $\ol{\mc M}(\ghat) \to {\mc M}(\ghat)$ discussed in the
previous section. Thus, the embedding of the quantum ${\mc W}$-algebra
${\mc W}_\beta(\g)$ into the Heisenberg algebra may be viewed as a
quantization of the Miura transformation.

Alternatively, the quantum ${\mc W}$-algebra ${\mc W}_\beta(\g)$ may
be defined via the quantum Drinfeld--Sokolov reduction, see
\cite{FF:ds,FF:gd} and \cite[Chapter 15]{FB} (the equivalence between
the two constructions of ${\mc W}_\beta(\g)$ is discussed in
\cite[Chapter 15.4]{FB}).

Now, by a quantization of local $\ghat$-KdV Hamiltonians we understand
a commutative subalgebra of ${\mc W}_\beta(\g)$ whose limit as $\beta
\to 0$ is the Poisson commutative subalgebra of ${\mc W}(\g)$
generated by the classical local KdV Hamilnonians \eqref{class
  local}. (If $\ghat$ is a twisted affine algebra, then by $\g$ we
mean here the Lie algebra $\gt_{\ol{0}}$, as in the previous section.)
The existence of this quantization is a non-trivial statement, which
has been proved in \cite{FF:laws}. Local quantum $\ghat$-KdV
Hamiltonians are elements $\wh{H}_s$ of this commutative subalgebra
which are quantizations of the $H_s$ given by formula \eqref{class
  local} in the sense that $\wh{H}_s$ tends to $H_s$ when $\beta \to
0$, if we rescale the generators of ${\mc W}_\beta(\g)$ in the
appropriate way (see \cite{FF:laws} for details).

If we apply the quantum Miura transformation to the local quantum
$\ghat$-KdV Hamiltonians $\wh{H}_s$, we obtain the corresponding local
quantum $\ghat$-mKdV Hamiltonians. Those may also be viewed as quantum
integrals of motion of the affine Toda field theory associated to
$\ghat$. The latter is equivalent to saying that they commute with the
screening operators associated to the simple roots of $\ghat$ (see
\cite{FF:laws} for details).

The commutative algebra of local quantum $\ghat$-KdV Hamiltonians acts
on any module over ${\mc W}_\beta(\g)$ on which the eigenvalues of the
Virasoro operator $L_0$ are ``bounded from below'' (that is, the set
of these eigenvalues is the union of the sets of the form $\gamma +
\Z_+, \gamma \in \C$). We will refer to such modules as highest weight
modules. In particular, we can consider their action on the Fock
representations $\pi_\mu, \mu \in \h^*$.

After passing to the periodic coordinate $\varphi$ such that
$z=e^{i\varphi}$, we obtain a commuting algebra of quantum
Hamiltonians which includes the Virasoro operator $L_0$. Thus, all
local quantum Hamiltonians are homogeneous of degree $0$ with respect
to the grading defined by $L_0$, and so they preserve the homogeneous
components of highest weight modules. In the case of Fock
representations $\pi_\mu$ these components are finite-dimensional. It
is natural to ask what are the spectra of the local quantum
$\ghat$-KdV Hamiltonians on these components. This problem naturally
arises in the study of deformations of conformal field theories with
${\mc W}$-algebra symmetry \cite{Zam,EY,KM,FF:laws}. However, it
proved to be elusive for general $\beta$, because we do not have much
structure on the commutative algebra generated by the local quantum
KdV Hamiltonians.

\subsection{Non-local Hamiltonians}    \label{nlham}

A breakthrough in the study of the quantum KdV system was made by
Bazhanov, Lukyanov, and Zamolodchikov in a series of papers starting
with \cite{BLZ1}, in which a procedure for quantization of the {\em
  non-local} classical KdV Hamiltonians was developed. We recall from
the previous section that those can be obtained from the traces of the
monodromy matrix of the first order differential operators \eqref{u t}
(taken in the Miura form) with spectral parameter $z$, expanded as a
power series near $z=0$.

The quantum non-local $\ghat$-KdV Hamiltonians were introduced in
\cite{BLZ1} in the case of $\ghat=\wh\sw_2$ and in \cite{BHK} in the
case of $\g=\wh\sw_3$. As noted in \cite{BHK}, the latter construction
generalizes in a straightforward way to other affine algebras $\ghat$.

These non-local Hamiltonians are defined as elements of the completed
Heisenberg algebra acting on the Fock representations $\pi_\mu$ for
$\mu \in \h^*$. More precisely, according to the construction of
\cite{BLZ1,BHK}, for each finite-dimensional representation $V$ of
$U_q(\ghat)$, one defines a generating series of commuting non-local
Hamiltonians $T_{V,n}$ by the formula
\begin{equation}    \label{tv}
T_V(z) = \sum_{n\geq 0} T_{V,n} z^n = \on{Tr}_{V(z)}(e^{\pi i {\mb P}
  \cdot h} \; {\mc L}).
\end{equation}
Here ${\mc L}$ is the operator obtained from the reduced universal
$R$-matrix of $U_q(\ghat)$, which is an element of a completion of the
tensor product $U_q(\wh\n_+) \otimes U_q(\wh\n_-)$, by mapping the first
factor to $\on{End} V(z)$ and the second factor to $\on{End} \oplus
\pi_\mu$ using the screening operators corresponding to the simple
roots of $\ghat$ (they satisfy the Serre relations of $U_q(\wh\n_-)$),
and
$$
{\mb P} \cdot h = \sum_{j=1}^n {\mb P}^j h_j.
$$
Here $\{ h_j \}$ is a basis in $\h$ and $\{ {\mb P}^j \}$ is the dual
basis in $\h$ with respect to the inner product on $\h$ obtained by
restricting the normalized invariant inner product $\kappa_0$ on $\g$
(see Section \ref{untw}). The elements $h_j$ are defined so that $K_j
= q^{h_j}$ are the standard Drinfeld-Jimbo generators of $U_q(\g)
\subset U_q(\ghat)$, where
$$
q = e^{\pi i \beta^2}
$$
The $h_j$'s act on $V$. On the other hand, the elements ${\mb P}^j$
are elements of the constant Cartan subalgebra $\h$ in the Heisenberg
algebra $\hhat$. They act on the Fock representation $\pi_\mu$
according to the formula ${\mb P}^j \mapsto \mu({\mb P}^j) \on{Id}$.

Note that this $T_V(z)$ may be viewed as a quantum analogue of the
monodromy matrix $M_V(z)$, see Section \ref{mon}.

The claim of \cite{BLZ1,BHK} is that the formal power series $T_V(z)$
commute with each other:
$$
[T_V(z),T_W(w)] = 0
$$
for any finite-dimensional representations $V$ and $W$ of
$U_q(\ghat)$. They also commute with the local quantum $\ghat$-KdV
Hamiltonians. Technically, this is proved in \cite{BLZ1,BHK} for
$\g=\sw_2$ and $\sw_3$, but the proof generalizes to other simple Lie
algebras. Alternatively, these statements can be proved \cite{FFS}
using the methods of \cite{FF:laws} (see Sects. 3.3 and 5.3 of
\cite{FF:sol} for an outline).

Furthermore, it follows from the construction that
$$
T_{V \oplus W}(z) = T_V(z) + T_W(z), \qquad T_{V \otimes W}(z) =
T_V(z) T_W(z).
$$
Thus, we obtain an action of the commutative algebra $K_0({\mc C})$,
where ${\mc C}$ is the category of finite-dimensional representations
of $U_q(\ghat)$ (see Section \ref{fdrep}), by endomorphisms of the
Fock representation $\pi_\mu$ for any $\mu \in \h^*$.

All of the non-local Hamiltonians $T_{V,n}$ are homogeneous
endomorphisms of $\pi_\mu$ of degree 0 with respect to the grading by
the Virasoro operator $L_0$ (which is in fact the simplest local
quantum KdV Hamiltonian). Therefore we obtain an action of $K_0({\mc
  C})$ on each finite-dimensional graded component of $\pi_\mu$. This
action of non-local quantum $\ghat$-KdV Hamiltonians commutes with the
action of the local quantum $\ghat$-KdV Hamiltonians, and so these
Hamiltonians have common eigenvectors. Furthermore, according to
\cite{BLZ1,BHK,BLZ}, the eigenvalues of the local Hamiltonians can be
recovered from the eigenvalues of the non-local ones as coefficients
of their asymptotic expansions. Therefore it makes sense to replace
the question of describing the spectra of the local $\ghat$-KdV
Hamiltonians with the question of describing the joint eigenvalues of
the non-local Hamiltonians. We now have a lot of additional structure
because, by construction, these Hamiltonians correspond to elements of
the algebra $K_0({\mc C})$ and therefore they must satisfy the
relations in this algebra, such as the $T$-system satisfied by the
classes of the KR modules (see \cite{HCrelle, IIKNS} and references
therein).

Bazhanov, Lukyanov, and Zamolodchikov then made an additional step of
generalizing this construction to infinite-dimensional representations
of $U_q(\wh\bo_+)$. Indeed, formula \eqref{tv} only requires $V$ to be
a representation of $U_q(\wh\bo_+)$, and hence we can take as $V$ a
representation from the category ${\mc O}$ (see Section
\ref{catO}). For instance, we can take $V=L^+_{i,a}$. Then, in the
case $\g=\sw_2$, if we denote $T_V(z)$ by $Q(z)$, we obtain the Baxter
relation linking $Q(z)$ and $T_W(z)$, where $W$ is the two-dimensional
fundamental representation. In fact, it is in this context that the
representations $L^+_{i,a}$ were discovered in \cite{BLZ2,BLZ3} and
\cite{BHK} for $\g=\sw_2$ and $\sw_3$ (in \cite{Ko} this construction
was generalized to the case of $\sw_r$).

However, if $V$ is infinite-dimensional, then we need to address the
convergence of the trace over $V$. {\em A priori}, it is not clear
that the trace over $V$, and hence $T_V(z)$ given by formula
\eqref{tv}, is well-defined. One way to approach this question is to
use the term $e^{\pi i {\mb P} \cdot h}$ in formula \eqref{tv}, which
acts on $\pi_\mu$ as $\exp(\pi i \sum_j \mu({\mb P}^j) h_j)$. If we
denote $e^{\pi i \mu({\mb P}^j)}$ by $u_j$, we can express the trace
as a power series in the $u_j$. Then the trace will make sense as a
formal power series in the $u_j, j=1,\ldots,n$. This is similar to the
approach taken in \cite{FH}. However, in applications to quantum field
theories (such as conformal field theories and their deformations),
one needs to consider the $u_j$ as specific numbers. Then the
convergence of the trace could become problematic. Nevertheless, it is
natural to conjecture that the series will converge for generic values
of $u_j$ -- that is, for generic $\mu \in \h^*$. We formulate this as
the following

\begin{conj}    \label{generic}
  For generic $\mu \in \h^*$, there is an action of the commutative
  algebra $K_0({\mc O})$ on the Fock representation $\pi_\mu$ of
  $\hhat$ which commutes with the local quantum $\ghat$-KdV
  Hamiltonians (including the operator $L_0$).
\end{conj}

\subsection{Connection to the $Q\wt{Q}$-system}

Conjecture \ref{generic} means that for any relation in $K_0({\mc
  O})$, the joint eigenvalues of the corresponding non-local
$\ghat$-KdV Hamiltonians will satisfy this relation for any joint
eigenvector in $\pi_\mu$ for generic $\mu \in \h^*$. In particular,
each joint eigenvector of the non-local $\ghat$-KdV Hamiltonians
should give rise to a solution of the $Q\wt{Q}$-system \eqref{Qsyst}
with $q=e^{\pi i \beta^2}$ and the corresponding Bethe Ansatz
equations \eqref{BAE} (provided that the genericity assumption of
Section \ref{bethe} holds).

In the case $\g=\sw_2$, the $Q\wt{Q}$-system reduces to the quantum
Wronskian relation of \cite{BLZ2,BLZ3,BLZ}, and in the case $\g=\sw_3$
the $Q\wt{Q}$-system reduces to relations considered in \cite{BHK}
(see Section \ref{firstex} for more details). However, for general
simple Lie algebras the $Q\wt{Q}$-system has not previously been
considered as a relation on the spectra of the non-local $\ghat$-KdV
Hamiltonians. Using the $Q\wt{Q}$-system in the study of non-local
$\ghat$-KdV Hamiltonians has the important advantage that we can use
the results of \cite{MRV1,MRV2}, where it was shown that solutions of
the same $Q\wt{Q}$-system naturally arises from the $^L\ghat$-opers
that were proposed in \cite{FF:sol} as the parameters for the spectra
of the non-local $\ghat$-KdV Hamiltonians. We will discuss these
affine opers in detail in the next section.

\subsection{Langlands duality of the spectra of quantum KdV
  Hamiltonians}    \label{duality}

We recall the Langlands duality of quantum ${\mc W}$-algebras
established in \cite{FF:gd} (see \cite[Section 4.8.1]{FF:laws} or
\cite[Section 15.4.15]{FB} for an exposition):
$$
{\mc W}_\beta(\g) \simeq {\mc W}_{\check\beta}({}^L\g),
$$
where $^L\g$ is the Langlands dual Lie algebra to $\g$ and
\begin{equation}    \label{betacheck}
\check\beta = -(\check{r})^{1/2}/\beta,
\end{equation}
$\check{r}$ being the maximal number of edges connecting the vertices
of the Dynkin diagram of $\g$ (so $\check{r}=1$ for simply-laced $\g$;
$\check{r}=2$ for $\g$ of types $B_n, C_n$, and $F_4$; and
$\check{r}=3$ for $\g=G_2$). The proof follows from the fact that for
generic $\beta$ the kernels of the screening operators corresponding
to simple roots of $\g$ and parameter $\beta$ are equal to the kernels
of the screening operators corresponding to the simple roots of $^L\g$
and parameter $\check\beta$. This implies the above isomorphism for
all $\beta$.

Since the local quantum $\ghat$-KdV Hamiltonians are defined as
elements in the intersection of the kernels of the screening operators
corresponding to the simple roots of $\g$ and parameter $\beta$, by
using the same argument, we identify for generic $\beta$ the
commutative algebra of local quantum $\ghat$-KdV Hamiltonians and the
commutative algebra of local quantum $^L\ghat$-KdV Hamiltonians (see
the end of \cite[Section 4.8.1]{FF:laws}).

Furthermore, recall from Section \ref{nlham} that the non-local
quantum $\ghat$-KdV Hamiltonians are defined \cite{BLZ,BHK} using the
the screening operators corresponding to the simple roots of $\ghat$
and parameter $\beta$. Since the latter essentially commute with the
screening operators corresponding to the simple roots of $^L\ghat$ and
parameter $\check\beta$, it is natural to expect that the non-local
quantum $\ghat$-KdV Hamiltonians commute with the non-local quantum
$^L\ghat$-KdV Hamiltonians (we use the normalized inner product on
$\h$ to identify $\h^*$ with $(^L\h)^* = \h$, so that the $\ghat$- and
$^L\ghat$-KdV Hamiltonians act on the same Fock representations). This
has been stated in \cite{BLZ2} in the case $\ghat=\wh\sw_2$ (see
formula (2.26)). We formulate this as a conjecture in general. (Note
that since the local quantum $\ghat$- and $^L\ghat$-KdV Hamiltonians
coincide, they automatically commute with the non-local quantum
$\ghat$- and $^L\ghat$-KdV Hamiltonians.)

\begin{conj}    \label{commute}
  The action of the non-local quantum $\ghat$-KdV Hamiltonians on a
  Fock representation $\pi_\mu, \mu \in \h^*$, commutes with the
  action of the non-local quantum $^L\ghat$-KdV Hamiltonians.
\end{conj}

This implies that using joint eigenvectors of the non-local quantum
$\ghat$- and $^L\ghat$-KdV Hamiltonians in $\pi_\mu, \mu \in \h^*$, we
obtain a surprising correspondence between solutions of the
$Q\wt{Q}$-systems (as well as other equations stemming from $K_0({\mc
  O})$ such as the $QQ^*$-system of \cite{HL}, see Section
\ref{QQstar}) for $U_q(\ghat)$ and $U_{\check{q}}({}^L\ghat)$, where
\begin{equation}    \label{qcheck}
q = e^{\pi i \beta^2}, \qquad
  \check{q} = e^{\pi i \check\beta^2} = e^{\pi i \check{r}/\beta^2}.
\end{equation}
This correspondence (or duality) deserves further study.

\section{Spectra of the quantum KdV Hamiltonians}    \label{kdv
  quantum}

In \cite{FF:sol}, it was conjectured that the spectra of the quantum
$\ghat$-KdV Hamiltonians can be paramet\-rized by $^L\ghat$-affine
opers of special kind. This conjecture was motivated by the results of
\cite{BLZ} in the case of $\su$ and an analogy between the quantum
$\ghat$-KdV system and the Gaudin model associated to a simple Lie
algebra $\g$, in which case the joint eigenvalues of the commuting
Hamiltonians are known to be parametrized by $^L \g$-opers. We refer
the reader to \cite{FF:sol} (especially, Sections 4.4, 5.4, and 5.5)
for the explanation of this analogy.

The general definition of a $\ghat$-affine oper was given in Section
4.1 of \cite{FF:sol}, inspired by the work of Drinfeld and Sokolov
\cite{DS} and Beilinson and Drinfeld \cite{BD,BD:opers}. Here we
review the $\ghat$-opers that are related to the joint eigenvalues of
the $\ghat$-KdV Hamiltonians according to the conjecture of
\cite{FF:sol}.

\subsection{The case of $\su$}    \label{sl2}

The $\su$-opers that appear here are gauge equivalence classes of the
first order differential operators on $\pone$, equipped with a
coordinate $z$, with values in the Lie algebra $\C {\mb d} \ltimes
\sw_2 \pparl$, where ${\mb d} = \la \pa_\la$, of the form
$$
\pa_z + \begin{pmatrix} a(z) & b(z) + \la \\
  1 & -a(z) \end{pmatrix} + \frac{k}{z} {\mb d}, \qquad k \in \C,
$$
(where $a(z)$ and $b(z)$ are rational functions in $z$) under the
action of the group $N_+$-valued rational functions in $z$, where
$N_+$ is the upper unipotent subgroup of $SL_2$.

Each gauge equivalence class contains a unique operator of the form
\begin{equation}    \label{un op}
\pa_z + \begin{pmatrix} 0 & v(z) + \la \\
  1 & 0 \end{pmatrix} + \frac{k}{z} {\mb d},
\end{equation}
and the operators relevant to the spectra of the quantum KdV
Hamiltonians are the ones with
\begin{equation}    \label{kdv form1}
v(z) = \frac{r(r+1)}{z^2} + \frac{1}{z} \left( 1 - \sum_{j=1}^m
\frac{k}{w_j} \right) + \sum_{j=1}^{m} \frac{2}{(z-w_j)^2} +
\sum_{j=1}^{m} \frac{k}{w_j} \frac{1}{z-w_j},
\end{equation}
such that the coefficients $v_{j,k}$ in the expansion of $v(z)$ in
$z-w_j$ satisfy the equations
\begin{equation}    \label{third order again}
\frac{1}{4} \left( \frac{k}{w_j} \right)^3 - \frac{k}{w_j} v_{j,0} +
v_{j,1} = 0, \qquad j=1,\ldots,m.
\end{equation}
As explained in Section 4.4 of \cite{FF:sol}, this is the condition that
the solutions of the differential equation corresponding to the
operator \eqref{un op} have no monodromy around the singular points
$w_j$.

As in \cite{FF:sol}, applying gauge transformation by $z^{k{\mb d}}$,
we can eliminate the term $\dfrac{k}{z} {\mb d}$ from \eqref{un op}
obtaining the operator
\begin{equation}
\pa_z + \begin{pmatrix} 0 & v(z) + \la z^k \\
  1 & 0 \end{pmatrix},
\end{equation}
which is equivalent to the following second order differential
operator with spectral parameter:
\begin{equation}    \label{final form kdv}
\pa_z^2 - \frac{1}{z} \left( 1 - \sum_{j=1}^m \frac{k}{w_j} \right) -
\frac{r(r+1)}{z^2} - \sum_{j=1}^{m} \frac{2}{(z-w_j)^2} -
\sum_{j=1}^{m} \frac{k}{w_j} \frac{1}{z-w_j} - \la z^k.
\end{equation}
Again, the equations \eqref{third order again} are equivalent to the
condition that this operator has no monodromy around $w_j,
j=1,\ldots,m$, and therefore no monodromy on $\pone$, except around
the points $0$ and $\infty$, for all values of $\la$. This operator
also has regular singularity at $z=0$ and the mildest possible
irregular singularity at $z=\infty$ (indeed, the restriction of
\eqref{final form kdv} to the punctured disc at the point $z=\infty$
has the form $\pa_s^2 - \wt{v}(s)$, where $\wt{v}(s) = 1/s^3 + \ldots$
with respect to the local coordinate $s=z^{-1}$ at $\infty$).

According to the proposal of \cite{FF:sol}, for generic $r$ and $k$
the differential operators \eqref{final form kdv} should encode the
common eigenvalues of the quantum KdV Hamiltonians on the irreducible
highest weight module over the Virasoro algebra with the central
charge
\begin{equation}    \label{ck}
c_k = 1 - \frac{6(k+1)^2}{k+2}
\end{equation}
and highest weight (with respect to the operator $L_0$)
\begin{equation}    \label{delta}
\Delta_{r,k} = \frac{(2r+1)^2-(k+1)^2}{4(k+2)}
\end{equation}
This ``numerology'' is explained as follows: the Virasoro algebra can
be obtained by the Drinfeld--Sokolov reduction (with respect to
$\n_-\ppart$) from the affine Kac--Moody algebra of level $k$ has
central charge $c_k$ (see \cite{FF:lmp}). Further, the
Drinfeld--Sokolov reduction of the irreducible $\su$-module with
highest weight $\la = 2r$ (``spin'' $r$) and level $k$ is the
irreducible module over the Virasoro algebra with the highest weight
$\Delta_{r,k}$ (provided that $k-\la \not\in \Z_+$); see
\cite{FF:lmp}. The number $m$ of poles of the $\su$-oper on $\pone \bs
\{ 0,\infty \}$ should be equal to the $L_0$-degree of the
corresponding eigenvector. By that we mean that it should occur in the
subspace in the irreducible module with highest weight $\Delta_{r,k}$
on which $L_0$ acts by of $\Delta_{r,k}+m$.

Note that in general, some of the poles $w_j$ may coalesce, see
Section 5.5 of \cite{FF:sol}.

\subsection{Change of variables}

The advantage of formula \eqref{final form kdv} is that it is clearly
linked, via a gauge transformation, to an affine oper \eqref{kdv
  form1}. However, the disadvantage of \eqref{final form kdv} is that
the spectral parameter $\lambda$ appears not by itself but multiplied
with $z^k$. To fix that, we make a change of variables $z \mapsto x$,
where
$$
z = \frac{x^{2\al+2}}{(2\al+2)^2}
$$
and
$$
\al = -\frac{k+1}{k+2}
$$
(we assume that $k \neq -2$). We note that this $\al$ is related to
the parameter $\beta$ discussed in Section \ref{qkdv} by the formula
\begin{equation}    \label{albeta}
\al+1 = \frac{1}{\beta^2}.
\end{equation}

Recall (see, e.g., \cite{FB}, Ch. 8.2) that the general transformation
formula for a second order operator (also known as a projective
connection) of the form
$$
\pa_z^2 - v(z): \Omega^{-1/2} \to \Omega^{3/2}
$$
(we need to consider our second order operators as acting from
$\Omega^{-1/2}$ to $\Omega^{3/2}$ to ensure that their property of
having the principal symbol $1$ and subprincipal symbol $0$ is
coordinate-independent) under the change of variables $z=\varphi(x)$
is
\begin{equation}    \label{proj conn}
v(z) \mapsto v(\varphi(x)) \left( \varphi'
\right)^2 - \frac{1}{2} \{ \varphi,x \},
\end{equation}
where
\begin{equation}    \label{schwarzian}
\{\varphi,x\} = \frac{\varphi'''}{\varphi'} - \frac{3}{2}\left(
\frac{\varphi''}{\varphi'} \right)^2
\end{equation}
is the Schwarzian derivative of $\varphi$.

Applying the change of variables $z \mapsto x$ to the operator
\eqref{final form kdv}, we obtain the operator
\begin{equation}    \label{new operator}
\pa_x^2 - \frac{\ell(\ell+1)}{x^2} -
x^{2\al} + 2 \frac{d^2}{dx^2} \sum_{j=1}^m \log(x^{2\al+2}
- z_j) + E,
\end{equation}
where
$$
\ell(\ell+1) = 4(\al+1)^2 r(r+1) + \al^2 + 2\al +
\frac{3}{4} = 4(\al+1)\Delta_{r,k} + \al^2 - \frac{1}{4},
$$
so that we have
$$
\Delta_{r,k} = \Delta(\ell,\al) = \frac{(2\ell+1)^2 -
  4\al^2}{16(\al+1)},
$$
and
\begin{align*}
z_j &= (2\al+2)^2 w_j, \\
E &= -(2\al+2)^{\frac{2\al}{\al+1}} \la.
\end{align*}
The operator \eqref{new operator} has the spectral parameter $E$
(which is obtained by rescaling $\la$) without any additional factors,
so that \eqref{new operator} looks like a typical Schr\"odinger
operator with a spectral parameter (the price to pay for this is that
this operator is multivalued with respect to the coordinate $x$).

As shown in \cite{FF:sol}, the operators \eqref{new operator} coincide
with the Schr\"odinger operators in \cite{BLZ} (formula
(1))\footnote{Note that what we denote by $\ell$ here coincides with
  $\ell$ of \cite{BLZ} but was denoted by $\wt{\ell}$ in
  \cite{FF:sol}, whereas what we denote by $r$ here was denoted by
  $\ell$ in \cite{FF:sol}.} parametrizing the spectra of the quantum
KdV Hamiltonians on the irreducible module over the Virasoro algebra
with highest weight $\Delta(\al,\ell)$ and central charge
$$
c(\al) = c_k = 1 - 6 \frac{\al^2}{\al+1}
$$
Moreover, as explained above, our condition \eqref{third order again}
means that the operator \eqref{new operator} has no monodromy around
the points $z_j$ for all $E$. This condition is equivalent to the
algebraic equations given by formula (3) in \cite{BLZ}.

In particular, the oper associated to the highest weight vector
(corresponding to $m=0$) is given by the formula
$$
\pa_x^2 - \frac{\ell(\ell+1)}{x^2} -
x^{2\al} + E,
$$
or in the matrix form
$$
\pa_x + \begin{pmatrix} \dfrac{\ell}{x} &  x^{2\al} - E \\
  1 & -\dfrac{\ell}{x} \end{pmatrix}.
$$

We note that for generic $\ell$ and $k$ the above irreducible module
over the Virasoro algebra is isomorphic to the Verma module and the
Fock representation $\pi_\mu$ with $\mu = (2\ell+1)(\al+1)/4$.

\subsection{The case of $\wh{\sw}_r$}    \label{slr}

Next, we consider the case of $\ghat=\wh{\sw}_r$. In this case the
Langlands dual Lie algebra $^L\ghat$ is also $\wh\sw_r$. We consider
$\wh\sw_r$-opers on $\pone$ which are the gauge equivalence classes of
differential operators with values in $\C {\mb d} \ltimes \sw_r
\pparl$, where ${\mb d} = \la \pa_\la$, of the form
$$
\partial_z + \left( \begin{array}{ccccc}
* & * & * &\cdots& * + \la \\
1&*&*&\cdots&*\\
0&1&*&\cdots&*\\
\vdots&\ddots&\ddots&\cdots&\vdots\\
0&0&\cdots&1&*
\end{array}\right) + \frac{k}{z} {\mb d}
$$
(where each $*$ stands for a rational function on $\pone$) under the
action of the group of $N_+$-valued rational functions on $\pone$
where $N_+$ is the upper unipotent subgroup of $SL_r$.

Each gauge equivalence class contains a unique operator of the form
\begin{equation}    \label{sln-oper1}
\partial_z + \left( \begin{array}{ccccc}
0&v_1(z)&v_2(z)&\cdots&v_{r-1}(z) + \la\\
1&0&0&\cdots&0\\
0&1&0&\cdots&0\\
\vdots&\ddots&\ddots&\cdots&\vdots\\
0&0&\cdots&1&0
\end{array}\right) + \frac{k}{z} {\mb d}.
\end{equation}
After the gauge transformation by $z^{k{\mb d}}$, we can eliminate the last
term at the cost of multiplying $\la$ by $z^k$:
\begin{equation}    \label{sln-oper2}
\partial_z + \left( \begin{array}{ccccc}
0&v_1(z)&v_2(z)&\cdots&v_{r-1}(z) + z^k \la\\
1&0&0&\cdots&0\\
0&1&0&\cdots&0\\
\vdots&\ddots&\ddots&\cdots&\vdots\\
0&0&\cdots&1&0
\end{array}\right) .
\end{equation}
The last operator is equivalent to the $r$th order scalar differential
operator with spectral parameter:
\begin{equation}    \label{nth order oper}
\pa_z^r - v_1(z) \pa_z^{r-2} + \ldots + (-1)^r v_{r-2}(z) \pa_z - (-1)^r
v_{r-1}(z) - (-1)^r z^k \la,
\end{equation}
acting from $\Omega^{-(r-1)/2}$ to $\Omega^{(r+1)/2}$. This determines
the transformation properties of these operators under the changes of
coordinate $z$; in particular, this ensures that the property that
their principal symbol is $1$ and the subprincipal symbol is $0$ is
preserved by the changes of coordinate.

As conjectured in \cite{FF:sol}, the operators \eqref{nth order oper}
that are relevant to the spectra of the quantum $\wh\sw_r$-KdV
Hamiltonians are those in which the functions $v_i(z)$ is a rational
function in $z$ with poles at $z=0, \infty$, and finitely many other
points $w_j, j=1,\ldots,m$. They satisfy the following properties:

\medskip

\noindent (1) At $z=\infty$, the operator \eqref{nth order oper}
has the mildest possible irregular singularity; namely, we have
\begin{align*}
\wt{v}_i(s) &\sim \frac{\wt{c}_i}{s^{i+1}} + \ldots, \qquad
r=1,\ldots,r-2; \\
\wt{v}_{r-1}(s) &\sim \frac{(-1)^r}{s^{r+1}} + \ldots,
\end{align*}
where $\wt{v}_i(s)$ are the coefficients of the operator obtained from
\eqref{nth order oper} by the change of variables $z \mapsto s =
z^{-1}$. By rescaling the coordinate $z$ we can make the leading
coefficient of $\wt{v}_{r-1}(s)$ to be equal to any non-zero number;
we choose that number to be $(-1)^r$.

\smallskip

\noindent (2) At $z=0$, the operator \eqref{nth order oper} has
regular singularity, that is
$$
v_i(z) \sim \frac{c_i(\nu)}{z^{i+1}} + \ldots,
$$
where the coefficients $c_i(\nu)$ are determined by the highest weight
$\nu$ of the $\wh\sw_r$-module $L_{\nu,k}$. Namely, representing $\nu$
as $(\nu_1,\ldots,\nu_r)$, where $\nu_i \in \C$ and $\sum_{i=1}^r
\nu_i = 0$, we find the $c_i(\nu)$'s from the following formula:
\begin{equation}    \label{zero}
\pa_z^r + \sum_{i=1}^{r-1} (-1)^i \frac{c_i(\nu)}{z^{i+1}}
\pa_z^{r-i-1} = \left(\pa_z - \frac{\nu_1}{z} \right) \ldots
\left(\pa_z - \frac{\nu_r}{z} \right).
\end{equation}

\smallskip

\noindent (3) At the points $w_j$, the operator \eqref{nth order oper}
has regular singularity,
\begin{equation*}
v_i(z) \sim \frac{c_i(\theta)}{(z-w_j)^{i+1}} + \ldots,
\end{equation*}
where $\theta = (1,0\ldots,0,-1)$ is the maximal root of $\sw_r$,
which is the highest weight of its adjoint representation. In
addition, we require that the operator \eqref{nth order oper} has {\em
  trivial monodromy} around the point $w_j$ for each $j=1,\ldots,m$,
and all $\la$.

\medskip

The proposal of \cite{FF:sol} is that the $r$th order differential
operators of this kind should correspond to the common eigenvalues of
the quantum $\wh\sw_r$-KdV Hamiltonians on the subspace of
$L_0$-degree $m$ in the generic irreducible module over the ${\mc
  W}$-algebra obtained by the quantum Drinfeld--Sokolov reduction
(with respect to $\n_-\ppart$) of the irreducible $\wh\sw_r$-module
$L_{\nu,k}$ of generic highest weight $\nu \in \h^*$ and level
$k$. The central charge of this module is
$$
c_k = r-1 - r(r^2-1)\frac{(k+r-1)^2}{k+r}.
$$
Note that for generic $\nu$ and $k$ this irreducible module over the
${\mc W}$-algebra is isomorphic to the Verma module and the Fock
representation $\pi_\mu$ with appropriate $\mu \in \h^*$.

In order to obtain a stand-alone spectral parameter, we apply the
change of variables $z \mapsto x$, where
$$
z = \frac{x^{r\al+r}}{(r\al+r)^r}
$$
and
$$
\al = -\frac{k+r-1}{k+r}.
$$
The dependence of the central charge on $\al$ is
$$
c(\al) = c_k = r-1 - r(r^2-1)\frac{\al^2}{\al+1}.
$$

It is easy to see that under this change of variables the irregular
term $z^{-r+1}$ in the last summand $v_{r-1}$ of our differential
operator becomes the term $x^{r\al}$, and the term $\la x^k$ gives
rise to the new spectral parameter term $-E$, where
$$
E = - (r\al+r)^{\frac{r\al}{\al+1}} \la,
$$
which is independent of $x$ (the sign is just a matter of
normalization). The poles of the new operator will be at the points
$x=0,\infty$ and $x^{r\al+r} = z_j, j=1,\ldots,m$, where
$$
z_j = (r\al+r)^r w_j.
$$

In particular, the $\wh{\sw}_r$-oper corresponding to the highest
weight vector may be written in the following way:
\begin{equation}    \label{special hw}
\pa_x + \ol{p}_{-1} + \frac{\nu}{x} + (x^{r\alpha} - E) e_\theta,
\end{equation}
where $\ol{p}_{-1} = \ds \sum_{i=1}^{r-1} f_i$, the sum of the
generators of the lower nilpotent subalgebra of $\sw_r$ ($f_i$ is the
matrix having 1 in the $i$th place below the diagonal, and 0
everywhere else), and $e_\theta$ is a generator of the maximal root
subspace (the matrix having 1 in the upper right corner and 0
everywhere else).

If we re-write these opers as $r$th order differential operators, we
obtain the differential operators of \cite{BHK,DDT,Dorey1}. They
correspond to the highest weight vectors of the representations of
${\mc W}$-algebras.

\subsection{General simply-laced case}    \label{slcase}

Let $\g$ be a simply-laced simple Lie algebra (that is, of $ADE$
type). Then $^L\g = \g$ and $^L\ghat = \ghat$. If $\g$ is of classical
type ($A$ or $D$), then we can realize the corresponding affine opers
as scalar (pseudo)differential operators, following \cite{DS}. It is
therefore possible to describe those of them that encode the spectra
of the quantum Hamiltonians of the $\ghat$-KdV system in a way that is
similar to the case of $\wh\sw_r$ (see the Section \ref{slr}). And in
fact, in the special case that there are no singular points other than
$0$ and $\infty$ (this is the case of the highest weight vector),
after a change of variables one obtains the (pseudo)differential
operators proposed in \cite{BHK,DDT,Dorey1} (see \cite{Dorey2} for a
survey). We can generalize these operators by including extra singular
points $w_j$ on $\pone$, as in the case of $\wh\sw_r$ (see
above).

However, in order to describe these affine opers for an arbitrary
affine Kac--Moody algebra $\ghat$ (including the case of non-simply
laced $\ghat$ dicussed below), we need to define $\ghat$-affine opers
as gauge equivalence classes of first order $\ghat$-valued
differential operators (as in Section 4.1 of \cite{FF:sol}).

Recall a basis $\{ \ol{p}_1,\ldots,\ol{p}_n \}$ of the canonical slice
$V_{\can} \subset \n_+ \subset \g$, as introduced in Section
\ref{untw} and Lemma \ref{free}. We will also use the elements
$\ol{p}_{-1} = \sum_{i=1}^n f_i$ and $f_0 = e_{\theta} \la$ of
$\ghat$. These are the same elements as those introduced in
\ref{twisted aff} except that now we use the variable $\la$ instead of
$z$ ($z$ is reserved for the spectral parameter of the KdV system and
appears here as the coordinate on $\pone$ on which the affine opers
are defined), and we take the completion in Laurent power series in
$\la$ rather than in $z^{-1}$ (the reason for this is explained in
Section 4.1 of \cite{FF:sol}).

The $\ghat$-opers (or equivalently, $^L \ghat$-opers) that parametrize
the common eigenvalues of the $\ghat$-KdV Hamiltonians are $\C {\mb d}
\ltimes \ghat$-valued differential operators of the form
\begin{equation} \label{oper spectra}
\pa_z + \ol{p}_{-1} + (z^{-h^\vee+1} + \la) e_{\theta} +
\sum_{i=1}^n v_i(z) \cdot \ol{p}_i + \frac{k}{z} {\mb d},
\end{equation}
where each $v_i(z)$ is a rational function on $\pone$ with poles at
$z=0, \infty$, and finitely many other points $w_j, j=1,\ldots,m$.

Here we denote by $h^\vee$ the {\em dual Coxeter number} of $\ghat$,
which coincides with the Coxeter number of $\ghat$ if $\g$ is
simply-laced.

As before, applying the gauge transformation by $z^{k{\mb d}}$ to the
operator \eqref{oper spectra}, we eliminate the last term at the cost
of multiplying $\la$ by $z^k$:
\begin{equation}    \label{oper spectra1}
\pa_z + \ol{p}_{-1} + (z^{-h^\vee+1} + z^k \la) e_{\theta} +
\sum_{i=1}^n v_i(z) \cdot \ol{p}_i.
\end{equation}

Recall that the quantum $\ghat$-KdV Hamiltonians act on the
irreducible module over the ${\mc W}$-algebra associated to $\g$
obtained by the quantum Drinfeld--Sokolov reduction (with respect to
$\n_-\ppart$) \cite{FF:lmp,FF:ds} from the irreducible module
$L_{\nu,k}$ over $\ghat$ with highest weight $\nu$ and level $k$ (such
a module is isomorphic to the Fock representation $\pi_\mu$ with
appropriate $\mu \in \h^*$). The proposal of \cite{FF:sol} is that for
generic $\nu$ and $k$, the common eigenvalues of these Hamiltonians
are encoded by the $\ghat$-opers (equivalently, $^L \ghat$-opers)
\eqref{oper spectra1} satisfying the following properties:

\medskip

\noindent (1) At $z=\infty$, the operator \eqref{oper spectra1}
has the mildest possible irregular singularity. The terms $v_i(z)$ are
regular:
\begin{align*}
\wt{v}_i(s) &\sim \frac{\wt{c}_i}{s^{d_i+1}} + \ldots, \qquad
i=1,\ldots,\ell,
\end{align*}
where $\wt{v}_i(s)$ are the coefficients of the operator obtained from
\eqref{oper spectra1} by the change of variables $z \mapsto s =
z^{-1}$. But the term $z^{-h^\vee+1} e_{\theta}$ creates an
irregular singularity term $(-1)^{h^\vee} s^{-h^\vee-1}
e_{\theta}$.

\smallskip

\noindent (2) At $z=0$, the operator \eqref{oper spectra1} has
regular singularity, that is
$$
v_i(z) \sim \frac{c_i(\nu)}{z^{d_i+1}} + \ldots,
$$
where $\nu \in \h^* \simeq \h$ and the $c_i(\nu)$ are determined by
the following rule: the element $$\ol{p}_{-1} + \sum_{i=1}^n \left(
  c_i(\nu) + \frac{1}{4} \delta_{i,1} \right) \ol{p}_i$$ is the unique
element in the Kostant slice of regular elements of $\g$,
$$
\ol{p}_{-1} + {\mb v}, \qquad {\mb v} \in {}\bb,
$$
which is conjugate to $\ol{p}_{-1} - \nu$ (see \cite{F:book},
Section 9.1).

Here we use the identification of $\h^*$ and $\h$ (so $\nu \in \h^*$
becomes an element of $\h \subset \bb$) provided by the restriction of
the invariant inner product on $\g$ normalized so that the square
length of each root is equal to $2$.

\smallskip

\noindent (3) At the points $w_j$, the operator \eqref{oper spectra1}
has regular singularity, and for all $i=1,\ldots,\ell$,
\begin{equation*}
  v_i(z) \sim \frac{c_i(\theta)}{(z-w_j)^{d_i+1}} + \ldots,
\end{equation*}
where $\theta$ is the maximal root of $\g$. In addition, we require
{\em trivial monodromy} around the point $w_j$ for each
$j=1,\ldots,m$, and all $\la$.

\medskip

As before, the number $m$ should correspond to the $L_0$-degree of the
corresponding eigenvector.

\subsection{Change of variables: general case}    \label{change
  var}

Next, we make a change of variables so as to make the spectral
parameter appear independently of the coordinate:
\begin{equation}    \label{change}
z = \frac{x^{h^\vee(\al+1)}}{(h^\vee(\al+1))^{h^\vee}}
\end{equation}
where
$$
\al = -\frac{k+h^\vee-1}{k+h^\vee}.
$$
When we make this change, we find that $\ol{p}_{-1}$, $\ol{p}_i$, and
$e_{\theta}$ all get multiplied by
$$
\frac{dz}{dx} = \frac{x^{h^\vee(\al+1)-1}}{(h^\vee(\al+1))^{h^\vee-1}}.
$$
In order to bring the operator to the oper form $\pa_x + \ol{p}_{-1} +
\ldots$, we then need to apply a gauge transformation by an element of
the Cartan subgroup $H$.

Recall the weight $\rho \in \h^*$, the half-sum of positive roots. We
have
$$
\langle \rho,\al_i^\vee \rangle = 1, \qquad i=1,\ldots,n.
$$
Here we view $\rho$ as a coweight of $\h$ identified with $\h^*$ as
above, via a normalized invariant inner product. Therefore, $\rho$
gives rise to a homomorphism (one-parameter subgroup) $\C^\times \to
{}H$, the Cartan subgroup of the simply-connected Lie group $G$ with
the Lie algebra $\g$. We will denote this one-parameter subgroup by
$\rho$. It acts as follows:
$$
\rho(a) \cdot \ol{p}_{-1} = a^{-1} \ol{p}_{-1}, \qquad \rho(a) \cdot
\ol{p}_i = a^{d_i} \ol{p}_i, \quad i=1,\ldots,n,
$$
and
$$
\rho(a) \cdot e_{\theta} = a^{h^\vee-1} e_{\theta}.
$$

Let us apply the gauge transformation by
$$
\frac{\rho(x)^{h^\vee(\al+1)-1}}{\rho(h^\vee(\al+1))^{h^\vee-1}}.
$$
As the result we get back $\ol{p}_{-1}$, and the coefficient
$z^k\lambda e_{\theta}$ becomes $-E e_{\theta}$, where
$$
E = - (h^\vee(\al+1))^{h^\vee\frac{\al}{\al+1}} \la,
$$
which is a pure spectral parameter, independent of $x$ (due to the
identity $kh^\vee(\al+1)+h^\vee(h^\vee(\al+1)-1) = 0$). In addition,
the term $z^{-h^\vee+1} e_{\theta}$ becomes $x^{h^\vee \al}$ (that
was the reason for including the coefficient
$(h^\vee(\al+1))^{-h^\vee}$ in the coordinate change from $z$ to $x$).

The singularities at $z=w_j$ become singularities at $x^{h^\vee(\al+1)}
= z_j$, where
$$
z_j = (h^\vee(\al+1))^{h^\vee} w_j.
$$

The oper \eqref{oper spectra1} now takes the form
\begin{equation}    \label{oper spectra gen}
\pa_z + \ol{p}_{-1} + (x^{h^\vee \al} - E) e_{\theta} +
\sum_{i=1}^n \ol{v}_i(x) \cdot \ol{p}_i,
\end{equation}
where
\begin{align*}
\ol{v}_1(x) &= v_1(\varphi(x))(\varphi')^2 - \frac{1}{2} \{ \varphi,x
\}, \\
\ol{v}_i(x) &= v_i(\varphi(x))(\varphi')^{d_i+1}, \qquad i>1,
\end{align*}
and $\varphi(x)$ is the function on the right hand side of formula
\eqref{change} (see \cite{F:book}, Section 4.2.4).

In the case $m=0$, the $^L\ghat$-oper corresponding to the highest
weight vector may be written in the form
\begin{equation}    \label{special hw gen}
\pa_x + \ol{p}_{-1} + \frac{\nu}{x} +  (x^{h^\vee \al} - E) e_{\theta},
\end{equation}
where $\nu \in \h^* = {}^L\h$.

\subsection{The non-simply laced case}    \label{nslcase}

Let us recall that to each affine Dynkin diagram we can associate an
affine Kac--Moody algebra. We are interested in its quotient by the
central element. In the untwisted case, this is the semi-direct
product of ${\mb d} = \la \pa_\la$ and the universal central extension
of $\g\pparl$, where $\g$ is a simple Lie algebra. In the twisted
case, it is the semi-direct product of ${\mb d} = \la \pa_\la$ and the
universal central extension of the twisted loop algebra $L_\sigma\gt$,
defined as in Section \ref{twisted aff}, except that the loop variable
is now denoted by $\lambda$ rather than $z$, and we take the
completion in formal power series in $\lambda$ rather than in
$z^{-1}$.

Given an affine Kac--Moody algebra $\ghat$ associated to a Dynkin
diagram $\wh\Gamma$, we define its Langlands dual as the affine
Kac--Moody algebra $^L\ghat$ associated to the dual Dynkin diagram $^L
\wh\Gamma$ which is obtained from $\wh\Gamma$ by reversing all
arrows. The dual of $^L\ghat$ is of course $\ghat$ itself.

Note that if we remove the 0th nodes from the Dynkin diagrams
$\wh\Gamma$ and $^L \wh\Gamma$, we obtain the Dynkin diagrams of two
simple Lie algebras that are Langlands dual to each other. These are
the degree zero Lie subalgebras of $\ghat$ and $^L\ghat$ with respect
to the grading defined by $d$. We denote them by $\g$ and $^L\g$.

If $\ghat$ is the untwisted Kac--Moody algebra associated to a
simply-laced simple Lie algebra $\g$, then $^L\ghat = \ghat$. This is
the case we have just discussed. However, if $\ghat$ is the untwisted
Kac--Moody algebra associated to a non-simply laced simple Lie algebra
$\g$, then $^L\ghat$ is a twisted affine Kac--Moody algebra. In this
case the Dynkin diagram of $^L\g$ is the quotient of the Dynkin
diagram of a simply-laced Lie algebra $\gt$ by an automorphism of
order $r^\vee=2$ or $3$, and $^L\ghat$ is the corresponding twisted
affine algebra of type $\gt^{(r^\vee)}$ (see Section \ref{twisted
  aff}). Note that in this case $^L\g = \gt_0$, the $\sigma$-invariant
part of $\gt$, where $\sigma$ is an outer automorphism of $\gt$ of
order $r^\vee$ corresponding to the automorphism of the Dynkin diagram
of $\gt$.

Here we consider the case that $\g$ is non-simply laced, and so
$^L\ghat$ is a twisted affine Kac--Moody algebra. The conjecture of
\cite{FF:sol} is that the joint eigenvalues of the quantum $\ghat$-KdV
Hamiltonians are encoded by the $^L \ghat$-affine opers, and so the
relevant finite-dimensional Lie algebra is $^L\g$. Hence we consider
the basis $\{ \ol{p}_1,\ldots,\ol{p}_n \}$ of the canonical slice
$V_{\can}$ of $^L\g$ rather than $\g$ (see Section \ref{untw} and
Lemma \ref{free}). We will also use the elements $\ol{p}_{-1} =
\sum_{i=1}^n f_i$ and $f_0$ of $^L\ghat$. The latter is now equal to
$f_0 = e_{\theta_0} \la$, where $e_{\theta_0}$ is a highest weight
vector of the $^L\g$-module $\gt_1$, the eigenspace of $\sigma$ with
the eigenvalue $e^{2\pi i/r^\vee}$ (see Section \ref{twisted aff}).

It is natural to generalize formula \eqref{oper spectra} to this case
as follows:
\begin{equation} \label{oper spectra-ns}
\pa_z + \ol{p}_{-1} + (z^{-h^\vee+1} + \la) e_{\theta_0} +
{\mb v}(z) + \frac{k}{z} {\mb d}.
\end{equation}

However, as D. Masoero and A. Raimondo pointed out to us, this formula
requires a further justification since the operator \eqref{oper
  spectra-ns} does not, on the face of it, take values in $^L\ghat$
if $\ghat$ is non-simply laced. Indeed, according to the definition of
the twisted affine Kac--Moody algebras (see Section \ref{twisted
  aff}), the $\la^0$-part of an element of $^L\ghat$ should be in
$^L\g = \gt_0$, and the $\la^1$-part should be in $\gt_1$. But in
formula \eqref{oper spectra-ns}, we have the $\la^0$-term
$z^{-h^\vee+1} e_{\theta_0}$, which is in $\gt_1$.

In order to make sense of formula \eqref{oper spectra-ns} in the
non-simply laced case, we recall the notion of twisted opers
introduced by B. Gross and one of the authors in \cite{FG}, Section
5.1.

The idea is to make the automorphism $\sigma$ act on both the Lie
algebra $\gt$ and the space on which our differential operators are
defined; that is, the projective line with the coordinate $z$. We will
view this $\pone = \pone_z$ as an $r^\vee$-sheeted cover of another
projective line $\pone_t$ with the coordinate $t$ such that
$z^{r^\vee} = t$. We define an automorphism $\wt\sigma$ of $\pone_z
\times \gt$ by the formula
$$
(z,g) \mapsto (ze^{-2\pi i/r^\vee},\sigma(g)).
$$
We then have a natural notion of a $\sigma$-twisted connection on
$\pone_t$: namely, a $\wt\sigma$-invariant (meromorphic) connection on
the trivial $G'$-bundle on $\pone_z$ (it is automatically flat, since
$\pone_z$ is one-dimensional as a complex manifold, so the flatness
condition is vacuous),
\begin{equation}    \label{tw conn}
\nabla = d + A(z) dz,
\end{equation}
where $A(z) dz$ is a $\wt\sigma$-invariant $\gt$-valued one-form on
$\pone_z$ (and $d$ is the de Rham differential on $\pone_z$).

Furthermore, in \cite{FG} the notion of a twisted oper was introduced,
as a $\wt\sigma$-invariant connection \eqref{tw conn} satisfying a
natural generalization of the oper condition. Here's an example of a
twisted oper constructed in \cite{FG} (we use the notation of the
present paper):
\begin{equation}    \label{tw oper}
d + (\ol{p}_{-1} + z e_{\theta_0}) \frac{dz}{z}.
\end{equation}
It is clear that the restriction of this operator to the formal disc
around the point $z=0$ can be viewed as an element of the dual space
to the twisted affine Kac--Moody algebra $^L\ghat$ (with respect to
the coordinate $z$).

The difference between the twisted opers introduced in \cite{FG}, such
as \eqref{tw oper}, and the $^L\ghat$-affine opers that we need here
is that we have to incorporate the additional (formal) variable
$\la$. We do that as follows.

Given $\gt$ and $^L\ghat$ as above, we define a $^L\ghat$-affine
oper on $\pone_t$ (where, as above, $t=z^{1/r^\vee}$) as an operator
\begin{equation}    \label{double tw conn}
\nabla = d + A(z,\la) dz, \qquad A(z,\la) \in \left( \gt \otimes
  \C(z)\pparl \right) \oplus \left( {\mb d} \otimes \C(z) \right)
\end{equation}
(note that we denote the element of the affine algebra $\la \pa_\la$
by ${\mb d}$ to avoid any confusion with the de Rham differential
$d$), which are invariant under the action of the automorphism
$\wt{\wt\sigma}$:
$$
(z,\la,g) \mapsto (ze^{-2\pi i/r^\vee},\la e^{-2\pi
  i/r^\vee},\sigma(g)), \qquad g \in \gt.
$$
Thus, in our definition of $^L\ghat$-affine opers, the automorphism is
given not only as $\sigma$ on $\gt$ and the map $\la \mapsto \la
e^{-2\pi i/r^\vee}$, but also the map $z \mapsto z e^{-2\pi i/r^\vee}$
acting on the coordinate $z$ of $\pone_z$, on which the oper is
defined.

Now we interpret our operator \eqref{oper spectra-ns} as a
$^L\ghat$-affine oper in the sense of this definition. Let us recall
the element $\rho \in \h^*$ defined in Section \ref{slcase}. Since we
have a canonical isomorphism $\h^* \simeq {}^L\h$, we view $\rho$ as
an element of $^L\h$. Therefore, $\rho$ gives rise to a homomorphism
$\C^\times \to {}^LH$, the Cartan subgroup of the simply-connected Lie
group $^LG$ with the Lie algebra $^L\g$, and we have
$$
\rho(a) \cdot \ol{p}_{-1} = a^{-1} \ol{p}_{-1}, \qquad
\rho(a) \cdot e_{\theta_0} = a^{h^\vee-1} e_{\theta_0}.
$$
Applying the gauge transformation by $\rho(z)$ to the operator
\eqref{oper spectra-ns} (and multiplying it by $dz$ in order to make
it into a connection operator of the form \eqref{double tw conn}), we
obtain
\begin{equation}    \label{oper spectra-ns1}
d + \left( \ol{p}_{-1} + (z + \la z^{h^\vee}) e_{\theta_0} +
\ol{\mb v}(z) \right) \frac{dz}{z} + k {\mb d} \; \frac{dz}{z},
\end{equation}
where $\ol{\mb v}(z) = z \rho(z) \cdot {\mb v}(z) -
\rho$. Finally, we apply the gauge transformation by
$z^{h^\vee {\mb d}}$ to bring it to the following form:
\begin{equation}    \label{oper spectra-ns2}
d + \left( \ol{p}_{-1} + (z + \la) e_{\theta_0} +
\ol{\mb v}(z) \right) \frac{dz}{z} + (k+h^\vee) {\mb d} \;
\frac{dz}{z}
\end{equation}
Note that the resulting shift of $k$ by the dual Coxeter number
$h^\vee$ is akin to the shift of ${\mb v}(z)$ by $-\rho$; in a sense,
it properly centers the weights around the affine version of $\rho$
(note that the shift of ${\mb v}(z)$ is here by $-\rho$, rather than
$\rho$, because of our convention that the generator $f_0$ is
$e_{\theta_0} \la$, rather than $e_{\theta_0} \la^{-1}$ which leads to
the multiplication of the weights by $-1$, so that the ``central''
weight is $\rho$ rather than the more traditional $-\rho$).

The connection \eqref{oper spectra-ns2} is $\wt{\wt\sigma}$-invariant,
and is indeed a $^L\ghat$-oper, provided that
\begin{equation}    \label{v inv}
\ol{\mb v}(z) = {\mb v}(ze^{-2\pi i/r^\vee})
\end{equation}
(in other words, $\ol{\mb v}(z)$ really depends on $t=z^{r^\vee}$).

Note that formula \eqref{oper spectra-ns2} makes perfect sense for
{\em any} affine Kac--Moody algebra $^L\ghat$ and can be taken as an
alternative definition of the $^L\ghat$-opers in the simply-laced case
(in this case, of course, $\sigma$ is the identity, $r^\vee=1$, $t=z$,
and we set $\theta_0=\theta$).

As before, we can eliminate the last term in \eqref{oper spectra-ns2}
by applying the gauge transformation by $z^{(k+h^\vee) {\mb d}}$. Then
we obtain
\begin{equation}    \label{oper spectra-ns3}
d + \left( \ol{p}_{-1} + (z + \la z^{k+h^\vee}) e_{\theta_0} +
\ol{\mb v}(z) \right) \frac{dz}{z}.
\end{equation}

By using an $^LN$-valued gauge transformation, we can bring $\ol{\mb
  v}(z)$ to the form
$$
\ol{\mb v}(z) = \sum_{i=1}^\ell \ol{v}_i(z) \ol{p}_i.
$$

Now we list the conditions that the $^L\ghat$-opers \eqref{oper
  spectra-ns3} should satisfy in order to encode joint eigenvalues of
the quantum KdV Hamiltonians. Note that these conditions are slightly
different from (and simpler than) the conditions listed in Section
\ref{slcase}. This is in part because in formula \eqref{oper
  spectra-ns3} we have the overall factor $1/z$ that creates a more
convenient ``gauge'' for our opers.

\medskip

\noindent (0) $\ol{\mb v}(z)$ should satisfy \eqref{v inv}, and hence
each $\ol{v}_i(z)$ should satisfy the same invariance property.

\smallskip

\noindent (1) At $z=\infty$ (if we set $\la=0$), the operator
\eqref{oper spectra1} should have the mildest possible irregular
singularity. The terms $\ol{v}_i(z)$ are regular:
\begin{equation*}
\wt{\ol{v}}_i(s) \sim - \ol{c}_i(\nu+\rho) + \ldots, \qquad
i=1,\ldots,\ell,
\end{equation*}
but the term $e_{\theta_0} dz$ creates an irregular singularity term
$-s^{-2} e_{\theta_0} ds$ (here, as above, $s=z^{-1}$).

\smallskip

\noindent (2) At $z=0$, the operator \eqref{oper spectra1} should have
regular singularity, that is
$$
\ol{v}_i(z) \sim \ol{c}_i(\nu+\rho) + \ldots,
$$
for some $\nu \in \h^* = {}^L\h$, where $\ol{c}_i(\mu)$ is
determined by the following rule:

\noindent The element
$$
\ol{p}_{-1} + \sum_{i=1}^\ell \ol{c}_i(\mu) \ol{p}_i
$$
is the unique element in the Kostant slice of regular elements of
$\g$, which is conjugate to $\ol{p}_{-1} - \mu$ (see \cite{F:book},
Section 9.1).

\smallskip

\noindent (3) The $\ol{v}_i(z)$ are allowed to have regular
singularities at $m \cdot r^\vee$ points on $\C^\times \subset
\pone_z$,
$$
w^{(p)}_j = w_j e^{2\pi i
  p/r^\vee}, \qquad j=1,\ldots,m, \quad p=0,1,r^\vee-1,
$$
so that we have the following expansions in $z-w^{(p)}_j$:
\begin{equation*}
  \ol{v}_i(z) \sim \ol{c}_i(\theta_0) +
  \ldots
\end{equation*}
In addition, the connections \eqref{oper spectra-ns2} and \eqref{oper
  spectra-ns3} should have {\em trivial monodromy} around each of the
points $w^{(p)}_j$ for all $\la$.

\medskip

Thus, we see that the set of singular points (other than $0$ and
$\infty$) forms a union of $r^\vee$ families -- orbits of the cyclic
group $\Z_{r^\vee}$ naturally acting on $\pone_z$; that's because we
need $\ol{\mb v}(z)$ to be invariant under the action of this
group. As before, the number $m$ should correspond to the $L_0$-degree
of the corresponding eigenvector of the $\ghat$-KdV Hamiltonians.

\bigskip

We can make a change of variables \eqref{change} so as to make the
spectral parameter appear independently of the coordinate. All
calculations of Section \ref{change var} apply in the same way as in
the simply-laced case.

\subsection{$Q\wt{Q}$-system from affine opers}    \label{QQ tw aff}

Let $\ghat$ be an untwisted quantum affine algebra (simply-laced or
non-simply laced). In the papers
\cite{MRV1,MRV2} a solution of the $Q\wt{Q}$-system \eqref{Qsyst} is
assigned to the affine $^L\ghat$-oper \eqref{special hw gen} for
generic $\nu \in \h^*$ (note that our $\al$ corresponds to $M$ in
\cite{MRV1,MRV2}, and our $q$ corresponds to $\Omega^{-1/2}$ in
\cite{MRV1,MRV2}). This construction generalizes earlier results
\cite{DT,BLZ4,DT1,DDT,BHK,S}.

More precisely, for an untwisted affine Kac--Moody algebra $\ghat$,
the authors of \cite{MRV1,MRV2} define $n$ evaluation representations
$V^{(i)}, i=1,\ldots,n$, of the affine algebra $^L\ghat$ (which is
twisted if $\g$ is not simply-laced) and study the first order linear
differential equations obtained from the operator \eqref{special hw
  gen} specialized in the representations $V^{(i)}$. They show that
each of these equations has a unique (properly normalized) solution
that goes to 0 most rapidly as $x \to +\infty$. Then they define ${\mb
  Q}_i$ and $\wt{\mb Q}_i$ as the leading coefficients appearing in
the expansion of the above solutions near $x=0$ (these functions may
be viewed as generalizations of the spectral determinants). Finally,
they show that for generic $\nu \in \h^*$ these functions are entire
functions of $E$ which satisfy the $Q\wt{Q}$-system \eqref{Qsyst} (with
$E$ instead of $u$ and particular values of $v_i, i=1,\ldots,n$). They
obtain this system from a system of equations satisfied by the above
solutions (which they call the $\Psi$-system).

Now let us consider the case that $\ghat$ is a twisted affine
algebra. In this case $^L\ghat$ is untwisted (unless $\ghat =
A_{2n}^{(2)}$, in which case $^L\ghat = A_{2n}^{(2)}$ as well). Note
that this case was not considered in \cite{MRV1,MRV2}. However, it is
natural to expect that using the construction of \cite{MRV1,MRV2}, one
can attach to the $^L\ghat$-affine oper \eqref{special hw gen} a
solution of the $Q\wt{Q}$-system \eqref{QQ tw} associated to $\ghat$
from Section \ref{QQ twisted}. This leads us to the following
conjecture.

\begin{conj} \label{tw aff op} Let $\ghat$ be a twisted affine
  algebra. Then to any $^L\ghat$-affine oper \eqref{special hw gen}
  with generic $\nu$ one can attach a solution of the $U_q(\ghat)$
  $Q\wt{Q}$-system \eqref{QQ tw}, with $\left
    [\pm\frac{\alpha_i}{2}\right]$ mapping to some $v_i^{\pm 1} \in
  \C^\times$ depending on $\al$ and $\nu$.
\end{conj}

In this section, we have discussed the conjecture of \cite{FF:sol}
linking the eigenvectors (or, equivalently, the spectra) of the
$\ghat$-KdV Hamiltonians to $^L\ghat$-affine opers on $\pone$ with
special analytic behavior. Note that this conjecture was not based
on a direct construction, but rather on an analogy with the Gaudin
model (the fact that the spectra of the $\g$-Gaudin model can be
encoded by $^L\g$-opers on $\pone$ with special analytic behavior, see
\cite{FF:sol}, especially Sections 4.4, 5.4, and 5.5, for more
details).

However, if the results of \cite{MRV1,MRV2} (as well as Conjecture
\ref{tw aff op}) could be generalized to other $^L\ghat$-opers
discussed in this section, given by formulas \eqref{oper spectra1} and
\eqref{oper spectra-ns3}, that we expect to correspond to the excited
states in the highest weight representations of ${\mc W}$-algebra,
then, as we discussed in the Introduction, this would indeed open the
possibility of establishing a direct link between the spectra of
quantum $\ghat$-KdV Hamiltonians and $^L\ghat$-opers.

In \cite[Section 4]{FF:sol}, another quantum integrable system was
discussed: the ``shift of argument'' $\ghat$-Gaudin model, and its
spectrum was also conjectured to be encoded by $^L\ghat$-affine opers
on $\pone$. These opers differ from the $^L\ghat$-affine opers arising
in the quantum $\ghat$-KdV system in the way they behave near $\infty
\in \pone$. Nevertheless, we expect that solutions of the $U_q(\ghat)$
$Q\wt{Q}$-system can be attached to these $^L\ghat$-affine opers as
well. At the same time, the joint eigenvalues of the quantum
Hamiltonians of this $\ghat$-Gaudin model should satisfy the same
$Q\wt{Q}$-system for the same reason as in the quantum KdV
case. Therefore, we again expect that $Q\wt{Q}$-system would provide a
link between between these joint eigenvalues and the $^L\ghat$-affine
opers described in \cite[Section 4]{FF:sol}.

\subsection{Duality of affine opers}    \label{duality opers}

Let us recall that Conjecture \ref{commute} implies that there is a
correspondence between solutions of the $Q\wt{Q}$-systems (as well as
other equations stemming from $K_0({\mc O})$ such as the $QQ^*$-system
of \cite{HL}, see Section \ref{QQstar}) for $U_q(\ghat)$ and
$U_{\check{q}} {}^L\ghat$, where
$$
q = e^{\pi i \beta^2}, \qquad \check{q} = e^{\pi i \check{r}/\beta^2}.
$$
If that is true, then this should also hold on the side of affine
opers. This means that there should be a correspondence (or duality)
between the $^L\ghat$-opers of the form discussed in this section with
the parameter $\alpha$, and $\ghat$-opers of the same form but with
the parameter $\check\alpha$, where
$$
\check\al+1 = \frac{1}{\check{r}(\al+1)}
$$
(see formulas \eqref{betacheck} and \eqref{albeta}). For
$\ghat=\wh\sw_r$, this duality was discussed in \cite{DDT}.

\subsection{Two appearances of opers}    \label{last}

At first glance, it may appear that the conjecture of \cite{FF:sol}
discussed in this section is not so surprising: after all, the phase
space of the classical KdV system is the space of opers. Why should we
then be surprised that the spectra of the quantum KdV Hamiltonians
would be linked to opers as well? However, it is important to realize
that the two spaces of opers appearing here are quite different.

The phase space of the classical $\ghat$-KdV system is the space of
$\g$-opers (not $\ghat$-opers!) on a circle, or a punctured disc
(with coordinate $t$), see Section \ref{ckdv}. We inserted a
spectral parameter $z$ into these $\g$-opers in order to construct the
Poisson commuting KdV Hamiltonians.

On the other hand, the spectra of the corresponding algebra of quantum
$\ghat$-KdV Hamiltonians are conjecturally encoded by $^L\ghat$-affine
opers on the projective line $\pone$ with coordinate $z$ (it is the
same $z$ as the spectral parameter in the classical story). So opers
appear again, but these are {\em affine} opers, and they are
associated to the {\em Langlands dual} affine algebra
$^L\ghat$. Therefore, {\em a priori} they have nothing to do with the
$\g$-opers appearing in the definition of the classical KdV system
(other than the fact that a coordinate on the space on which the
$^L\ghat$-affine opers ``live'' is the spectral parameter of the
$\g$-opers).

For instance, in the case of $\sw_2$, the points of the phase space of
the KdV system are $\sw_2$-opers with spectral parameter (see Section
\ref{mon})
$$
\pa_t^2 - v(t) - z,
$$
where $t$ is a coordinate on a circle, or a punctured disc, and $z$ is
the spectral parameter. The classical KdV Hamiltonians are constructed
by expanding the monodromy matrix of this operator, considered as
function of $z$, near $z=0$ (non-local) or $z=\infty$ (local).

On the other hand, the $\su$-opers that encode the eigenvalues of the
quantum KdV Hamiltonians have the form (see Section \ref{sl2})
$$
\pa_z^2 - v(z) - \lambda z^k,
$$
where $z$ is the spectral parameter of the classical KdV system, which
is now viewed as a coordinate on $\pone$, and $v(z)$ is a meromorphic
function on this $\pone$ with poles at $z=0, \infty$, and finitely
many other points (see formula \eqref{final form kdv}). There is
another spectral parameter $\lambda$. When we make a change of
variables $z \mapsto x$, we obtain the differential operators of the
form \eqref{new operator} with the spectral parameter $E$.

For non-simply laced $\g$, the difference between the two spaces is
even more drastic because of the appearance of the Langlands dual
affine algebra $^L \ghat$.

As argued in \cite{FF:sol}, the ``quantum KdV -- affine opers''
duality may be viewed as a generalization of the duality observed in
the generalized Gaudin quantum integrable systems, in which the
spectra of the quantum Hamiltonians in a model associated to a simple
Lie algebra $\g$ turn out to be encoded by $^L\g$-opers
\cite{FFR,F:icmp,FFT}. The latter is explained by the isomorphism
between the center of the completed enveloping algebra of $\ghat$ at
the critical level and the algebra of functions on $^L\g$-opers on the
formal disc \cite{FF:gd} (see \cite{F:book} for an exposition). In
other words, in order to understand the duality between spectra of the
generalized $\g$-Gaudin systems and $^L\g$-opers, we need to use the
affinization of the Lie algebra $\g$: The ``master algebra''
lurking behind the generalized Gaudin quantum integral systems is the
center of the completed enveloping algebra of $\ghat$ at the critical
level, and the fact that it is isomorphic to the algebra of functions
on $^L\g$-opers on the formal disc gives rise to identifications of
the spectra of the $\g$-Gaudin Hamiltonians with $^L\g$-opers of
particular kind.

Therefore it is natural to expect that in order to understand the
``quantum KdV -- affine opers'' duality we need to study the
affinization of $\ghat$; that is to say, a toroidal algebra of $\g$,
but the big open problem here is to figure out what is the analogue of
the ``critical level'' of $\ghat$ and the corresponding center (see
Section 7 of \cite{FF:sol} for a discussion of this point). It may
well be that to do so, one needs to study the ``gerbal representations''
of the toroidal algebra introduced in \cite{FZ}.

{}From this point of view, the ``quantum KdV -- affine opers'' duality
offers us glimpses into the mysterious ``critical level'' structures
arising in toroidal algebras.

\end{document}